\documentclass[12pt]{article} 
\usepackage{amsmath, amssymb}
\usepackage{amsthm} 
\usepackage{mathtools}
\usepackage{mathrsfs}
\usepackage{enumerate}
\usepackage{empheq}
\usepackage{esint}
\usepackage{bbm}
\usepackage{color}
\usepackage[T1]{fontenc}
\usepackage{textcomp}
\usepackage[abbrev]{amsrefs}
\usepackage[dvipdfmx]{hyperref}
\usepackage[UKenglish]{babel}
\usepackage[UKenglish]{isodate}
\numberwithin{equation}{section}
\theoremstyle{plain} 
\newtheorem{theorem}{Theorem}[section]
\newtheorem{lemma}[theorem]{Lemma}
\newtheorem{corollary}[theorem]{Corollary}
\newtheorem{proposition}[theorem]{Proposition}

\theoremstyle{definition} 

\newtheorem{remark}[theorem]{Remark}

\def\N{\mathbb N}

\def\R{\mathbb R}

\usepackage{bbm}
\usepackage{color}

\allowdisplaybreaks

\makeatletter
\def\address#1#2{\begingroup
\noindent\parbox[t]{7.8cm}{%
\small{\scshape\ignorespaces#1}\par\vskip1ex
\noindent\small{\itshape E-mail address}%
\/: #2\par\vskip4ex}\hfill%
\endgroup}%
\makeatother
\title{\uppercase{A sharp criterion and complete classification of global-in-time solutions and finite time blow-up of solutions to a chemotaxis system in supercritical dimensions}} 
\author{
\bigskip \\
\textsc{Yuri Soga} 
}
\date{\today} 
\begin{document}

\maketitle

\footnote{ 
2020 \textit{Mathematics Subject Classification}.
Primary 35K55; Secondary 35B44, 35B51, 35Q92}
\footnote{ 
\textit{Key words and phrases}.
chemotaxis, blowup, global existence, comparison principles
}

\vspace{-1cm}
\begin{abstract}
We consider the chemotaxis system with indirect signal production in the whole space,
\begin{equation}\label{abst:p}\tag{$\star$}
\begin{cases}
u_t = \Delta u - \nabla \cdot (u\nabla v),\\
0 = \Delta v + w,\\
w_t = \Delta w + u
\end{cases}
\end{equation}
with emphasis on supercritical dimensions. In contrast to the classical parabolic-elliptic Keller--Segel system, where the analysis can be reduced to a single equation, the above system is essentially parabolic-parabolic and does not admit such a reduction. 
In this paper, we establish a sharp threshold phenomenon separating
global-in-time existence from finite time blow-up in terms of scaling-critical Morrey norms of the initial data. In particular, we prove the existence of singular stationary solutions and show that their Morrey norm values serve as the critical thresholds determining the long-time behavior of solutions. Consequently, we identify new critical exponents at which the long-time behavior of solutions changes. This yields a complete classification of the long-time behavior of solutions, providing the first such results for the essentially parabolic-parabolic chemotaxis system \eqref{abst:p} in supercritical dimensions.
\end{abstract}
\newpage

\tableofcontents
 

\section{Introduction}

\subsection{Biological and Mathematical Background}

Chemotaxis refers to the directed movement of cells in response to gradients of chemical signals and plays a crucial role in a variety of biological phenomena.
Examples include aggregation and pattern formation in microbial populations, as well as coordinated cell migration in developmental and pathological processes.
These observations have inspired extensive mathematical studies of chemotaxis models, most notably those based on systems of partial differential equations. Among them, the Keller--Segel model (cf. \eqref{KS} below), originally introduced by Keller and Segel \cite{KS1970} to describe the migration of cells guided by chemical signals generated by the cell population itself, has long served as a fundamental model in the study of chemotaxis. Originally motivated by biological observations, the Keller--Segel model has also stimulated extensive mathematical research over the years, leading to a deeper understanding of its structural and dynamical properties (see \cite{BHeN1994, B1998, B1999, HV1996, HV1997, N1995, NSS2000, NSY1997, BBTW2015}).
In this paper, we consider the Cauchy problem for the chemotaxis system:
\begin{equation}\label{p}\tag{P}
\begin{cases}
u_t = \Delta u - \nabla \cdot (u\nabla v)\qquad &\mathrm{in}\ \R^d \times (0,\infty),\\
0 = \Delta v + w\qquad &\mathrm{in}\ \R^d \times (0,\infty),\\
w_t = \Delta w + u \qquad &\mathrm{in}\ \R^d \times (0,\infty),\\
(u(\cdot,0),w(\cdot, 0)) = (u_0(\cdot),w_0(\cdot))\qquad &\mathrm{in}\ \R^d.
\end{cases}
\end{equation}
The system \eqref{p} features an indirect signal production mechanism, where the chemoattractant potential $v$ is determined elliptically from an intermediate variable $w$, while $w$ itself evolves according to a parabolic equation driven by the cell density $u$. The system \eqref{p} can be regarded as a simplified version of the chemotaxis model:
\begin{equation}\label{p_0}
\begin{cases}
u_t = \Delta u - \nabla \cdot (u\nabla v) & \mathrm{in}\ \Omega \times (0,\infty),\\
\tau_1v_t =  \gamma_1\Delta v - \delta_1v + w & \mathrm{in}\ \Omega \times (0,\infty),\\
\tau_2w_t = \gamma_2\Delta w - \delta_2w + u & \mathrm{in}\ \Omega \times (0,\infty),
\end{cases}
\end{equation}
where $\Omega$ is either a bounded domain or the whole space and the parameters $\tau_1, \tau_2, \gamma_1, \gamma_2, \delta_1, \delta_2$ are nonnegative. The chemotaxis model \eqref{p_0} is a specific instance of a system proposed by \cite{Luca2003} in the context of biological processes related to Alzheimer's disease. In particular, the model \eqref{p_0} with $\gamma_2 = 0$ has also been used to model the cluster attack behavior of the Mountain Pine Beetle, as proposed in \cite{STP2013}.
From a mathematical point of view, the chemotaxis model \eqref{p_0} has been studied in various parameter regimes.
In the fully parabolic case, several analytical results are available in the literature (see \cite{FS2017, FS2019, HL2025, J2015}, for instance). On the other hand, when $\gamma_2 = 0$ the system reduces to a parabolic–parabolic–ODE (or parabolic–elliptic–ODE) type, for which different aspects of solutions behavior have been investigated in recent years (see \cite{Lau2019, S2025, TW2017}, for instance). 

\subsection{Mass critical phenomena in critical dimensions}

The system \eqref{p} can be viewed as a variant of the classical Keller--Segel system:
\begin{equation}\label{KS}
\begin{cases}
u_t = \Delta u - \nabla \cdot (u\nabla v),\\
\tau v_t = \Delta v + u
\end{cases}
\end{equation}
in $\R^2$, where the parameter $\tau$ is nonnegative. The Keller--Segel system enjoys the scaling invariance:
\begin{align}
u_\lambda (x,t) = \lambda^{2}u(\lambda x, \lambda^2 t),\quad v_\lambda (x,t) = v(\lambda x, \lambda^2 t)\quad\mathrm{for}\ \lambda > 0.\label{KS:scale}
\end{align}
In addition, from a biological modeling perspective, the Keller--Segel system is expected to satisfy the mass conservation law $\|u(t)\|_{L^1(\R^2)} = \|u_0\|_{L^1(\R^2)}$, which follows from the fact that the first equation is written in divergence form. Consequently, the $L^1$-norm of $u$ plays a distinguished role. The scaling leaves the $L^1$-norm invariant if and only if the spatial dimension satisfies $d=2$, which therefore constitutes the critical dimension.  In view of the above observations, the two-dimensional case occupies a distinguished position in the theory of the Keller--Segel system, both in bounded domains and in the whole space. Indeed, in two space dimensions, the long-time behavior of solutions of \eqref{KS} is known to depend sensitively on the total initial mass. This criticality gives rise to a threshold phenomenon separating global existence from blow-up, which is commonly referred to as the $8\pi$-problem. More precisely, when the initial cell density $u_0$ with radial symmetry is given by a nonnegative measure whose total mass satisfies $\|u_0\|_{L^1} < 8\pi$, the corresponding radial solution remains well-defined for all $t > 0$ and does not develop singularities (see \cite{NSY1997, B1998, BN1994, M2013, CC2008}). While if $\|u_0\|_{L^1} > 8\pi$, then a finite time blow-up may occur (see \cite{N1995, NSS2000, HV1996, HV1997, M2020, BN1994}). 

The systems \eqref{p} and \eqref{p_0} with $\delta_1 = \delta_2 = 0$ admit a scaling transformation analogous to that of the Keller--Segel system.
Nevertheless, due to its indirect signal production mechanism, the associated scaling properties lead to a distinct notion of criticality. Indeed these systems are invariant under the natural scaling: for $\lambda > 0$
\begin{align}
u_\lambda (x,t) = \lambda^{4}u(\lambda x, \lambda^2 t),\quad v_\lambda (x,t) = v(\lambda x, \lambda^2 t),\quad
w_\lambda(x,t) = \lambda^2 w(\lambda x, \lambda^2 t).\label{scale}
\end{align}
Under this scaling, the critical dimension in the above sense is $d = 4$. Indeed, in \cite{FS2017, FS2019}, the authors revealed several characteristic
features of solutions to \eqref{p_0} in the critical dimension $d=4$
on bounded domains. Here, we restrict ourselves to the case where all parameters are set equal to one. If one has the assumption that the initial data $(u_0,v_0,w_0)$ satisfies $\|u_0\|_{L^1(\Omega)} < (8\pi)^2$, then the corresponding solution of \eqref{p_0} exists globally in time and remains uniformly bounded in time, whereas it was shown that there exists the initial data such that $\|u_0\|_{L^1(\Omega)} \in ((8\pi)^2, \infty) \setminus (8\pi)^2\N$ and the corresponding solution blows up. In the whole space $\R^4$, it was proved in \cite{HL2025} that the solution to \eqref{p_0} exists globally in time whenever $\|u_0\|_{L^1(\R^4)} < (8\pi)^2$.

Such mass critical phenomena are characteristic of scale critical dimensions and are absent in higher, supercritical settings. This naturally raises the question of what determines the long-time dynamics of solutions to \eqref{p} and \eqref{p_0}, and \eqref{KS} when the systems are considered in supercritical dimensions.

\subsection{Parabolic-elliptic model in higher dimensions $d \ge 3$}

In this subsection, we consider the Keller--Segel system \eqref{KS} with $\tau=0$ corresponding to the parabolic-elliptic case:
\begin{equation}\label{KS_0}
\begin{cases}
u_t = \Delta u - \nabla \cdot (u\nabla v)\qquad &\mathrm{in}\ \R^d \times (0,\infty),\\
0 = \Delta v + u\qquad &\mathrm{in}\ \R^d \times (0,\infty)
\end{cases}
\end{equation}
and discuss the long-time behavior of its solutions in supercritical dimensions. 
As noted in previous subsection, a central issue is to identify the mechanisms that control the long-time behavior of solutions. One possible perspective for addressing this issue is provided by the existence of a singular stationary solution, such as the well-known Chandrasekhar steady state solution in \cite{Chan1942}
\begin{align*}
u_C (x) = \dfrac{2(d-2)}{|x|^2}.
\end{align*}
Indeed, such singular stationary solution plays a crucial role in the analysis 
in supercritical dimensions of the solutions to the simplified Keller--Segel system \eqref{KS_0}. In this context, two different criteria have been employed
to characterize the role of the singular stationary solution $u_C$ in the dynamics of \eqref{KS_0}. 

One approach is based on conditions formulated in terms of critical Morrey norms. For instance, according to \cite{BKP2019},
if the initial data $u_0$ with radially symmetry belongs to the critical Morrey space $M^\frac{d}{2}(\R^d)$ which are invariant under the scaling argument \eqref{KS:scale}, and satisfies the upper bound
\begin{align*}
\|u_0\|_{M^\frac{d}{2}(\R^d)} < 2\sigma_d = \|u_C\|_{M^\frac{d}{2}(\R^d)},
\end{align*}
where $\sigma_d$ is the measure of the unit sphere in $\R^d$,
then corresponding solution of \eqref{KS_0} exists globally in time in an appropriate functional class and enjoys the same bound
\begin{align*}
\|u(t)\|_{M^\frac{d}{2}(\R^d)} < 2\sigma_d\quad\mathrm{for\ all}\ t> 0.
\end{align*}
Moreover, it was also demonstrated in \cite{N2021} that for some initial data with $\|u_0\|_{M^{\frac{d}{2}}(\R^d)} > 2\sigma_d$, the behavior of the corresponding solution to \eqref{KS_0} depends on the spatial dimension $d$. 
In dimensions $d\ge 10$, one can choose such $u_0$ so that the corresponding  solution of \eqref{KS_0} blows up in finite time,
whereas in dimensions $3\le d\le 9$, there exists initial data with the same condition such that the corresponding solution of \eqref{KS_0} exists globally in time. 

Another approach focuses on a direct comparison between the initial data and the singular stationary solution itself. Recently, \cite{W2023_2} established that, whenever the initial data $u_0 \in BUC(\R^d)$ satisfies $u_0 \leq u_C$, \eqref{KS_0} admits a unique global classical solution and the solution enjoys the upper bound:
\begin{align*}
\|u(t)\|_{M^\frac{d}{2}(\R^d)} \leq 2\sigma_d\quad\mathrm{for\ all}\ t > 0.
\end{align*}
Furthermore, it was shown in \cite{W2023_2} that, in dimensions $d \ge 10$, if initial data satisfies the asymptotic condition:
\begin{align*}
\liminf_{|x| \to \infty} |x|^2 u_0(x) > 2(d-2),
\end{align*}
which implies that the decay of $u_0$ at infinity is slower than that of the singular stationary solution $u_C$, then the corresponding solution blows up in finite time. Consequently, in supercritical dimensions, the behavior of initial data at spatial infinity is key to the analysis of the long-time dynamics of solutions to \eqref{KS_0}.
In particular, the singular stationary solution $u_C$ appears to act as a critical threshold separating different qualitative behaviors of solutions to \eqref{KS_0}.

\subsection{Statement of the problem}

The aim of the present paper is to provide a detailed description of the behavior of the solutions to \eqref{p} in the supercritical case $d \ge 5$. As noted in the previous subsection, extensive studies on supercritical dimensions have been carried out for the simplified Keller--Segel system \eqref{KS_0}.
These results crucially rely on the fact that, in the parabolic-elliptic setting, the second equation reduces to a Poisson-type equation. A key feature of the Poisson-type structure is that it effectively reduces the system to a single equation for the cell density through the use of a so-called mass function. In fact, once the mass function is defined in this manner $U(r,t) = r^{-d}\int_0^r \rho^{d-1}u(\rho,t) d\rho$, the simplified Keller--Segel system \eqref{KS_0} can be transformed into a single equation governing its evolution as follows:
\begin{align*}
U_t = U_{rr} + \dfrac{d+1}{r}U_r + rU_rU + dU^2.
\end{align*}
At the same time, 
this formulation allows one to make extensive use of comparison principles (see \cite{W2023_2, N2021, BKP2019}).
Moreover, the simplified Keller--Segel system \eqref{KS_0} and the resulting formulation above exhibits a superlinear structure, 
which plays a crucial role in the analysis of aggregation and blow-up phenomena. 

By contrast, the system \eqref{p} remains intrinsically coupled and does not admit a reduction to a single equation nor a superlinear structure. Accordingly, in contrast to the classical setting, the present paper is devoted to clarifying the behavior of the solutions to \eqref{p} in the supercritical case. Since the analytical framework available for the simplified Keller--Segel system is not applicable to \eqref{p}, alternative approaches are required. In this paper, we place particular emphasis on the stationary problem corresponding to \eqref{p}, which we regard as a key step toward understanding the dynamics of solutions to \eqref{p}. The associated stationary system reads
\begin{equation}\label{S0}
\begin{cases}
0 = \Delta u - \nabla \cdot (u\nabla v)\qquad &\mathrm{in}\ \R^d \times (0,\infty),\\
0 = \Delta v + w\qquad &\mathrm{in}\ \R^d \times (0,\infty),\\
0 = \Delta w + u \qquad &\mathrm{in}\ \R^d \times (0,\infty).
\end{cases}
\end{equation}
By combining the second and third equations, this system is equivalently reduced to the following problem:
\begin{equation}\label{S1}
\begin{cases}
0 = \Delta u - \nabla \cdot (u\nabla v)\qquad &\mathrm{in}\ \R^d \times (0,\infty),\\
\Delta^2 v= u \qquad &\mathrm{in}\ \R^d \times (0,\infty).
\end{cases}
\end{equation}
Here focusing on the first equation in \eqref{S1}, one can rewrite it as
$\nabla \cdot (u\nabla (\log u - v)) = 0$ in $\R^d$. Hence, if $u > 0$, the stationary problem formally indicates that $u$ is comparable to $e^v$. 
This heuristic relation will be rigorously justified later (see Proposition~\ref{stationaly sol}).
Consequently we consider a radially symmetric solution of the forth-order elliptic problem:
\begin{equation}\label{S2}
\begin{cases}
\Delta^2 \phi = e^\phi\qquad \mathrm{for}\ r \in [0,\mathrm{R}(\alpha, \beta)),\\
\phi(0) = \log \alpha,\ \Delta \phi(0) = \beta,\ \phi^\prime(0) = (\Delta \phi)^\prime (0) = 0,
\end{cases}
\end{equation}
where $\alpha > 0$ and $\beta \in \R$, and $\mathrm{R}(\alpha, \beta)$ is the maximal interval of existence.
Then it is known from \cite{AGG2006, BFFG2012} that for any $\alpha > 0$ there exists a unique $\beta_0(\alpha) \in [-4de^\frac{\alpha}{2}, 0)$ such that if $\beta = \beta_0(\alpha)$, then $R(\alpha,\beta)=\infty$ and moreover the corresponding unique radial solution
$\phi\in C^4(\R^d)$ satisfies the following asymptotic profile:
\begin{align*}
\lim_{r \to \infty}(\phi(r) + 4\log r) = \log 8(d-4)(d-2).
\end{align*}
Furthermore, thanks to the refined analysis developed in \cite{G2014}, the asymptotic behavior of the radial solution $\phi$ near infinity is known in much greater detail.
In order to clarify the behavior of solutions to \eqref{p} in the supercritical case, we make use of the radially symmetric solution $\phi$ to the stationary problem \eqref{S2} associated with \eqref{p}, characterized by the above asymptotic profile.

\subsection{Notation}
Before stating our main results, we introduce several definitions and remarks
that will be used throughout the paper. Throughout the paper, we work with the homogeneous Morrey spaces $M^p(\R^d)$ of measurable functions on $\R^d$, defined by the norm (see \cite[Section 7.9]{GT1983}):
\begin{align*}
\|f\|_{M^p(\R^d)} := \sup_{x\in \R^d}\sup_{R > 0}R^{d(\frac{1}{p}-1)}
\int_{|x-y| < R} |f(y)| dy < \infty.
\end{align*}
We also set $M^\infty(\R^d) = L^\infty(\R^d)$ and $M^1(\R^d) = L^1(\R^d)$. 
In view of the scaling \eqref{scale}, the Morrey spaces
$M^{\frac{d}{4}}(\R^d)$ and $M^{\frac{d}{2}}(\R^d)$
naturally arise as the critical functional framework
for the analysis of \eqref{p}, corresponding to the variables
$u$ and $w$, respectively. Moreover, for $p\in [1,\infty)$, we define the subspace $\dot{M}^p(\R^d)$ of
$M^p(\R^d)$ by
\begin{align*}
\dot{M}^p(\R^d) := \{ f \in M^p(\R^d); 
\limsup_{R \to 0}R^{d(\frac{1}{p}-1)}\int_{|x-y| < R} |f(y)| dy=0\}
\end{align*}
(see \cite{K1992, B1995, T1992}). This subspace is essential for constructing local-in-time solutions to \eqref{p}
in the critical Morrey spaces $M^\frac{d}{4}(\R^d)$ and $M^\frac{d}{2}(\R^d)$.
Here we basically consider radially symmetric solutions of \eqref{p} in this paper. 
In this case, the norms of the critical Morrey spaces $M^\frac{d}{4}(\R^d)$ and $M^\frac{d}{2}(\R^d)$
are equivalent to
\begin{align*}
\sup_{R > 0}R^{4-d}\int_{|y| < R} |f(y)| dy,\quad \sup_{R > 0}R^{2-d}\int_{|y| < R} |f(y)| dy,
\end{align*}
respectively (see \cite[Proposition 7.1]{BKZ2018}).
Accordingly, in the radial setting considered in this paper, we shall adopt the above expressions as the definitions of the norms on
$M^{\frac{d}{4}}(\R^d)$ and $M^{\frac{d}{2}}(\R^d)$. Moreover, in the radial setting, the Laplacian of a radially symmetric function $f=f(r)$ is given by
\begin{align*}
\Delta f(r) := f_{rr}(r) + \dfrac{d-1}{r}f_r(r).
\end{align*}
As usual, by the Taylor expansion, the value at the origin is defined by $\Delta f(0):= df_{rr}(0)$. Finally, we note that there are several natural ways to treat the second equation in \eqref{p}, which is of Poisson type. One approach is to assume that $v$ is given by the convolution
\begin{align*}
v = E_d * w,
\end{align*}
where $E_d$ denotes the fundamental solution of the Laplacian, given by
\begin{align*}
E_d(x) = \dfrac{1}{(d-2)\sigma_d}\dfrac{1}{|x|^{d-2}}\quad\mathrm{for}\ d\ge 3,
\end{align*}
where $\sigma_d$ is the measure of the unit sphere in $\R^d$. A second approach is to take advantage of radial symmetry and define $v$ by an integral formula:
\begin{align}
v(x,t):= -\int_0^{|x|} s^{1-d}\int_0^s \rho^{d-1}w(\rho,t) d\rho ds + C\quad\mathrm{for}\ (x,t) \in \R^d \times (0,\infty)\label{integral formula}
\end{align}
with some constant $C$. We also denote the radial inverse of the Laplacian
$(-\Delta)^{-1}$ by the following integral form
\begin{align*}
(-\Delta)^{-1}f(x) :=-\int_0^{|x|} s^{1-d}\int_0^s \rho^{d-1}f(\rho) d\rho ds\quad\mathrm{for}\ x \in \R^d \setminus \{0\}
\end{align*}
and extend it to $r=0$ by continuity, so that $(-\Delta)^{-1}f(0)=0$. Throughout this paper, we shall employ these two definitions selectively,
according to the requirements of each theorem.

\subsection{Main results}

Our main results are stated as follows. 

\begin{theorem}[Singular stationary solutions]\label{th:1}
Let $d \ge 7$ and 
\begin{align*}
u_C (x) = \dfrac{8(d-4)(d-2)}{|x|^4},\quad w_C(x) = \dfrac{4(d-2)}{|x|^2},\quad v_C(x) = -4\log |x| + C
\end{align*}
for some constant $C$. Then $(u_C, v_C, w_C)$ is a distributional, stationary solution to the system \eqref{p}.
\end{theorem}

\begin{remark}
This stands in sharp contrast to the simplified Keller--Segel system \eqref{KS_0}, whose singular stationary state is generated by the Chandrasekhar steady state singular solution $|x|^{-2}$. Hence, our system admits a fundamentally different type of singular structure, that the dynamics of our system may display novel features that cannot appear in the theory of the classical Keller--Segel systems. 
\end{remark}

Let us present the result on the global-in-time solutions to the system \eqref{p}.
\begin{theorem}\label{th:2}
Let $d \ge 5$ and let $p \in (\frac{d}{3}, \frac{d}{2})$ be close to $\frac{d}{3}$.
Assume the radially symmetric nonnegative initial data $(u_0,w_0) \in (M^{\frac{d}{4}}(\R^d) \cap \dot{M}^p(\R^d)) \times (M^\frac{d}{2}(\R^d) \cap M^\frac{dp}{d-2p}(\R^d))$ satisfying 
\begin{align}
\|u_0\|_{M^\frac{d}{4}(\R^d)} &:= \sup_{ R > 0} R^{4-d}\int_{|x| < R} u_0 (x) dx < 6(d-2)\sigma_d,\label{th:2,asup1}\\
\|w_0\|_{M^\frac{d}{2}(\R^d)} &:= \sup_{ R > 0} R^{2-d}\int_{|x| < R} w_0 (x) dx < 3 \sigma_d.\label{th:2,asup2}
\end{align}
Then the corresponding solution $(u,v,w)$ 
to the problem \eqref{p} exists globally in time in the sense 
\begin{align*}
C_w([0,\infty); M^{\frac{d}{4}}(\R^d) \cap M^p(\R^d)) \cap \{u: (0, \infty) \to M^r(\R^d); \sup_{t > 0} t^{\frac{d}{2}(\frac{1}{p}-\frac{1}{r})}\|u(t)\|_{M^r(\R^d)} < \infty\}\\
\times C_w([0,\infty); M^\frac{d}{2}(\R^d) \cap M^\frac{dp}{d-2p}(\R^d)),
\end{align*}
where $p < r < \frac{dp}{d-p}$,
such that writing $v(\cdot, t) = E_d * w(\cdot, t)$ for all $t \in (0,\infty)$,
and moreover satisfies the upper bound
\begin{align}\label{th:2:upper}
\|u(t)\|_{M^\frac{d}{4}(\R^d)} < 6(d-2)\sigma_d,\quad \|w(t)\|_{M^\frac{d}{2}(\R^d)} < 3\sigma_d\quad \mathrm{for\ all}\ t > 0.
\end{align}
\end{theorem}

\begin{remark}
The straightforwrd computation implies
\begin{align*}
\|u_C\|_{M^\frac{d}{4}(\R^d)} = 8(d-2)\sigma_d,\quad \|w_C\|_{M^\frac{d}{2}(\R^d)} = 4\sigma_d.
\end{align*}
Therefore, in view of the result \cite{BKP2019} known for the simplified Keller--Segel system \eqref{KS_0}, the conclusion of
Theorem~\ref{th:2} can be regarded as weaker than the expected optimal statement for \eqref{p},
which would assert global-in-time existence together with the same bounds under natural
assumptions. More precisely, we expect under the assumption
\begin{align}
\|u_0\|_{M^\frac{d}{4}(\R^d)} < \|u_C\|_{M^\frac{d}{4}(\R^d)} = 8(d-2)\sigma_d,\quad \|w_0\|_{M^\frac{d}{2}(\R^d)} < \|w_C\|_{M^\frac{d}{2}(\R^d)} = 4\sigma_d,\label{singular:ineq}
\end{align}
the corresponding solution to \eqref{p}, defined in an appropriate functional setting, exists globally in time and enjoys the same bounds. Hence Theorem~\ref{th:2} does not yet reach this level and is weaker in this respect. Nevertheless, such a result in the supercritical case for an essentially
parabolic-parabolic system \eqref{p} is completely new.
\end{remark}

Focusing on the radial solution $\phi$ of \eqref{S2}, we construct both
global-in-time solutions and finite time blow-up of solutions.

\begin{theorem}\label{th:3}
Let $\phi \in C^4(\R^d)$ be the radial solution of \eqref{S2} with $\alpha > 0$ and $\beta_0(\alpha)$ and assume the radially symmetric nonnegative initial data $(u_0,w_0) \in (C(\R^d)\cap L^\infty(\R^d))^2$. Then the following three assertions hold.
\begin{enumerate}
\item[\rm{(i)}] Assume $ 0 \leq u_0 \leq \lambda e^{\phi}$ and $0 \leq w_0 \leq \lambda (-\Delta)\phi$ for all $x \in \R^d$ and with some $\lambda \in (0,1]$. Then the solution $(u,v,w)$ of \eqref{p} exists globally in time and satisfies the following estimates:
\begin{align}\label{th3:bound}
\int_{|x| < R} u(x,t) dx \leq \lambda \int_{|x| < R} e^{\phi(x)} dx,\quad \int_{|x| < R} w(x,t) dx \leq \lambda \int_{|x| < R} (-\Delta)\phi(x) dx
\end{align}
for all $(R,t) \in (0,\infty) \times [0,\infty)$.
\item[\rm{(ii)}] Assume $ \lambda e^{\phi} \leq u_0 $ and $\lambda (-\Delta)\phi \leq w_0$ for all $x \in \R^d$ and with some $\lambda > 1$. Then the solution $(u,v,w)$ of \eqref{p} blows up in finite time, in the sense that
\begin{align*}
\limsup_{t \to T_\mathrm{max}}(\|u(t)\|_{L^\infty(\R^d)} + \|w(t)\|_{L^\infty(\R^d)}) = \infty,
\end{align*}
where $T_\mathrm{max}$ denotes the maximal existence time of the solution $(u,v,w)$ and is finite.
\item[\rm{(iii)}] Assume $ \lambda_1 e^{\phi} \leq u_0 $ and $\lambda_2 (-\Delta)\phi \leq w_0$ for all $x \in \R^d$ and with some $\lambda_1 > \lambda_2 > 1$. Then the solution $(u,v,w)$ of \eqref{p} blows up in finite time.
\end{enumerate}
\end{theorem}

\begin{remark}
The local-in-time existence of a solution to \eqref{p} corresponding to the initial data
$(u_0,w_0)\in (C(\R^d)\cap L^\infty(\R^d))^2$ is guaranteed by
Proposition~\ref{local_2}. Especially the function $v$ is denoted by the integral formula \eqref{integral formula}.
\end{remark}

\begin{remark}
For the radially symmetric solution $\phi$ of the elliptic problem \eqref{S2}
associated with the parameters $\alpha$ and $\beta_0(\alpha)$,
the functions $e^{\phi}$ and $(-\Delta)\phi$ are asymptotic to the singular stationary
solutions $u_C$ and $w_C$, respectively, as $|x|\to\infty$
(see \cite{G2014, AGG2006, BFFG2012} and Section~\ref{chap:sta}). Consequently Theorem~\ref{th:3} highlights that the long-time behavior in supercritical dimensions of solutions to \eqref{p} 
is governed by whether the initial data $(u_0,w_0)$ lies above or below these singular profiles at infinity.
\end{remark}

Thanks to Theorem~\ref{th:3}, we obtain the following result. The following theorems yield an essentially complete classification for the behavior of the solutions to \eqref{p} in supercritical dimensions.

\begin{theorem}\label{th:3.5}
Let $\phi \in C^4(\R^d)$ be the radial solution of \eqref{S2} with $\alpha > 0$ and $\beta_0(\alpha)$ and assume the radially symmetric nonnegative initial data $(u_0,w_0) \in (C(\R^d)\cap L^\infty(\R^d))^2$. Then the following two assertions hold.
\begin{enumerate}
\item[\rm{(i)}]  If $d \ge 13$, whenever the initial data $(u_0,w_0)$ satisfies
\begin{align*}
0 \leq u_0 \leq e^{\phi},\quad 0 \leq w_0 \leq (-\Delta)\phi\quad \mathrm{for\ all}\  x \in \R^d,
\end{align*}
the corresponding solution $(u,v,w)$ of \eqref{p} exists globally in time and satisfies the following estimates:
\begin{align}\label{th3.5:bound}
\sup_{t > 0}\|u(t)\|_{M^\frac{d}{4}(\R^d)} \leq 8(d-2)\sigma_d,\quad \sup_{t > 0}\|w(t)\|_{M^\frac{d}{2}(\R^d)} \leq 4\sigma_d.
\end{align}
\item[\rm{(ii)}] If $5 \leq d \leq 12$, then there exists $\lambda \in (0,1)$ such that, 
whenever the initial data $(u_0,w_0)$ satisfies
\begin{align*}
0 \leq u_0 \leq \lambda e^{\phi},\quad 0 \leq w_0 \leq \lambda (-\Delta)\phi\quad\mathrm{for\ all} \ x \in \R^d,
\end{align*}
the corresponding solution $(u,v,w)$ of \eqref{p} exists globally in time and enjoys the upper bounds \eqref{th3.5:bound}.
\end{enumerate}
\end{theorem}

\begin{remark}
In the case $d\ge13$, the critical Morrey norms of $e^{\phi}$ and $(-\Delta)\phi$
coincide with the critical values, namely,
\begin{align*}
\|e^\phi\|_{M^\frac{d}{4}(\R^d)} = 8(d-2)\sigma_d,\quad \|(-\Delta)\phi\|_{M^\frac{d}{2}(\R^d)} =4\sigma_d
\end{align*}
as proved in Lemma~\ref{d13}. Hence the assumption in Theorem~\ref{th:3.5} (i) is sufficient to ensure that
\begin{align}
\|u_0\|_{M^\frac{d}{4}(\R^d)} \leq 8(d-2)\sigma_d,\quad \|w_0\|_{M^\frac{d}{2}(\R^d)} \leq 4\sigma_d.\label{extend;ineq}
\end{align}
Moreover, when $5\le d\le12$, the above condition can still be achieved
by choosing a suitable value of $\lambda$. Therefore this observation suggests that the initial assumptions \eqref{th:2,asup1} and \eqref{th:2,asup2} in Theorem~\ref{th:2} might be extendable to the condition \eqref{extend;ineq} which includes the limiting case in \eqref{singular:ineq}.
\end{remark}

Finally, we present the following result concerning large initial data.

\newpage
\begin{theorem}\label{th:4}
The following two statements are valid.
\begin{enumerate}
\item[\rm{(i)}] Let $d \ge 13$ and $K_0 > 8(d-2)\sigma_d$, and $K_1 > 4\sigma_d$. Then there exists $(u_0,w_0) \in (C(\R^d)\cap L^\infty(\R^d))^2$ such that
\begin{align*}
\|u_0\|_{M^\frac{d}{4}(\R^d)} =K_0,\quad
\|w_0\|_{M^\frac{d}{2}(\R^d)} =K_1
\end{align*}
and the corresponding solution blows up in finite time.
\item[\rm{(ii)}] Let $5 \leq d \leq 12$. Then there exists $(u_0,w_0) \in (C(\R^d)\cap L^\infty(\R^d))^2$ such that
\begin{align*}
\|u_0\|_{M^\frac{d}{4}(\R^d)} > 8(d-2)\sigma_d,\quad
\|w_0\|_{M^\frac{d}{2}(\R^d)} > 4 \sigma_d
\end{align*}
and the corresponding solution exists globally in time.
\end{enumerate}
\end{theorem}

\begin{remark}
Theorem~\ref{th:4} shows that, for $d\ge13$, the values
$8(d-2)\sigma_d$ and $4\sigma_d$ serve as critical thresholds in the corresponding
Morrey spaces, across which the long-time behavior of solutions to \eqref{p} changes. In this sense, these values are optimal. On the other hand, when $5\leq  d \leq 12$, the situation is different in this respect.
Therefore, it appears highly interesting to investigate the long-time behavior of
solutions in this dimensional range.
\end{remark}

\begin{remark}
For the essentially parabolic-parabolic chemotaxis system \eqref{p}, Theorem~\ref{th:3.5} and Theorem~\ref{th:4}, which yield a sharp criterion and a complete classification of global existence and finite time blow-up of solutions, have not been previously available. Therefore, our approach and results of the present paper are expected to be of independent interest.
\end{remark}

\subsection{Strategy and plan of the paper}

The main strategy and the plan of this paper are summarized as follows. 
A central difficulty in the analysis of system \eqref{p} lies in the fact that,
unlike the classical parabolic-elliptic Keller--Segel system \eqref{KS_0}, it does not reduce to a single equation nor possess an explicit superlinear structure.
As a consequence, many techniques developed for the Keller--Segel system cannot be applied directly. 

To overcome this difficulty, we introduce and exploit new comparison principles for the radial version of the original system \eqref{p}, together with suitably defined mass functions \eqref{mass:1} and \eqref{mass:2} associated with \eqref{p}.
These comparison principles, developed in Section~\ref{chap:2}, form a key analytical tool.  
More precisely, the comparison principle established for the original system
\eqref{p} allows us to construct local-in-time solutions from the initial data $(u_0,w_0) \in (C(\R^d) \cap L^\infty(\R^d))^2$.
In contrast, the comparison principle for the associated mass functions serves as a key ingredient in controlling the dynamical behavior of solutions to the system \eqref{mass:3}
for the mass function. 

In Section~\ref{chap:loc}, we establish local-in-time existence of solutions to \eqref{p} in appropriate function spaces. As a first step, we construct local-in-time solutions with $v = E_d * w$ in the scaling-critical Morrey spaces.
Such solutions are obtained by a standard application of the Banach fixed point
theorem (see \cite{BKP2019, B1995, FS2017, HL2025}). As a second approach, we construct the solutions using the comparison principle established in Section~\ref{chap:2} and radial symmetry.
In this case, the potential $v$ is expressed through a radial integral representation \eqref{integral formula}.

In Section~\ref{chap:th_1}, we present the proofs of Theorem~\ref{th:1} and Theorem~\ref{th:2}.
In particular, the proof of Theorem~\ref{th:2} is inspired by the approach developed in
\cite{BKP2019}.
However, since the system under consideration is essentially of
parabolic-parabolic type, consisting of two genuinely coupled equations,
additional care is required to adapt this method to the present setting.

In Section~\ref{chap:blowup}, we develop the analytical framework needed to prove convergence of solutions toward the stationary state. However, we cannot rely on arguments based on the parabolic-elliptic structure of the simplified Keller--Segel system \eqref{KS_0} (see \cite{W2023_2, N2021}). Instead, we construct a novel subsolution to the system \eqref{mass:3} satisfied by the mass functions that exhibits finite time blow-up. This approach enables us to prove that infinite time blow-up can occur only
at the origin, which in turn leads to convergence of the mass functions
toward the stationary state.
We stress that the novelty of our approach lies not only in the construction
of a new blowing up subsolution itself,
but also in the strategy behind its construction,
which may have further applications to other nonlinear
parabolic systems beyond the classical Keller--Segel system.

In Section~\ref{chap:sta}, we investigate the convergence toward stationary solutions and
their qualitative properties.
We first show that the stationary problem \eqref{S0} associated with \eqref{p} coincides with the forth-order elliptic problem \eqref{S2}.
This fact, together with the assumptions of Theorem~\ref{th:3}, allows the comparison principles developed earlier to be applied effectively in the proof of Theorem~\ref{th:3} (see the next section). Next, we establish convergence of the mass functions toward their stationary states.
After that, starting from the stationary mass functions, we construct the
corresponding stationary solutions to the original system \eqref{p}. Unlike approaches that focus primarily on the stationary mass function (see \cite{W2023_2, N2021}), we place particular emphasis on the analysis of the asymptotic
behavior of stationary solutions to the original problem \eqref{p}.
This constitutes a new perspective in the study of the system.

In Section~\ref{chap:th_2}, we provide the proof of Theorem~\ref{th:3} based on the results established
in the preceding sections.
A key point of the argument is that the proof relies on a detailed analysis of
the stationary solutions to the original problem \eqref{p}, rather than on the stationary problem for the mass functions.

Finally, in Section~\ref{chap:th_3}, motivated by the results \cite{G2014, BFFG2012} showing that the asymptotic behavior of radially symmetric solutions to \eqref{S2} depends on the
space dimension, we present the proofs of Theorem~\ref{th:3.5} and Theorem~\ref{th:4}.

\section{Comparison principles}\label{chap:2}

In this section, we develop the comparison principles that serve as a fundamental tool in proving the main theorems of this paper. The first comparison principle is formulated directly for the original system \eqref{p}.
In particular, it allows us to construct ODE–type subsolutions and supersolutions to \eqref{p}, which will play a crucial role in the proof of Proposition~\ref{local_2}.

\begin{lemma}\label{CP_1}
Let $d \ge 1$ and $T > 0$, and let $h \in C([0,\infty) \times [0,T))$ be such that
\begin{align*}
\sup_{(r,t) \in (0,\infty)\times (0,T_0)} |h(r,t)| < \infty\quad \mathrm{for\ any}\ T_0 < T.
\end{align*}
Assume the nonnegative functions $(\bar{z},	\underline{z}, \bar{y}, \underline{y}) \in [C([0,\infty) \times [0,T)) \cap C^{2,1}([0,\infty)\times (0,T))]^4$ satisfy 
\begin{align*}
\sup_{(r,t) \in (0,\infty)\times (0,T_0)} \{|\bar{z}(r,t)| + |\underline{z}(r,t)| + |\bar{y}(r,t)| + |\underline{y}(r,t)|\} < \infty \quad \mathrm{for\ any}\ T_0 < T
\end{align*}
and in addition for all $(r,t) \in (0,\infty) \times (0,T)$, 
\begin{equation}\label{ine:CP_1}
\begin{cases}
\underline{z}_t - \underline{z}_{rr} - \dfrac{d-1}{r}\underline{z}_r - r h \underline{z}_r - \underline{z}\underline{y} \leq \bar{z}_t-\bar{z}_{rr} - \dfrac{d-1}{r}\bar{z}_r - rh\bar{z}_r - \bar{z}\bar{y},\\[1ex]
\underline{y}_t - \underline{y}_{rr} - \dfrac{d-1}{r}\underline{y}_r - \underline{z}\leq \bar{y}_t - \bar{y}_{rr} - \dfrac{d-1}{r}\bar{y}_r - \bar{z}.
\end{cases}
\end{equation}
One can conclude that for all $(r,t) \in (0,\infty) \times [0,T)$,
\begin{align*}
\underline{z}(r,t) \leq \bar{z}(r,t),\quad \underline{y}(r,t) \leq \bar{y}(r,t)
\end{align*}
if the following conditions hold that for all $r \in (0,\infty),$
\begin{align}\label{asp:CP_1}
\underline{z}(r,0) \leq \bar{z}(r,0),\quad \underline{y}(r,0) \leq \bar{y}(r,0).
\end{align}
\end{lemma}

\begin{proof}
We fix $T_0 < T$ and let $\varepsilon > 0$ and $(\alpha, \beta)$ be such that
\begin{align*}
\alpha = \sup_{(r,t) \in (0,\infty) \times (0,T_0)} \{|\underline{y}(r,t)| + |\bar{z}(r,t)| + 2|h(r,t)|\} + 2d,\quad \beta > 1.
\end{align*}
Here our assumption implies $\alpha < \infty$.
Let us denote $(\varphi,\psi)$ by
\begin{align*}
\varphi := e^{-\alpha t}(\underline{z}(r,t)-\bar{z}(r,t)) -\varepsilon e^{\beta t}\xi(r),\quad \psi := e^{-\alpha t}(\underline{y}(r,t)-\bar{y}(r,t))-\varepsilon e^{\beta t}\xi(r)
\end{align*}
for all $(r,t) \in [0,\infty) \times [0,T_0]$, where $\xi (r) = r^2 + 1$ for all $r \ge 0$.
We now can write \eqref{ine:CP_1} as
\begin{align}\label{ine1:CP_1}
\varphi_t - \varphi_{rr} - \dfrac{d-1}{r}\varphi_r -rh\varphi_r
&\leq e^{-\alpha t}(\underline{z}-\bar{z})(\underline{y}-\alpha) + e^{-\alpha t}\bar{z}(\underline{y}-\bar{y}) -\varepsilon\beta e^{\beta t}\xi(r)\notag\\
&\hspace{1cm} + 2\varepsilon e^{\beta t} + 2(d-1)\varepsilon e^{\beta t} + 2r^2 h \varepsilon e^{\beta t}
\end{align}
and also as
\begin{align}\label{ine2:CP_1}
\psi_t - \psi_{rr} - \dfrac{d-1}{r}\psi_r \leq \varphi - \alpha e^{-\alpha t }(\underline{y}-\bar{y}) - 
\varepsilon(\beta - 1)e^{\beta t}\xi(r) + 2d \varepsilon e^{\beta t}.
\end{align}
We note that $\phi$ and $\psi$ diverge to $-\infty$ uniformly in $t \in [0,T_0]$ as $r \to \infty$.
Thanks to the initial condition \eqref{asp:CP_1} and the continuity of $\underline{z}, \bar{z}, \underline{y}, \bar{y}$, setting 
\begin{align*}
t_0 := \sup\{\hat{T} \in (0,T_0) ; \varphi(r,t) < 0\quad \mathrm{for}\ (r,t) \in [0,\infty) \times [0,\hat{T}]\},
\end{align*}
we confirm that $t_0$ is well-defined and positive with $t_0 \leq T_0$. Here if $t_0 < T_0$, from the continuity of these functions we may find the local maximum point $r_0 \in [0,\infty)$ of the function $\varphi$, that is $\varphi (r_0,t_0) = 0$, 
where
\begin{align*}
\varphi_t (r_0,t_0) \ge 0,\quad \varphi_r (r_0,t_0) = 0,\quad \Delta \varphi (r_0, t_0)\leq 0.
\end{align*}
Applying the above fact to the inequality \eqref{ine1:CP_1} and setting $\gamma = \sup_{(r,t) \in (0,\infty)\times (0,T_0)} |h(r,t)|$, we get
\begin{align}
0&\leq \varphi_t(r_0,t_0) - \Delta \varphi (r_0,t_0)\notag\\
&\leq r_0 h(r_0,t_0) \varphi_r(r_0,t_0) + e^{-\alpha t_0} (\underline{z}-\bar{z})(\underline{y}-\alpha)(r_0,t_0) + e^{-\alpha t_0}\bar{z}(\underline{y}-\bar{y})(r_0,t_0) - \varepsilon\beta e^{\beta t_0}\xi(r_0)\notag\\
&\hspace{1cm} + 2\varepsilon e^{\beta t_0} + 2(d-1)\varepsilon e^{\beta t_0} + 2r_0^2 h(r_0,t_0) \varepsilon e^{\beta t_0}\notag\\
&= e^{-\alpha t_0} (\underline{z}-\bar{z})(\underline{y}-\alpha)(r_0,t_0) + e^{-\alpha t_0}\bar{z}(\underline{y}-\bar{y})(r_0,t_0) - \varepsilon\beta e^{\beta t_0}\xi(r_0)\notag\\
&\hspace{1cm} + 2d\varepsilon e^{\beta t_0} + 2r_0^2 h(r_0,t_0) \varepsilon e^{\beta t_0}\notag\\
&\leq \varepsilon e^{\beta t_0}\xi(r_0) (\underline{y} + 2\gamma + 2d -\alpha)(r_0,t_0) + e^{-\alpha t_0}\bar{z}(\underline{y}-\bar{y})(r_0,t_0) - \varepsilon\beta e^{\beta t_0}\xi(r_0).\label{ine3:CP_1}
\end{align}
To finish the proof, we will next show that $\psi (r,t) \leq 0$ for any $(r,t) \in [0,\infty) \times [0,t_0]$. Since the definition of $t_0$ ensures that $\varphi(r,t) < 0$ for all $(r,t) \in [0,\infty) \times [0,t_0)$,  we can now proceed analogously to the argument developed above. In fact if our claim is false, then there exists $(r_1, t_1) \in [0,\infty) \times (0,t_0)$ such that
\begin{align*}
\psi(r_1,t_1) = 0,\quad\psi_t (r_1,t_1) \ge 0,\quad \psi_r(r_1,t_1) = 0,\quad \Delta \psi(r_1,t_1) \leq 0.
\end{align*}
This yields that
\begin{align*}
0 &\leq \psi_t (r_1,t_1) - \Delta\psi(r_1,t_1)\\
&\leq \varphi(r_1,t_1) -\alpha e^{-\alpha t_1}(\underline{y}-\bar{y})(r_1,t_1)
-\varepsilon (\beta -1)e^{\beta t_1}\xi(r_1) + 2d\varepsilon e^{\beta t_1}\\
&\leq -\alpha \varepsilon e^{\beta t_1} \xi (r_1) -\varepsilon (\beta -1)e^{\beta t_1}\xi(r_1) + 2d\varepsilon e^{\beta t_1}\xi(r_1)\\
&= \varepsilon e^{\beta t_1}\xi(r_1)(2d - \alpha) - \varepsilon (\beta -1)e^{\beta t_1}\xi(r_1)\\
&\leq -\varepsilon (\beta -1)e^{\beta t_1}\xi(r_1) < 0, 
\end{align*}
which is not reasonable. Thus we obtain that $\psi(r,t) < 0$ for all $(r,t) \in [0,\infty) \times [0,t_0)$, as a consequence of the continuity of $\psi$, we have $\psi (r_0,t_0) \leq 0$. Returning to the inequality \eqref{ine3:CP_1}
once again, we infer from the definition of $\alpha$ 
that 
\begin{align*}
0 &\leq \varphi_t(r_0,t_0) -\Delta\varphi(r_0,t_0)\\
&\leq \varepsilon e^{\beta t_0}\xi(r_0)(\underline{y} + 2\gamma + 2d-\alpha)(r_0,t_0) + \varepsilon e^{\beta t_0}\xi(r_0) \bar{z}(r_0,t_0) - \varepsilon \beta e^{\beta t_0}\xi(r_0)\\
&= \varepsilon e^{\beta t_0} \xi(r_0) (\underline{y} + \bar{z} + 2\gamma + 2d- \alpha)(r_0,t_0) -\varepsilon\beta e^{\beta t_0}\xi(r_0) < 0, 
\end{align*}
which is impossible. Consequently we conclude that $t_0 = T_0$ and we complete the proof by carrying out an analogous argument for the inequality \eqref{ine2:CP_1} and after that letting by $\varepsilon \to 0$ and then $T_0 \to T$.
\end{proof}

Let $d \ge 1$ and $T > 0$, and we assume the radially symmetric nonnegative initial data $(u_0,w_0) \in (C(\R^d) \cap L^\infty(\R^d))^2$. We consider the radially nonnegative solution $(u,v,w)$ of \eqref{p} in the classical sense so that
\begin{equation*}
\begin{cases}
u \in C(\R^d \times [0,T)) \cap C^{2,1}(\R^d \times (0,T)),\\
v \in C^{4,0}(\R^d \times (0,T)),\\
w \in C(\R^d \times [0,T)) \cap C^{2,1}(\R^d \times (0,T)).
\end{cases}
\end{equation*}
By setting the mass function $(M,W)$ by
\begin{equation}\label{mass:1}
M(r,t):=
\begin{cases}
r^{-d} \int_0^r \rho^{d-1}u(\rho,t) d\rho\quad &\mathrm{in}\ (0,\infty) \times [0,T),\\
\frac{u(0,t)}{d}\quad &\mathrm{in}\  \{0\} \times [0,T),
\end{cases}
\end{equation}
\begin{equation}\label{mass:2}
W(r,t):=
\begin{cases}
r^{-d} \int_0^r \rho^{d-1}w(\rho,t) d\rho\quad &\mathrm{in}\ (0,\infty) \times [0,T),\\
\frac{w(0,t)}{d}\quad &\mathrm{in}\  \{0\} \times [0,T).
\end{cases}
\end{equation}
One can readily verify that the mass function $(M,W) \in [C([0,\infty)\times [0,T)) \cap C^{2,1}([0,\infty) \times (0,T))]^2$
satisfies the parabolic system
\begin{equation}\label{mass:3}
\begin{cases}
M_t = M_{rr} + \dfrac{d+1}{r}M_r + rM_r W + dMW\quad &\mathrm{in}\ (0,\infty) \times (0,T),\\
W_t = W_{rr} + \dfrac{d+1}{r}W_r + M \quad &\mathrm{in}\ (0,\infty) \times (0,T),\\
M_r(0,t) = W_r(0,t) = 0 \quad &\mathrm{in}\ (0,T),\\
M(r,0) = M_0(r), W(r,0) = W_0(r) \quad &\mathrm{in}\ (0, \infty).
\end{cases}
\end{equation}

We then establish a comparison principle for the mass function system \eqref{mass:3} associated with \eqref{p}.
Here we emphasize that, unlike the simplified Keller--Segel system \eqref{KS_0}, transforming \eqref{p} into the mass variables $(M,W)$ does not reduce the problem to a single scalar parabolic equation, namely the resulting system \eqref{mass:3} remains genuinely coupled. 

\begin{lemma}\label{CP_2}
Let $d \ge 1$ and $T > 0$, and $R \in (0,\infty]$. Define
\begin{align*}
\mathcal{Q}_{T_0} := \{ (r,t) \in \R^2; r \in (0,R), t \in [0,T_0)\}\quad \mathrm{for}\ T_0 \in (0,T].
\end{align*}
Assume that $(\underline{M}, \overline{M}, \underline{W}, \overline{W})$ belongs to $(C^1(\overline{\mathcal{Q}_{T}}))^4$ 
satisfying
\begin{align}
\mathrm{either}\ r\underline{M}_r + d\underline{M} \ge 0\ \mathrm{or}\ r\overline{M}_r + d\overline{M} \ge 0\ \mathrm{in}\ \mathcal{Q}_{T}\label{CP:asp:ineq}
\end{align}
as well as the boundedness condition
\begin{align*}
\sup_{(r,t) \in \mathcal{Q}_{T_0}}\{|\underline{M}(r,t)| + |\overline{M}(r,t)| + |\underline{W}(r,t)| + |\overline{W}(r,t)|\} < \infty\quad \mathrm{for\ all}\ T_0 \in (0,T)
\end{align*}
and
\begin{align*}
\mathrm{either}\ \sup_{(r,t) \in \mathcal{Q}_{T_0}} |r\underline{M}_r(r,t)|< \infty\ \mathrm{or}\ \sup_{(r,t) \in \mathcal{Q}_{T_0}}|r\overline{M}_r(r,t)| < \infty\quad \mathrm{for\ all}\ T_0 \in (0,T)
\end{align*}
and such that for all $(r,t) \in (0,R) \times (0,T)$
\begin{equation}\label{ine1:CP_2}
\begin{cases}
\underline{M}_t - \underline{M}_{rr} - \dfrac{d+1}{r}\underline{M}_r - r \underline{M}_r \underline{W} - d\underline{M}\underline{W}\\
\hspace{1cm} \leq \overline{M}_t - \overline{M}_{rr} - \dfrac{d+1}{r}\overline{M}_r - r\overline{M}_r \overline{W} - d \overline{M}\overline{W},\\
\underline{W}_t - \underline{W}_{rr} - \dfrac{d+1}{r}\underline{W}_r - \underline{M}\\
\hspace{1cm}\leq 
\overline{W}_t - \overline{W}_{rr} - \dfrac{d+1}{r}\overline{W_r} - \overline{M}.
\end{cases}
\end{equation}
Assume further that the initial ordering holds:
\begin{align}\label{asp:CP_2}
\underline{M}(r,0) \leq \overline{M}(r,0), \quad \underline{W}(r,0) \leq \overline{W}(r,0)\quad\mathrm{for\ all}\ r \in (0,R).
\end{align}
Then it holds that:
\begin{enumerate}
\item[\rm{(i)}] if $R = \infty$, then for all $(r,t) \in \mathcal{Q}_T$
\begin{align*}
\underline{M}(r,t) \leq \overline{M}(r,t), \quad \underline{W}(r,t) \leq \overline{W}(r,t).
\end{align*}
\item[\rm{(ii)}] if $R < \infty$ and in addition for all $t \in (0,T)$
\begin{align*}
\underline{M}(R,t) \leq \overline{M}(R,t), \quad \underline{W}(R,t) \leq \overline{W}(R,t),
\end{align*}
then for all $(r,t) \in \mathcal{Q}_T$
\begin{align*}
\underline{M}(r,t) \leq \overline{M}(r,t), \quad \underline{W}(r,t) \leq \overline{W}(r,t).
\end{align*}
\end{enumerate}

\end{lemma}

\begin{proof}
The case (ii) can be handled by an entirely similar argument, so we restrict ourselves to the proof of (i).
We fix $T_0 < T$ and let $\varepsilon > 0$ and $(\alpha, \beta)$ be such that
\begin{align*}
\alpha = \sup_{(r,t) \in \mathcal{Q}_{T_0}} \{ (d+2) |\underline{W}(r,t)| + d|\overline{M}(r,t)| + |r\overline{M}_r(r,t)|\} + 2(d+2),\quad \beta > 1.
\end{align*}
Here we note that our assumption ensures $\alpha < \infty$.
Denoting $(\varphi,\psi)$ by
\begin{align*}
\varphi := e^{-\alpha t}(\underline{M}(r,t)-\overline{M}(r,t)) -\varepsilon e^{\beta t}\xi(r),\quad \psi := e^{-\alpha t}(\underline{W}(r,t)-\overline{W}(r,t))-\varepsilon e^{\beta t}\xi(r)
\end{align*}
for all $(r,t) \in [0,\infty) \times [0,T_0]$, where $\xi (r) = r^2 + 1$ for all $r \ge 0$, 
we may then rewrite the system \eqref{ine1:CP_2} in the form
\begin{align}
\varphi_t - \varphi_{rr} - \dfrac{d+1}{r}\varphi_r
&\leq e^{-\alpha t}(\underline{M}-\overline{M})(d\underline{W}-\alpha) + e^{-\alpha t}(\underline{W}-\overline{W})(r\overline{M}_r + d\overline{M})\notag\\
&\hspace{1cm} + e^{-\alpha t}r\underline{W}(\underline{M}_r - \overline{M}_r) -\varepsilon\beta e^{\beta t}\xi(r) + 2(d+2)\varepsilon e^{\beta t}\label{ine2:CP_2}
\end{align}
and further
\begin{align}\label{ine3:CP_2}
&\psi_t - \psi_{rr} - \dfrac{d+1}{r}\psi_r\notag\\
&\hspace{1cm} \leq \varphi -\alpha e^{-\alpha t}(\underline{W}(r,t)-\overline{W}(r,t))- \varepsilon(\beta - 1)e^{\beta t}\xi(r) + 2(d+2)\varepsilon e^{\beta t}.
\end{align}
Using the initial condition \eqref{asp:CP_2} together with the continuity of $\underline{M}, \overline{M}, \underline{W}, \overline{W}$, we define
\begin{align*}
t_0 := \sup\{\hat{T} \in (0,T_0) ; \varphi(r,t) < 0\quad \mathrm{for}\ (r,t) \in [0,\infty) \times [0,\hat{T}]\}.
\end{align*}
This quantity $t_0$ is strictly positive and satisfies $t_0 \leq T_0$. Here if $t_0 < T_0$, then continuity allows us to choose $r_0 \in [0,\infty)$ such that $\varphi (\cdot, t_0)$ attains a local maximum at $r_0$.
Hence at this point we have
\begin{align*}
\varphi_t (r_0,t_0) \ge 0,\quad \varphi_r (r_0,t_0) = 0,\quad \Delta_{d + 2} \varphi (r_0, t_0)\leq 0,
\end{align*}
which $\Delta_{d + 2}$ denotes Laplacian in $\R^{d+2}$.
Applying the above considerations to the inequality \eqref{ine2:CP_2}, we arrive at
\begin{align}
0&\leq \varphi_t(r_0,t_0) - \Delta_{d + 2} \varphi (r_0,t_0)\notag\\
&\leq \varepsilon e^{\beta t_0} \xi(r_0)(d\underline{W}-\alpha)(r_0,t_0) + e^{-\alpha t_0} (\underline{W}-\overline{W})(r\overline{M}_r + d\overline{M})(r_0,t_0)\notag\\
&\hspace{1cm} + 2r_0^2 \underline{W}(r_0,t_0)\varepsilon e^{\beta t_0} -\varepsilon \beta e^{\beta t_0}\xi(r_0) + 2(d+2) \varepsilon e^{\beta t_0}\notag\\
&\leq \varepsilon e^{\beta t_0} \xi(r_0)((d+2)\underline{W} + 2(d+2)-\alpha)(r_0,t_0) - \varepsilon \beta e^{\beta t_0}\xi(r_0)\notag\\
&\hspace{1cm}+ e^{-\alpha t_0} (\underline{W}-\overline{W})(r\overline{M}_r + d\overline{M})(r_0,t_0)
.\label{ine4:CP_2}
\end{align}
The final step is to show that $\psi (r,t) \leq 0$ for any $(r,t) \in [0,\infty) \times [0,t_0]$.
Because $\varphi(r,t)$ stays strictly negative for all $(r,t) \in [0,\infty) \times [0,t_0)$, the same contradiction argument can be carried out for $\psi$.
If the claim failed, then some point $(r_1, t_1) \in [0,\infty) \times (0,t_0)$ would satisfy
\begin{align*}
\psi(r_1,t_1) = 0,\quad\psi_t (r_1,t_1) \ge 0,\quad \psi_r(r_1,t_1) = 0,\quad \Delta_{d + 2} \psi(r_1,t_1) \leq 0.
\end{align*}
This leads to 
\begin{align*}
0 &\leq \psi_t (r_1,t_1) - \Delta_{d + 2}\psi(r_1,t_1)\\
&\leq \varphi(r_1,t_1) -\alpha e^{-\alpha t_1}(\underline{W}(r_1,t_1)-\overline{W}(r_1,t_1))- \varepsilon(\beta - 1)e^{\beta t_1}\xi(r_1) + 2(d+2)\varepsilon e^{\beta t_1}\\
&\leq -\alpha \varepsilon e^{\beta t_1}\xi(r_1)-\varepsilon (\beta -1)e^{\beta t_1}\xi(r_1)
 + 2(d+2)\varepsilon e^{\beta t_1}\\
 &\leq \varepsilon e^{\beta t_1} \xi (r_1) (2(d+2) - \alpha) -\varepsilon (\beta -1)e^{\beta t_1}\xi(r_1) < 0,
\end{align*}
which is a contradiction. Hence by continuity of $\psi$, this further implies that $\psi (r_0,t_0) \leq 0$.
We now recall the standing assumption:
\begin{align*}
r\overline{M}_r + d\overline{M} \ge 0\quad \mathrm{for\ all} \ (r,t) \in \mathcal{Q}_{T_0}.
\end{align*}
Here we remark that a similar argument can be carried out under the alternative assumption 
\begin{align*}
r\underline{M}_r + d\underline{M} \ge 0.
\end{align*}
Indeed, by a slight modification of the computation leading to \eqref{ine2:CP_2}, the following analysis remains valid under the assumption 
\begin{align*}
\sup_{(r,t) \in \mathcal{Q}_{T_0}}|r\underline{M}_r(r,t)| < \infty\quad \mathrm{for\ all}\ T_0 \in (0,T).
\end{align*}
Therefore returning once again to inequality \eqref{ine4:CP_2}, one can check the following computation from the definition of $\alpha$.
\begin{align*}
0 &\leq \varphi_t(r_0,t_0) -\Delta_{d+2}\varphi(r_0,t_0)\\
&\leq \varepsilon e^{\beta t_0} \xi(r_0)((d+2)\underline{W} + 2(d+2)-\alpha)(r_0,t_0) - \varepsilon \beta e^{\beta t_0}\xi(r_0)\notag\\
&\hspace{1cm}+ e^{-\alpha t_0} (\underline{W}-\overline{W})(r\overline{M}_r + d\overline{M})(r_0,t_0)\\
&\leq \varepsilon e^{\beta t_0} \xi(r_0)((d+2)\underline{W} + 2(d+2)-\alpha)(r_0,t_0) - \varepsilon \beta e^{\beta t_0}\xi(r_0)\\
&\hspace{1cm}+ \varepsilon e^{\beta t_0}\xi(r_0) (r\overline{M}_r + d\overline{M})(r_0,t_0)\\
&\leq \varepsilon e^{\beta t_0} \xi(r_0) ( \sup_{(r,t) \in \mathcal{Q}_{T_0}} \{ (d+2) |\underline{W}(r,t)| + d|\overline{M}(r,t)| + |r\overline{M}_r(r,t)|\} + 2(d+2) - \alpha)\\
&\hspace{1cm} - \varepsilon \beta e^{\beta t_0} \xi(r_0)\\
& < 0.
\end{align*}
This is impossible, so that we must have $t_0 = T_0$. An entirely analogous argument applied to the inequality \eqref{ine3:CP_2}, followed by letting $\varepsilon \to 0$ and then $T_0 \to T$, completes the proof.
\end{proof}

\section{Local-in-time solutions}\label{chap:loc}
In this section, we begin by introducing local existence of classical solutions to \eqref{p} and provide two approaches to the construction of solutions. The first one is based on representing $v = E_d * w$ by the Newtonian potential,
which is particularly useful for establishing Theorem~\ref{th:2}. The second approach relies on exploiting radial symmetry without using the Newtonian representation of $v$. This second construction provides the regularity properties required in the analysis of Theorem~\ref{th:3} and Theorem~\ref{th:3.5}, and Theorem~\ref{th:4}. 
To verify the first existence result, we establish several auxiliary lemmas.
\begin{lemma}[{\cite[Proposition 3.2]{GM1989}}]\label{morrey}
Let $(e^{t\Delta})_{ t\ge0}$ be the heat semigroup on $\R^d$. Then there exists $C > 0$ independent of $t > 0$ such that for all $1 \leq p \leq q \leq \infty$
\begin{align}
\|e^{t\Delta}f\|_{M^q(\R^d)} &\leq Ct^{-\frac{d}{2}(\frac{1}{p}-\frac{1}{q})}\|f\|_{M^p(\R^d)},\label{morrey:1}\\
\|\nabla e^{t\Delta}f\|_{M^q(\R^d)} &\leq Ct^{-\frac{d}{2}(\frac{1}{p}-\frac{1}{q})-\frac{1}{2}}\|f\|_{M^p(\R^d)}\label{morrey:2}
\end{align}
for all $f \in M^p(\R^d)$ and $t > 0$.
\end{lemma}

We also recall basic estimates for Riesz potentials on Morrey spaces.

\begin{lemma}[{\cite[Proposition 3.1]{GM1989}}]\label{Riesz potential}
The following estimates hold in Morrey spaces:
\begin{enumerate}
\item[\rm{(i)}] if $1 \leq p < d < r \leq \infty$ and $f \in M^p(\R^d)$, then it holds that $\nabla E_d * f \in M^q(\R^d)^d$ and there exists $C > 0$ independent of $f$
such that
\begin{align}\label{Riesz potential:1}
\|\nabla E_d * f\|_{M^q(\R^d)} \leq C\|f\|_{M^p(\R^d)}
\end{align}
with
\begin{align*}
\frac{1}{q} = \frac{1}{p}-\frac{1}{d}
\end{align*}
\item[\rm{(ii)}] if $1 \leq p < d < r \leq \infty$ and $f \in M^p(\R^d) \cap M^r(\R^d)$, then $\nabla E_d * f \in L^\infty(\R^d)^d$ and there exists $C > 0$ independent of $f$ such that
\begin{align}\label{Riesz potential:2}
\|\nabla E_d * f\|_{L^\infty(\R^d)} \leq C \|f\|_{M^p(\R^d)}^\nu \|f\|_{M^r(\R^d)}^{1-\nu}
\end{align}
with 
\begin{align*}
\nu = \dfrac{\frac{1}{d}-\frac{1}{r}}{\frac{1}{p}-\frac{1}{r}}.
\end{align*}
\end{enumerate}
\end{lemma}

We construct a local-in-time mild solution of \eqref{p} in critical spaces $M^\frac{d}{4}(\R^d)$ and $M^\frac{d}{2}(\R^d)$.

\begin{proposition}\label{local_1}
Let $d \ge 5$ and $p \in (\frac{d}{3}, \frac{d}{2})$.
Given $(u_0,w_0) \in (M^{\frac{d}{4}}(\R^d) \cap \dot{M}^p(\R^d)) \times (M^\frac{d}{2}(\R^d) \cap M^\frac{dp}{d-2p}(\R^d))$, there exist $T \in (0,\infty]$ and a nonnegative solution $(u,v,w)$ of \eqref{p} such that
\begin{align*}
&u \in C_w([0,T]; M^{\frac{d}{4}}(\R^d) \cap M^p(\R^d)) \cap \{u: (0, T) \to M^r(\R^d); \sup_{0 < t \leq T} t^{\frac{d}{2}(\frac{1}{p}-\frac{1}{r})}\|u(t)\|_{M^r(\R^d)} < \infty\},\\
&w \in C_w([0,T]; M^\frac{d}{2}(\R^d) \cap M^\frac{dp}{d-2p}(\R^d)),
\end{align*}
where $p < r < \frac{dp}{d-p}$, and that $v(\cdot, t) = E_d * w(\cdot, t)$ for all $t \in (0, T)$.
\end{proposition}

\begin{remark}
The assumptions of the parameters $p$ and $r$ ensure that 
\begin{align*}
 r < d  < \dfrac{dp}{d-2p}.
\end{align*}
This condition is essential in establishing a local-in-time solution via the Banach fixed-point argument.
\end{remark}

\begin{remark}
In view of \cite{B1995} and \cite[Lemma 2.2]{K1992},
the membership $u_0 \in \dot{M}^p(\R^d)$ is a necessary condition for constructing a local-in-time solution to \eqref{p} in the critical Morrey space. Indeed, when $u_0 \in \dot{M}^p(\R^d)$ it holds that
\begin{align*}
\lim_{t \to 0}t^{\frac{d}{2}(\frac{1}{p}-\frac{1}{r})}\|e^{t\Delta}u_0\|_{M^r(\R^d)} = 0.
\end{align*}
Hence, by choosing $T > 0$ sufficiently small, this term can be made arbitrarily small and therefore controlled within the fixed point framework.

\end{remark}

\begin{remark}
As pointed out in \cite[Section 2 and Lemma 3.1]{K1992}, one generally obtains only weak convergence of $e^{t\Delta}u_0$ and $e^{t\Delta}w_0$ when $u_0 \in M^\frac{d}{4}(\R^d)$ and $w_0 \in M^\frac{d}{2}(\R^d)$. However, for every 
$t > 0$ the heat semigroup is strongly continuous in the corresponding Morrey norms. Moreover we can confirm that the mild solution established by Proposition~\ref{local_1} can be shown to extend to a nonnegative classical solution for all $t > 0$, by employing the iteration method developed in \cite{T1992} and the parabolic maximum principle.
\end{remark}

\begin{proof}
Let $T > 0$. We denote the space $\mathcal{X}_T$ by
\begin{align*}
\mathcal{X}_T &:= C_w([0,T]; M^{\frac{d}{4}}(\R^d) \cap M^p(\R^d))\\
&\hspace{1cm} \cap \{u: (0, T) \to M^r(\R^d); \sup_{0 < t \leq T} t^{\frac{d}{2}(\frac{1}{p}-\frac{1}{r})}\|u(t)\|_{M^r(\R^d)} < \infty\}\\
&\hspace{0.5cm}\times C_w([0,T]; M^\frac{d}{2}(\R^d) \cap M^\frac{dp}{d-2p}(\R^d))
\end{align*}
with the natural norm
\begin{align*}
\|(u,w)\|_{\mathcal{X}_T} &:= \sup_{0 \leq t \leq T}\|u(t)\|_{M^\frac{d}{4}(\R^d)} + \sup_{0 \leq t \leq T}\|u(t)\|_{M^p(\R^d)} + \sup_{0 < t \leq T} t^{\frac{d}{2}(\frac{1}{p}-\frac{1}{r})}\|u(t)\|_{M^r(\R^d)}\\
&\hspace{1cm} + \sup_{0 \leq t \leq T}\|w(t)\|_{M^\frac{d}{2}(\R^d)} + \sup_{0 \leq t \leq T}\|w(t)\|_{M^\frac{dp}{d-2p}(\R^d)}.
\end{align*}
Let $(u, w) \in \mathcal{X}_T$, we denote the mapping $\Phi (u,w)$ by
\begin{align*}
\Phi ((u,w))(t) 
&:= \begin{pmatrix} e^{t\Delta}u_0 - \displaystyle \int_0^t \nabla \cdot e^{(t-s)\Delta}(u\nabla E_d * w)(s) ds\\[1ex]
e^{t\Delta}w_0 + \displaystyle \int_0^t e^{(t-s)\Delta}u(s)ds \end{pmatrix}.
\end{align*}
Since the heat semigroup acts continuously on the Morrey spaces involved in the definition of $\mathcal{X}_T$, the term on the right-hand side belongs to $\mathcal{X}_T$.
The Morrey estimate \eqref{morrey:1} guarantees that $e^{t\Delta }u_0$ and $e^{t\Delta}w_0$ lie in $\mathcal{X}_T$. Moreover, invoking the conclusion of \cite[Lemma 2.2]{K1992}, one can choose $T > 0$ sufficiently small so that the term:
\begin{align*}
\sup_{0 < t \leq T} t^{\frac{d}{2}(\frac{1}{p}-\frac{1}{r})}\|e^{t\Delta}u_0\|_{M^r(\R^d)}
\end{align*}
remain arbitrarily small. Hence it suffices to estimate the Duhamel term. From \eqref{morrey:2} and \eqref{Riesz potential:2} we get
\begin{align*}
&\left\|\int_0^t \nabla \cdot e^{(t-s)\Delta}(u\nabla E_d * w)(s) ds\right\|_{M^\frac{d}{4}\cap M^p(\R^d)}\\
 &\leq 
C\int_0^t (t-s)^{-\frac{1}{2}}\|u(s)\|_{M^\frac{d}{4}\cap M^p(\R^d)} \|\nabla E_d*w(s)\|_{L^\infty(\R^d)}ds\\
&\leq C\sup_{0\leq t\leq T}\|u(t)\|_{M^\frac{d}{4}\cap M^p(\R^d)}\int_0^t (t-s)^{-\frac{1}{2}}\|w(s)\|_{M^\frac{d}{2}(\R^d)}^\nu\|w(s)\|_{M^\frac{dp}{d-2p}(\R^d)}^{1-\nu} ds\\
&\leq C\sup_{0\leq t\leq T}\|u(t)\|_{M^\frac{d}{4}\cap M^p(\R^d)}\sup_{0\leq t\leq T}\|w(t)\|_{M^\frac{d}{2}(\R^d)}^\nu\sup_{0\leq t\leq T}\|w(t)\|_{M^\frac{dp}{d-2p}(\R^d)}^{1-\nu}T^\frac{1}{2},
\end{align*}
where $\nu = \frac{3p-d}{4p-d} \in (0,1) $. The assumption $r < \frac{dp}{d-p}$ enables us to carry out an analogous argument and to obtain that
\begin{align*}
&t^{\frac{d}{2}(\frac{1}{p}-\frac{1}{r})}\left\|\int_0^t \nabla \cdot e^{(t-s)\Delta}(u\nabla E_d * w)(s) ds\right\|_{M^r(\R^d)}\\
&\leq Ct^{\frac{d}{2}(\frac{1}{p}-\frac{1}{r})} \int_0^t (t-s)^{-\frac{1}{2}} \|u(s)\|_{M^r(\R^d)} \|\nabla E_d * w(s)\|_{L^\infty(\R^d)} ds\\
&\leq C\sup_{0 < t \leq T}t^{\frac{d}{2}(\frac{1}{p}-\frac{1}{r})}\|u(t)\|_{M^r(\R^d)}\sup_{0\leq t\leq T}\|w(t)\|_{M^\frac{d}{2}(\R^d)}^\nu\sup_{0\leq t\leq T}\|w(t)\|_{M^\frac{dp}{d-2p}(\R^d)}^{1-\nu}T^{\frac{1}{2}-\frac{d}{2}(\frac{1}{p}-\frac{1}{r})}.
\end{align*}
Our assumption on the parameters is precisely tuned to the critical regime, and therefore we can repeat the argument for the second term and obtain the corresponding estimate. Therefore employing the Banach fixed point argument, we construct local-in-time solutions of \eqref{p}. 
\end{proof}

\begin{corollary}\label{local_1.5}
Let $d \ge 5$ and $p \in (\frac{d}{3}, \frac{d}{2})$, and $q \in (\frac{d}{2}, d)$. Assume that the initial data $(u_0,w_0)$ is nonnegative and satisfies
\begin{align*}
(u_0, w_0) \in (C(\R^d) \cap L^\infty(\R^d) \cap L^p(\R^d)) \times (C(\R^d) \cap L^\infty(\R^d) \cap L^q(\R^d)).
\end{align*}
Then there exist $T_\mathrm{max} \in (0,\infty]$ and a nonnegative classical solution $(u,v,w)$ of \eqref{p} such that $v(\cdot, t) = E_d * w(\cdot, t)$ for $t \in (0,T_\mathrm{max})$ and $\nabla v \in L^\infty_{loc}([0,T_\mathrm{max}); L^\infty(\R^d)^d)$, and 
\begin{equation*}
\begin{cases}
u \in C(\R^d \times [0,T_\mathrm{max})) \cap C([0,T_\mathrm{max}); L^p(\R^d)) \cap
C^{\infty}(\R^d \times (0,T_\mathrm{max})),\\
v \in C^{\infty}(\R^d \times (0,T_\mathrm{max})),\\
w \in C(\R^d \times [0,T_\mathrm{max})) \cap C([0,T_\mathrm{max}); L^q(\R^d)) \cap C^{\infty}(\R^d \times (0,T_\mathrm{max})).
\end{cases}
\end{equation*}
Moreover, if $T_\mathrm{max} < \infty$, then
\begin{align*}
\limsup_{t \to T_\mathrm{max}}(\|u(t)\|_{L^\infty(\R^d)} + \|u(t)\|_{L^p(\R^d)} + \|w(t)\|_{L^\infty(\R^d)} + \|w(t)\|_{L^q(\R^d)}) = \infty.
\end{align*}
\end{corollary}

\begin{proof}
By the argument in the proof of Proposition~\ref{local_1} combination with \cite[Lemma 2.1]{W2023_1},
we can construct a mild solution on a short time interval.
Moreover, the continuity at $t = 0$ is guaranteed by the fact that the heat semigroup preserves pointwise continuity at $t = 0$ when $(u_0, w_0) \in (C(\R^d) \cap L^\infty(\R^d))^2$ (see \cite{L1995}). The nonnegativity of the solution follows from the parabolic maximum principle together with the nonnegativity of the initial data.
\end{proof}

We next show the existence of local-in-time solutions to \eqref{p} by employing a different approach from the one used above.
The key ingredient here is Lemma~\ref{CP_1}.
More precisely, we construct an ODE-type supersolution of \eqref{p}, to which Lemma~\ref{CP_1} can be applied.

\begin{proposition}\label{local_2}
Let $d \ge 5$ and $(u_0,w_0) \in (C(\R^d) \cap L^\infty(\R^d))^2$ with the radial symmetry and nonnegativity $(u_0 \not \equiv 0).$ Then there exist $T = T(\|u_0\|_{L^\infty(\R^d)}, \|w_0\|_{L^\infty(\R^d)}) > 0$ and a local-in-time solution $(u,v,w)$ in the classical sense of \eqref{p}
\begin{equation*}
\begin{cases}
u \in C(\R^d \times [0,T]) \cap C^{\infty}(\R^d \times (0,T)),\\
v \in C^{\infty}(\R^d \times (0,T)),\\
w \in C(\R^d \times [0,T]) \cap C^{\infty}(\R^d \times (0,T)).
\end{cases}
\end{equation*}
Moreover, $u$ and $w$ are nonnegative in $\R^d \times [0,T]$.

\end{proposition}

\begin{proof}
Since $(u_0,w_0) \in L^\infty(\R^d)^2$, we can take a sufficiently large $M$ such that
\begin{align*}
\|u_0\|_{L^\infty(\R^d)} \leq M,\quad \|w_0\|_{L^\infty(\R^d)} \leq \sqrt{2M}.
\end{align*}
We denote $T = T(M)$ by
\begin{align*}
T =\dfrac{1}{\sqrt{2M}}.
\end{align*}
Let $(F,G)$ be the solution to the ODE system
\begin{equation}\label{ODE}
\begin{cases}
F_t = FG\qquad &\mathrm{for}\ t \in (0,T),\\
G_t = F\qquad &\mathrm{for}\ t \in (0,T),\\
F(0) = M, G(0) = \sqrt{2M}.
\end{cases}
\end{equation}
A direct computation shows that
\begin{align*}
F(t) = \dfrac{M}{(1-\frac{1}{2}\sqrt{2M}t)^2},\quad G(t) = \dfrac{\sqrt{2M}}{1-\frac{1}{2}\sqrt{2M}t}
\end{align*}
Moreover, throughout the interval $t \in [0,T]$, the pair $(F,G)$ enjoys the uniform bound:
\begin{align}\label{ODE:bound}
0 \leq F(t) \leq 4M,\quad 0 \leq G(t) \leq 2 \sqrt{2M}.
\end{align}
According to the method by \cite[Lemma 3.1]{W2023_2}, we approximate $u_0$ and $w_0$ by a sequence of compactly supported, radial functions $\{u_0^j\}_{j \in \N} \subset C(\R^d)$ and $\{w_0^j\}_{j \in \N} \subset C(\R^d)$ such that
\begin{align*}
u_0^j = u_0,\quad w_0^j = w_0\quad\mathrm{in}\ B_j,\quad \mathrm{supp}\ u_0^j \subset B_{j + 1}, \quad \mathrm{supp}\ w_0^j \subset B_{j + 1}
\end{align*}
and for all $j \in \N$
\begin{align*}
0 \leq u_0^j \leq u_0,\quad 0 \leq w_0^j \leq w_0\quad \mathrm{in}\ \R^d.
\end{align*}
For each $j \in \N$, Corollary~\ref{local_1.5} ensures that there exist $T_j \in (0,T]$ and a nonnegative classical solution $(u^j,v^j,w^j)$ of \eqref{p} such that $v^j(\cdot, t) = E_d * w^j(\cdot, t)$ for $t \in (0,T_j]$. Moreover if $T_j < T$, then
\begin{align}\label{ODE:1}
\limsup_{t \to T_j}(\|u^j(t)\|_{L^\infty(\R^d)} + \|w^j(t)\|_{L^\infty(\R^d)}) = \infty.
\end{align}
Here we employ the parabolic Schauder theory together with the Arzel\`{a}--Ascoli theorem to extract a convergent subsequence of $(u^j,w^j)$ and to identify its limit $(u,w)$. In particular, we also obtain the limit of $(v^j)_r$, denoted by $f$. To carry out this compactness argument, it is essential to ensure that $(u^j,w^j)$ and $(v^j)_r$ admit uniform bounds for $j \in \N$ and the existence time $T_j$ satisfy $T_j = T$.
These requirements follow from the regularity properties of $(u^j,w^j)$ and $(v^j)_r$ obtained in Corollary~\ref{local_1.5} and the initial condition. Indeed, since the solution $(F,G)$ of \eqref{ODE} serves as a supersolution to the radial formulation of the system \eqref{p}, Lemma~\ref{CP_1} yields that
\begin{align*}
u^j(r,t) \leq F(t) \leq 4M,\quad w^j(r,t) \leq G(t) \leq 2\sqrt{2M}\quad \mathrm{for\ all}\ (r,t) \in [0,\infty) \times [0,T_j).
\end{align*}
Therefore, both $u^j$ and $w^j$ are uniformly bounded in $j$, which implies that $T_j$ must be $T$ due to \eqref{ODE:1}. Moreover, thanks to the radial symmetry, we can write
\begin{align*}
(v^j)_r = -r^{1-d}\int_0^r \rho^{d-1}w^j(\rho,t) d\rho\quad \mathrm{for\ all}\ (r,t) \in (0,\infty) \times (0,T_j).
\end{align*}
This representation immediately yields that $(v^j)_r$ is locally uniformly bounded on $[0,\infty)$,  with a bound that does not depend on $j$. Consequently, the Schauder theory and the Arzel\`{a}--Ascoli theorem allow us to pass to the limit and obtain a classical radially symmetric solution $(u,v,w)$.
Here the function $v$ is defined by
\begin{align*}
v(x,t) := \int_0^{|x|} f(r,t) dr\quad \mathrm{for\ all}\ (x,t) \in \R^d \times (0,T),
\end{align*}
where $f$ denotes the limit of $(v^j)_r$. Finally, continuity at $t = 0$ follows directly from the argument in \cite[Lemma 3.1]{W2023_2}.
\end{proof}

\begin{remark}\label{local_2:rem}
According to Proposition~\ref{local_2}, one can conclude that there exits a maximal existence time $T_\mathrm{max} \in (0,\infty]$ such that either $T_\mathrm{max} = \infty$, or $T_\mathrm{max} < \infty$ and 
\begin{align*}
\limsup_{t \to T_\mathrm{max}}\{\|u(t)\|_{L^\infty(\R^d)} + \|w(t)\|_{L^\infty(\R^d)}\} = \infty.
\end{align*}
\end{remark}

We next show that any radially symmetric solution of \eqref{p} preserves the monotonicity in the radial variable 
$r > 0$, provided that the initial data are radially monotone. This property, established in Lemma~\ref{CP_4}, is pivotal in the analysis of blow-up solutions to \eqref{p}.

\begin{lemma}\label{CP_4}
Let $d \ge 5$ and $T > 0$, and let $(u_0,w_0)$ be a pair of radially nonincreasing functions.
Let $(u,v,w)$ denote the corresponding radially symmetric classical solution in $\R^d \times (0,T)$ given by Proposition~\ref{local_2}.
Then $(u,w)(\cdot, t)$ remains radially nonincreasing in $r > 0$ for each $t \in [0,T)$.

\end{lemma}

\begin{proof}
Regardless of the monotonicity of the initial data $(u_0,w_0)$, the second component $v$ is radially nonincreasing in $r > 0$. Indeed, by exploiting the radial symmetry and the nonnegativity of $w$, we integrate the second equation of \eqref{p} over $(0,r)$ for any $r \in (0,\infty)$ to obtain
\begin{align}
r^{d-1} v_r(r,t) = -\int_0^r \rho^{d-1}w(\rho,t) d\rho \leq 0 \quad \mathrm{for\ all}\ (r,t) \in (0,\infty) \times (0,T),\label{v_r:bdd}
\end{align}
which directly yields $v_r \leq 0$ for all $(r,t) \in (0,\infty) \times (0,T)$. Setting $(u_r,w_r) = (\tilde{u},\tilde{w})$ for all $(r,t) \in (0,\infty) \times (0,T)$, in view of the radial symmetry and the radial version of \eqref{p}, we derive the following parabolic system for $\tilde{u}, \tilde{w}$:
\begin{equation}\label{ine1:CP_4}
\begin{cases}
\tilde{u}_t = \tilde{u}_{rr} + \left(\frac{d-1}{r}-v_r\right)\tilde{u}_r + \left(-\frac{d-1}{r^2}-2v_{rr} - \frac{d-1}{r}v_r\right) \tilde{u}- \left(v_{rrr} + \frac{d-1}{r}v_{rr}- \frac{d-1}{r^2}v_r\right) u,\\
\tilde{w}_t = \tilde{w}_{rr} + \frac{d-1}{r}\tilde{w}_{r} - \frac{d-1}{r^2}\tilde{w} + \tilde{u}.
\end{cases}
\end{equation}
Here we fix $T_0 < T$ and we denote the function $g$ and $h$ by
\begin{align*}
g(r,t) := e^{-\alpha t}\tilde{u}(r,t) - \varepsilon e^{\beta t}\xi(r),\quad h(r,t) :=  e^{-\alpha t}\tilde{w}(r,t)- \varepsilon e^{\beta t}\xi(r)
\end{align*}
for $(r,t) \in [0,\infty) \times [0,T_0)$, where $\xi(r) = r^2 + 1$ for all $r \ge 0$ and the parameters $\alpha$ and $\beta$ are chosen such that 
\begin{align*}
\alpha = 4\sup_{(r,t) \in (0,\infty) \times (0,T_0)} w(r,t) + 2d,\quad \beta > 1.
\end{align*}
In light of Remark~\ref{local_2:rem}, the parameter $\alpha$ is finite.
Assume to the contrary that there exists $(r_0,t_0) \in [0, \infty) \times (0,T_0)$ such that
$g(r_0,t_0) = 0$. Then it necessarily follows that $g < 0$ for $(r,t) \in [0,\infty) \times [0,t_0)$ since $(r_0,t_0)$ is a point where $g$ attains zero at the earliest time. Since we obtain from the second equation of \eqref{ine1:CP_4} that for all $(r,t) \in (0,\infty) \times (0,t_0)$
\begin{align*}
&h_t - h_{rr} - \dfrac{d-1}{r}h_r + \left(\frac{d-1}{r^2} + \alpha\right)h\\
&\hspace{1cm}= g(r,t) -\alpha \varepsilon e^{\beta t}\xi(r) - \varepsilon (\beta -1)e^{\beta t}\xi(r) + 2d\varepsilon e^{\beta t} - \dfrac{d-1}{r^2}\varepsilon e^{\beta t} \xi(r)\\
&\hspace{1cm}\leq \varepsilon e^{\beta t}\xi(r) (2d - \alpha) - \varepsilon (\beta -1)e^{\beta t}\xi(r)- \dfrac{d-1}{r^2}\varepsilon e^{\beta t} \xi(r) < 0,
\end{align*}
from standard parabolic maximum principles with the initial assumption $h(r,0) < 0$ it readily follows that $h(r,t) < 0$ for all $(r,t) \in [0,\infty) \times [0,t_0)$, which implies $h(r_0,t_0) \leq 0$ due to the continuity of $h$. Now the function $g$ attains the maximum at $(r_0,t_0)$ where
\begin{align*}
g(r_0,t_0) = 0,\quad g_t(r_0,t_0) \ge 0,\quad g_r(r_0,t_0) = 0,\quad \Delta g(r_0,t_0) \leq 0.
\end{align*}
In particular, $r_0$ cannot be zero because of $\tilde{u}(0,t) =0$ for all $t \in [0,T_0)$.
This allows us to conclude from the first equation of \eqref{ine1:CP_4} that the following holds:
\begin{align*}
0&\leq g_t(r_0,t_0) - \Delta g(r_0,t_0) +v_rg_r(r_0,t_0)
\\
& = e^{-\alpha t_0} \left(-\frac{d-1}{r_0^2}-2v_{rr} - \frac{d-1}{r_0}v_r-\alpha\right) \tilde{u}(r_0,t_0)\\
&\hspace{1cm} - e^{-\alpha t_0}\left(v_{rrr} + \frac{d-1}{r_0}v_{rr}- \frac{d-1}{r^2}v_r\right) u(r_0,t_0) - \varepsilon \beta e^{\beta t_0}\xi(r_0)\\
&\hspace{1cm} + 2d\varepsilon e^{\beta t_0} - 2r_0v_r(r_0,t_0) \varepsilon e^{\beta t_0}
\end{align*}
Noticing that it follows from the second equation of \eqref{p} and the property $h(r_0,t_0) \leq 0$ that 
\begin{align}\label{ine2:CP_4}
v_{rr}(r_0,t_0) + \dfrac{d-1}{r_0}v_r(r_0,t_0) = -w(r_0,t_0) 
\end{align}
as well as
\begin{align*}
v_{rrr}(r_0,t_0) + \dfrac{d-1}{r_0} v_{rr}(r_0,t_0) - \dfrac{d-1}{r_0^2}v_r(r_0,t_0) = -h(r_0,t_0) \ge 0,
\end{align*}
thanks to the nonnegativity of $u$
we particularly see that
\begin{align*}
0&\leq g_t(r_0,t_0) - \Delta g(r_0,t_0) +v_rg_r(r_0,t_0)
\\
& \leq e^{-\alpha t_0} \left(-\frac{d-1}{r_0^2}-2v_{rr} - \frac{d-1}{r_0}v_r-\alpha\right) \tilde{u}(r_0,t_0) - \varepsilon \beta e^{\beta t_0}\xi(r_0)\\
&\hspace{1cm} + 2d\varepsilon e^{\beta t_0}- 2r_0v_r(r_0,t_0) \varepsilon e^{\beta t_0}
\\
&= e^{-\alpha t_0} \left(-\frac{d-1}{r_0^2}-v_{rr} + w -\alpha\right) \tilde{u}(r_0,t_0) - \varepsilon \beta e^{\beta t_0}\xi(r_0)\\
&\hspace{1cm}+ 2d\varepsilon e^{\beta t_0}- 2r_0v_r(r_0,t_0) \varepsilon e^{\beta t_0}.
\end{align*}
Since $g(r_0,t_0) = 0$ implies that
\begin{align*}
e^{-\alpha t_0}\tilde{u}(r_0,t_0) = \varepsilon e^{\beta t_0} >0
\end{align*}
and from \eqref{v_r:bdd} one can compute
\begin{align*}
|v_r(r,t)| \leq r \sup_{(r,t) \in (0,\infty) \times (0,T_0)} w(r,t),
\end{align*}
we can continue the computation, with the help of \eqref{ine2:CP_4} and $v_r(r_0,t_0) \leq 0$, to obtain 
\begin{align*}
0&\leq g(r_0,t_0) - \Delta g(r_0,t_0) + v_rg_r(r_0,t_0)
\\
&\leq e^{-\alpha t_0} \left(-v_{rr} + w -\alpha\right) \tilde{u}(r_0,t_0) - \varepsilon \beta e^{\beta t_0}\xi(r_0)\\
&\hspace{1cm} + 2d\varepsilon e^{\beta t_0} \xi(r_0) + 2r_0^2 \varepsilon e^{\beta t_0}\sup_{(r,t) \in (0,\infty) \times (0,T_0)} w(r,t)\\
&\leq \varepsilon e^{\beta t_0}\xi(r_0) \left(2w + \dfrac{d-1}{r_0}v_r(r_0,t_0) + 2d - \alpha\right) - \varepsilon \beta e^{\beta t_0}\xi(r_0)\\
&\hspace{1cm} + 2 \varepsilon e^{\beta t_0}\xi(r_0)\sup_{(r,t) \in (0,\infty) \times (0,T_0)} w(r,t)\\
&\leq \varepsilon e^{\beta t_0}\xi(r_0) \left(4\sup_{(r,t) \in (0,\infty) \times (0,T_0)} w(r,t) + 2d-\alpha \right) -\varepsilon e^{\beta t_0}\xi(r_0) < 0,
\end{align*}
which is impossible due to the definitions of $\alpha$ and $\beta$. Consequently by repeating the same argument and letting $\varepsilon \to 0$, thereafter $T_0 \to T$, we complete the proof.
\end{proof}

\begin{remark}
Thanks to Lemma \ref{CP_4}, if the initial data $(u_0,w_0)$ is radially nonincreasing in $r > 0$, then the solution $(u,v,w)$ retains this radial monotonicity for as long as it exists. 
\end{remark}

Finally, relying on Lemma~\ref{CP_2}, we construct a mass function that is monotone with respect to time.
The assumptions on the initial data required for the validity of the following comparison principle will be essential in the analysis carried out in Section~\ref{chap:sta}.

\begin{proposition}\label{CP_3}
Let $d \ge 5$ and $T > 0$, and let $(u,v,w)$ be a classical solution of \eqref{p} in $\mathbb{R}^d \times (0,T)$ established by Proposition~\ref{local_2}, 
with the associated mass function $(M,W)$ defined by \eqref{mass:1}, \eqref{mass:2}.
Then it holds that:
\begin{enumerate}
\item[\rm{(i)}] if the initial data $(u_0,w_0) \in (C^2(\R^d) \cap L^\infty(\R^d))^2$ satisfies 
\begin{equation}\label{sub:1}
\begin{cases}
-\Delta u_0 + \nabla \cdot (u_0  \nabla v_0) \ge 0\quad \mathrm{in}\ \R^d,\\
-\Delta w_0 - u_0 \ge 0\quad \mathrm{in}\ \R^d,
\end{cases}
\end{equation}
where $v_0$ is defined by $v_0 = (-\Delta)^{-1}w_0 + C$ for some constant $C$, then $(M,W)$ is nonincreasing in $t \in (0,T)$ for $r \in [0,\infty)$.
\item[\rm{(ii)}] if the initial data $(u_0,w_0) \in (C^2(\R^d) \cap L^\infty(\R^d))^2$ satisfies 
\begin{equation}\label{super:1}
\begin{cases}
-\Delta u_0 + \nabla \cdot ( u_0 \nabla v_0) \leq 0\quad \mathrm{in}\ \R^d,\\
-\Delta w_0 - u_0 \leq 0\quad \mathrm{in}\ \R^d,
\end{cases}
\end{equation}
where $v_0$ is defined by $v_0 = (-\Delta)^{-1}w_0 + C$ for some constant $C$, then $(M,W)$ is nondecreasing in $t \in (0,T)$ for $r \in [0,\infty)$.
\end{enumerate}
\end{proposition}

\begin{proof}
(i) The assumption implies that for all $r \in (0,\infty)$
\begin{align*}
0 \leq -M_{rr}(r,0) - \dfrac{d+1}{r}M_r(r,0) - rM_r(r,0) W(r,0) - dM(r,0)W(r,0).
\end{align*}
From the definition of the mass function, together with Proposition~\ref{local_2} and Remark~\ref{local_2:rem}, for all $T_0 < T$ one can check that
\begin{align*}
M(r,t) \leq \dfrac{1}{d}\sup_{0 < t < T_0}\|u(t)\|_{L^\infty(\R^d)},\quad W(r,t) \leq \dfrac{1}{d}\sup_{0 < t < T_0}\|w(t)\|_{L^\infty(\R^d)}
\end{align*}
for all $(r,t) \in [0,\infty) \times (0,T_0)$. Moreover the nonnegativity of $u$ provides that
\begin{align*}
rM_r + dM = u \ge 0.
\end{align*}
This relation, together with the $L^\infty$-bound on $u$, yields
\begin{align*}
rM_r(r,t) = -d M(r,t) + u(r,t) \leq 2\sup_{0 < t < T_0}\|u(t)\|_{L^\infty(\R^d)}
\end{align*}
for all $(r,t) \in [0,\infty) \times (0,T_0)$. Therefore owing to Lemma \ref{CP_2} we arrive at
\begin{align*}
M(r,t) \leq M(r,0),\quad W(r,t) \leq W(r,0)\quad \mathrm{for\ all}\ (r,t) \in [0,\infty) \times [0,T).
\end{align*}
Put $(M^\tau,W^\tau)(r,t) := (M,W)(r, t + \tau)$ for $(r,t) \in [0,\infty) \times [0,T-\tau)$. Then the above inequality make sure that
\begin{align*}
M^\tau (r,0) \leq M(r,0),\quad W^\tau(r,0) \leq W(r,0)\quad \mathrm{for\ all}\ r \in [0,\infty),
\end{align*}
thereby we use Lemma \ref{CP_2} again to obtain
\begin{align*}
M^\tau(r,t) \leq M(r,t),\quad W^\tau(r,t) \leq W(r,t)\quad \mathrm{for\ all}\ (r,t) \in [0,\infty) \times [0,T-\tau).
\end{align*}
Finally dividing both sides by $\tau$ and passing to the limit $\tau \to 0$, we get our desire. The same argument applies to the case (ii) as well.
\end{proof}

\section{Proof of Theorem~\ref{th:1} and Theorem~\ref{th:2}}\label{chap:th_1}

In what follows, we emphasize that the fundamental problem we face for the system \eqref{p} is to identify the critical quantity or, more precisely, the explicit function that determines the long-time behavior of its solutions.
We first point out that the problem is deeply related to stationary, radial, and homogeneous solutions of the system \eqref{p}. Since the system \eqref{p} is invariant by a standard scaling argument, any such stationary profile must necessarily be a multiple of $|x|^\alpha$. Motivated by this observation, we perform a formal computation to determine the constant appearing in such stationary profiles and thereby introduce a new singular stationary solution. 
\begin{proof}[Proof of Theorem~\ref{th:1}]
By a scaling argument, we may look for stationary solutions of the form
\begin{align*}
u = \frac{A}{|x|^4}, \quad v= C\log |x| + C_0,\quad w = \frac{B}{|x|^2}
\end{align*}
for some constants $A, B, C, C_0$. These functions are understood to satisfy the stationary problem \eqref{S0} associated with \eqref{p} in the sense of distributions. Throughout the following computations, all identities are to be interpreted in the distributional sense.
The third equation of \eqref{S0} implies that
\begin{align}\label{chan:1}
\dfrac{B(2d-8)}{|x|^4} = \dfrac{A}{|x|^4}
\end{align}
and we rewrite the second equation as
\begin{align}\label{chan:2}
\dfrac{C(2-d)}{|x|^2} = \dfrac{B}{|x|^2}.
\end{align}
Moreover it readily from the first equation of \eqref{S0} that
\begin{align*}
0=\dfrac{-A(4+C)(d-6)}{|x|^6},
\end{align*}
which means that for every $d \ge 7$ we must have $-A(4 + C) = 0$. Combing this fact with \eqref{chan:1} and \eqref{chan:2}, we conclude that 
\begin{align*}
A = 8(d-4)(d-2),\quad B= 4(d-2),\quad C=-4.
\end{align*}
We emphasize that the condition $d\ge7$ is not only required for the above algebraic argument,
but is also essential to ensure that the corresponding singular stationary solution
satisfies \eqref{S0} in the sense of distributions. Indeed  for the singular profile
$u_C \sim r^{-4}$ and $(v_C)_r \sim r^{-1}$, the chemotactic flux term satisfies
\begin{align*}
\nabla \cdot (u_C\nabla v_C) = r^{1-d}(r^{d-1}u_C(v_C)_r)_r \sim r^{1-d}(r^{d-6})_r.
\end{align*}
Thus the requirement $d \ge 7$ is necessary to ensure that the chemotactic flux term $u_C\nabla v_C$ is locally integrable and that the equation holds in the distributional sense.
\end{proof}

We now proceed to provide a rigorous proof of Theorem~\ref{th:2}. In order to rigorously establish Theorem~\ref{th:2}, we make use of a specially designed mass function, distinct from \eqref{mass:1} and \eqref{mass:2}, that turns out to be a fundamental tool in the forthcoming argument (see \cite{BKW2023, BKZ2018, BKP2019, BHN1994}). The simplified Keller--Segel system \eqref{KS_0} can be reduced to a scalar parabolic problem through an appropriate change of variables. In contrast, our system \eqref{p} admits no such reduction and therefore truly remains a coupled system.
Assume that $(u,v,w)$ is a radially symmetric solution classically in $\R^d \times (0,T)$, obtained in Proposition~\ref{local_1}. We introduce the mass function $(X,Y)$ by
\begin{align}
X(r,t) &:= \int_{|x| < r} u(x,t) dx = \sigma_d\int_0^r \rho^{d-1} u(\rho,t) d\rho,\label{mass function_1}\\
Y(r,t) &:= \int_{|x| < r}w(x,t) dx\label{mass function_2}
\end{align}
for all $(r,t) \in (0,\infty) \times [0,T)$, where we impose the natural conditions
\begin{align*}
X(0,t) =0,\quad Y(0,t) = 0\quad \mathrm{for}\ t \in [0,T).
\end{align*}
Then $(X,Y)$ belongs to $[C^1([0,\infty) \times [0,T)) \cap C^{2,1}((0,\infty) \times (0,T))]^2$ satisfying the system
\begin{equation}\label{radeq}
\begin{cases}
X_t = X_{rr} -\dfrac{d-1}{r}X_r + \dfrac{1}{\sigma_d}r^{1-d}X_rY\qquad &\mathrm{in}\ (0,\infty)\times (0, T),\\
Y_t = Y_{rr} - \dfrac{d-1}{r}Y_r + X \qquad &\mathrm{in}\ (0,\infty) \times (0,T).
\end{cases}
\end{equation}
In the study \cite{BKP2019}
on the simplified Keller--Segel system \eqref{KS_0}, 
the elliptic equation enables a further reduction to a single equation, and the coupling term $X_rY$ in the first equation of \eqref{radeq} is replaced by the quadratic term $X_rX$. 
The proof of the next proposition is inspired by \cite[Proposition 4.1]{BKP2019}. However, it is important to stress that our analysis is carried out for the system \eqref{radeq}.

\begin{proposition}\label{rad:prop}
Let $d \ge 5$ and $T>0$, and $p \in (\frac{d}{3}, \frac{d}{2})$ be close to $\frac{d}{3}$. 
Let $(u,v,w)$ be a solution of \eqref{p} in $\R^d \times (0,T)$ given by Proposition~\ref{local_1}, with radially symmetric initial data $(u_0,w_0) \in (M^{\frac{d}{4}}(\R^d) \cap \dot{M}^p(\R^d)) \times (M^\frac{d}{2}(\R^d) \cap M^\frac{dp}{d-2p}(\R^d))$.
Assume that the initial data $(u_0,w_0)$ satisfies for all $r > 0$
\begin{align}
X_0(r) &= \int_{|x|< r} u_0 (x) dx < \min\{Kr^{d-\frac{d}{p}}, 8(d-2) \sigma_d \varepsilon r^{d-4}\} =: b_0(r),\label{rad:asp1}\\
Y_0(r) &= \int_{|x| < r}w_0 (x) dx < \min\{Kr^{d-\frac{d}{p}+2}, 4\sigma_d \varepsilon^\prime r^{d-2}\} =: b_1(r),\label{rad:asp2}
\end{align}
where $K$ is some constant and $ 0 < \varepsilon < \varepsilon^\prime < \frac{d}{4p} < 1$ such that $\varepsilon^\prime \leq 2(d-2)\varepsilon$.
Then the mass function $(X,Y)$, defined by \eqref{mass function_1} and \eqref{mass function_2}, enjoys the same bound 
\begin{align*}
X(r,t) < b_0(r),\quad Y(r,t) < b_1(r)\quad \mathrm{for\ all}\ (r,t) \in (0,\infty) \times [0,T).
\end{align*}
\end{proposition}

\begin{proof}
Assume for contradiction that there exists a point $(r_0, t_0) \in (0,\infty) \times (0,T)$ where $X(r_0,t_0) = b_0(r_0)$. At this point, $X(r,t)$ meets the barrier function $b_0$. From the assumptions \eqref{rad:asp1} and \eqref{rad:asp2}, the two branches of the barrier functions $b_0(r)$ and $b_1(r)$ intersect at the points, respectively,
\begin{align*}
r^* = \left(\dfrac{8(d-2)\sigma_d \varepsilon}{K}\right)^\frac{1}{4-\frac{d}{p}},\quad r_* = \left(\dfrac{4\sigma_d \varepsilon^\prime}{K}\right)^\frac{1}{4-\frac{d}{p}}.
\end{align*}
Then the relationship of $\varepsilon$ and $\varepsilon^\prime$ yields $r_* \leq r^*$.
We will first show that
\begin{align}\label{radeq:3}
Y(r,t) < b_1(r)\quad \mathrm{for\ all}\ (r,t) \in (0,\infty) \times [0,t_0].
\end{align}
If it was false, then we could find the point $(r_1,t_1) \in (0,\infty) \times (0,t_0]$ where 
$Y(r_1, t_1) = b_1(r_1)$.

\medskip
\textbf{(Case 1)} We first consider the case $r_* \leq r^* \leq r_1$. Then, thanks to the initial assumption \eqref{rad:asp2},
$\overline{Y}(r,t) = r^{2-d} Y(r,t)$ reaches the value $4\sigma_d \varepsilon^\prime$ at some first touching point $(r_1,t_1)$, which is a local maximum point of $\overline{Y}(\cdot,t_1)$. Therefore we get
\begin{align}\label{radeq:1}
\overline{Y}_t(r_1,t_1) \ge 0,\quad \overline{Y}_r(r_1,t_1) = 0,\quad \overline{Y}_{rr}(r_1,t_1) \leq 0.
\end{align}
Upon combining the second equation of \eqref{radeq} and \eqref{radeq:1}, we obtain the following calculation, although it should be noted that $X(r_1,t_1) \leq b_0(r_1) = 8(d-2)\sigma_d \varepsilon r_1^{d-4}$ due to $r_* \leq r_1$.
\begin{align*}
0&\leq\overline{Y}_t(r_1,t_1)\\
 &= r_1^{2-d}\left\{(r_1^{d-2} \overline{Y}(r_1,t_1))_{rr} -\dfrac{d-1}{r_1}(r_1^{d-2}\overline{Y}(r_1,t_1))_r + X(r_1,t_1)\right\}\\
&\leq (d-2)(d-3)r_1^{-2}(4\sigma_d \varepsilon^\prime) - (d-1)(d-2)r_1^{-2}(4\sigma_d\varepsilon^\prime) + r_1^{2-d}8(d-2)\sigma_d\varepsilon r_1^{d-4}\\
&= -8(d-2)\sigma_d \varepsilon^\prime r_1^{-2} + 8(d-2)\sigma_d \varepsilon r_1^{-2}\\
&= 8(d-2)\sigma_d r_1^{-2} (\varepsilon - \varepsilon^\prime),
\end{align*}
which is impossible in view of $\varepsilon < \varepsilon^\prime$. 

\medskip
\textbf{(Case 2)} We next consider the case $r_* \leq r_1 \leq r^*$. By observing that $X(r_1,t_1) \leq b_0(r_1) = Kr_1^{d-\frac{d}{p}}$ because $r_1 \leq r^*$, while the function $Y(r,t)$ hits the barrier function $b_1(r) = 4\sigma_d \varepsilon^\prime r^{d-2}$ at the point $(r_1,t_1)$,
an entirely analogous argument applied to the function $\overline{Y}$ yields 
\begin{align}
0\leq\overline{Y}_t(r_1,t_1) &\leq -8(d-2) \sigma_d \varepsilon^\prime r_1^{-2} + r_1^{2-d}Kr_1^{d-\frac{d}{p}}\notag\\
&=r_1^{-2}\{-8(d-2)\sigma_d \varepsilon^\prime + Kr_1^{4-\frac{d}{p}}\}.\label{radeq:2}
\end{align}
Since the case $r_1 \leq r^*$ and the definition of $r_*$ make sure that
\begin{align*}
r_1^{4-\frac{d}{p}} \leq \dfrac{8(d-2)\sigma_d\varepsilon}{K},
\end{align*} 
we may apply this inequality to \eqref{radeq:2} as follows
\begin{align*}
0\leq\overline{Y}_t(r_1,t_1) &\leq r_1^{-2}\{-8(d-2)\sigma_d \varepsilon^\prime + 8(d-2)\sigma_d\varepsilon\}\\
&=8(d-2)\sigma_d \varepsilon r_1^{-2}(\varepsilon-\varepsilon^\prime),
\end{align*}
which leads to a contradiction with $\varepsilon < \varepsilon^\prime$.

\medskip
\textbf{(Case 3)} We finally consider the case $r_1 \leq r_* \leq r^*$, which indicates that the function $Y$ touches the barrier function $b_1(r) = Kr^{d-\frac{d}{p}+2}$ at $(r_1,t_1)$. Similar to the argument in (Case 1), we introduce the rescaled function
$\underline{Y}(r,t) = r^{-(d-\frac{d}{p}+2)}Y(r,t)$. Hence it follows that $\underline{Y}(r_1,t_1) = K$ as well as
\begin{align*}
\underline{Y}_t(r_1,t_1) \ge 0,\quad \underline{Y}_r(r_1,t_1) = 0,\quad \underline{Y}_{rr}(r_1,t_1) \leq 0,
\end{align*}
so that
\begin{align*}
0 \leq \underline{Y}_t(r_1,t_1) &= r_1^{-(d-\frac{d}{p}+2)}\left\{(r_1^{d-\frac{d}{p}+2} \underline{Y}(r_1,t_1))_{rr} -\dfrac{d-1}{r_1}(r_1^{d-\frac{d}{p}+2}\underline{Y}(r_1,t_1))_r + X(r_1,t_1)\right\}\\
&\leq \left(d-\dfrac{d}{p}+2\right)\left(d-\dfrac{d}{p} + 1\right)Kr_1^{-2} - (d-1)\left(d-\dfrac{d}{p}+2\right)Kr_1^{-2}\\
&\hspace{1cm} + r_1^{-(d-\frac{d}{p}+2)} K r_1^{d-\frac{d}{p}}\\
&=Kr_1^{-2}\left\{\left(d-\dfrac{d}{p}+2\right)\left(2-\dfrac{d}{p}\right) + 1\right\}.
\end{align*}
Here we readily confirm that the assumption for $p$ ensures that
\begin{align*}
\left(d-\dfrac{d}{p}+2\right)\left(2-\dfrac{d}{p}\right) + 1 < 0, 
\end{align*}
thereby this leads to a contradiction. Hence in any case, we obtain our claim \eqref{radeq:3}. We are in a position to complete the proof by carrying out an analogous argument for the function $X(r,t)$.

\medskip
\textbf{(Case 1)} We first study the case $r_*  \leq r^* \leq r_0$. This condition yields that $\overline{X}(r,t) = r^{4-d} X(r,t)$ hits the level $8(d-2)\sigma_d\varepsilon$ at the point $(r_0,t_0)$ where
\begin{align}\label{radeq:4}
\overline{X}_t(r_0,t_0) \ge 0,\quad \overline{X}_r(r_0,t_0) = 0,\quad \overline{X}_{rr}(r_0,t_0) \leq 0.
\end{align}
Using the above properties together with the first equation of \eqref{radeq} and taking into account that $Y(r_0,t_0) < b_1(r_0) = 4\sigma_d \varepsilon^\prime r_0
^{d-2}$ by $r_*\leq r_0$, we arrive at
\begin{align*}
0\leq\overline{X}_t(r_0,t_0) &= r_0^{-(4-d)}\left\{(r_0^{d-4}\overline{X}(r_0,t_0))_{rr} -\dfrac{d-1}{r_0}(r_0^{d-4}\overline{X}(r_0,t_0))_r\right\}\\
&\hspace{1cm} + r_0^{-(d-4)}\dfrac{1}{\sigma_d}r_0^{1-d}(r_0^{d-4}\overline{X}(r_0,t_0))_rY(r_0,t_0)\\
&\leq (d-4)(d-5)r_0^{-2}8(d-2)\sigma_d\varepsilon - (d-4)(d-1)r_0^{-2}8(d-2)\sigma_d\varepsilon\\
&\hspace{1cm} + 8(d-4)(d-2)\varepsilon r_0^{-d} 4\sigma_d\varepsilon^\prime r_0^{d-2}\\
&= 32(d-4)(d-2)\sigma_dr_0^{-2}\varepsilon (\varepsilon^\prime -1).
\end{align*}
However, such a situation cannot arise due to the assumption $\varepsilon^\prime < 1$.

\medskip
\textbf{(Case 2)} We next study the case $r_* \leq r_0 \leq r^*$. 
In this situation, $\underline{X}(r,t) = r^{-(d-\frac{d}{p})}X(r,t)$ hits the level $K$ at $(r_0,t_0)$ while $Y(r_0,t_0) < b_1(r_0) = 4\sigma_d \varepsilon^\prime r_0^{d-2}$. Using the characterization of $(r_0,t_0)$ and applying relation \eqref{radeq:4} with $\underline{X}$ replaced by $\overline{X}$, we obtain
\begin{align*}
0\leq\underline{X}_t(r_0,t_0) &\leq \left(d-\dfrac{d}{p}\right)K\left(-\dfrac{d}{p}\right)r_0^{-2} + 4\left(d-\dfrac{d}{p}\right)K\varepsilon^\prime r_0^{-2}\\
&= K\left(d-\dfrac{d}{p}\right)r_0^{-2}\left(4\varepsilon^\prime - \dfrac{d}{p}\right),
\end{align*}
which is inconsistent with the assumption $\varepsilon^\prime< \frac{d}{4p}$.

\medskip
\textbf{(Case 3)} We finally study the case $r_0 \leq r_* \leq r^*$. The same reasoning remains valid in (Case 2) while $Y(r_0,t_0) < b_1(r_0) = Kr_0^{d-\frac{d}{p} + 2}$. We thus gain
\begin{align*}
0\leq\underline{X}_t(r_0,t_0) &= r_0^{-(d-\frac{d}{p})}\left\{(r_0^{d-\frac{d}{p}}\underline{X}(r_0,t_0))_{rr} -\dfrac{d-1}{r_0}(r_0^{d-\frac{d}{p}}\underline{X}(r_0,t_0))_r\right\}\\
&\hspace{1cm} + r_0^{-(d-\frac{d}{p})}\dfrac{1}{\sigma_d}r_0^{1-d}(r_0^{d-\frac{d}{p}}\underline{X}(r_0,t_0))_rY(r_0,t_0)\\
&\leq \left(d-\dfrac{d}{p}\right)\left(d-\dfrac{d}{p}-1\right)r_0^{-2}K - (d-1)\left(d-\dfrac{d}{p}\right)r_0^{-2}K\\
&\hspace{1cm} + \dfrac{1}{\sigma_d}\left(d - \dfrac{d}{p}\right)r_0^{-d}K^2r_0^{d-\frac{d}{p} + 2}\\
&= K\left(d-\dfrac{d}{p}\right)r_0^{-2}\left(-\dfrac{d}{p} + \dfrac{K}{\sigma_d}r_0^{4-\frac{d}{p}}\right).
\end{align*}
Using the fact $r_0 \leq r_*$, we obtain
\begin{align*}
r_0^{4-\frac{d}{p}} \leq \dfrac{4\sigma_d\varepsilon^\prime}{K}.
\end{align*}
According to the assumption $\varepsilon^\prime < \frac{d}{4p}$, however, this would lead to the impossible conclusion.

Hence, in every possible situation we reach a contradiction, and our desired conclusion follows by repeating the same line of reasoning for $Y(r,t)$.
\end{proof}

\begin{proof}[Proof of Theorem~\ref{th:2}]
Let $T > 0$ and $(u,v,w)$ be the classical and radially symmetric solution in $\R^d \times (0,T)$ obtained in Proposition~\ref{local_1}. Since $u_0 \in \dot{M}^p(\R^d)$ and $w_0 \in M^\frac{dp}{d-2p}(\R^d)$, we obtain the estimates
\begin{align*}
X_0(r) \leq r^{d-\frac{d}{p}}\|u_0\|_{M^p(\R^d)},\quad Y_0(r) \leq r^{d-\frac{d}{p} + 2}\|w_0\|_{M^\frac{dp}{d-2p}(\R^d)}\quad \mathrm{for\ all}\ r\in (0,\infty).
\end{align*}
Choose a constant $K$ sufficiently large so that
\begin{align*}
K > \max\{\|u_0\|_{M^p(\R^d)},\|w_0\|_{M^\frac{dp}{d-2p}(\R^d)}\}.
\end{align*}
The assumptions \eqref{th:2,asup1} and \eqref{th:2,asup2} allow us to take $\varepsilon < \varepsilon^\prime < \frac{3}{4}$ such that
\begin{align*}
X_0(r) < 8(d-2)\sigma_d \varepsilon r^{d-4},\quad Y_0(r) < 4\sigma\varepsilon^\prime r^{d-2}\quad \mathrm{for\ all}\ r\in (0,\infty).
\end{align*}
We also note that, since $p \in (\frac{d}{3}, \frac{d}{2})$ is taken sufficiently close to $\frac{d}{3}$, the quantity $\frac{d}{4p}$ is sufficiently close to $\frac{3}{4}$. Therefore, one can select parameters $\varepsilon$ and $\varepsilon^\prime$ such that $\varepsilon < \varepsilon^\prime < \frac{d}{4p}$ with both $\varepsilon$ and $\varepsilon^\prime$ arbitrarily close to $\frac{3}{4}$, and satisfying the natural requirement
$\varepsilon^\prime \leq 2(d-2)\varepsilon$. 
Hence, invoking Proposition~\ref{rad:prop}, we deduce that
\begin{align*}
X(r,t) < b_0(r),\quad Y(r,t) < b_1(r)\quad \mathrm{for\ all}\ (r,t) \in (0,\infty) \times [0,T),
\end{align*}
where the barrier functions $b_0(r)$ and $b_1(r)$ are exactly those defined in \eqref{rad:asp1} and \eqref{rad:asp2} respectively. From the definition of $(X,Y)$, we deduce
\begin{align*}
\sup_{0 < t < T}(\|u(t)\|_{M^\frac{d}{4}(\R^d)} + \|u(t)\|_{M^p(\R^d)} + \|w(t)\|_{M^\frac{d}{2}(\R^d)} + \|w(t)\|_{M^\frac{dp}{d-2p}(\R^d)}) < \infty.
\end{align*}
Finally, arguing in the same manner as in the proof of Proposition~\ref{local_1}, we obtain the differential inequality
\begin{align*}
\|u(t)\|_{M^r(\R^d)} &\leq Ct^{-\frac{d}{2}(\frac{1}{p}-\frac{1}{r})}\|u_0\|_{M^p(\R^d)} +
C\int_0^t (t-s)^{-\frac{1}{2}} \|u(s)\|_{M^r(\R^d)} \|\nabla E_d * w(s)\|_{L^\infty(\R^d)} ds\\
&\leq Ct^{-\frac{d}{2}(\frac{1}{p}-\frac{1}{r})}\|u_0\|_{M^p(\R^d)}\\
&\hspace{0.4cm} + C\sup_{0 < t < T}\|w(t)\|_{M^\frac{d}{2}(\R^d)} \sup_{0 < t < T}\|w(t)\|_{M^\frac{dp}{d-2p}(\R^d)}
\int_0^r (t-s)^{-\frac{1}{2}}\|u(s)\|_{M^r(\R^d)} ds,
\end{align*}
from which the singular Gronwall inequality (see  \cite[Theorem II.3.3.1]{A1995}
and \cite[Lemma 1.2.9]{CD2000}) yields the desired estimate:
\begin{align*}
\sup_{0 < t < T}t^{\frac{d}{2}(\frac{1}{p}-\frac{1}{r})}\|u(t)\|_{M^r(\R^d)} < \infty.
\end{align*}
Combining the above estimates and invoking a standard continuation argument, we conclude that the solution $(u,v,w)$ of \eqref{p} exists globally in time, meaning $T= \infty$. Furthermore, such a global-in-time solution satisfies the upper bounds \eqref{th:2:upper} as a direct consequence of Proposition~\ref{rad:prop}. 
\end{proof}

\section{The existence of blow up points at infinite time via Comparison Principles}\label{chap:blowup}

In this section, we develop the arguments required to establish convergence toward the stationary state. 
For the simplified Keller--Segel system \eqref{KS_0}, the second equation reduces to the Poisson equation, and by utilizing the radial symmetry and the corresponding mass function, the system can essentially be rewritten as a single equation. This structure allows one to exploit suitable test functions to derive superlinear inequalities for certain energy type quantities, leading to Kaplan-type blow-up arguments (see \cite{K1963, W2023_2, N2021}). As a consequence, it can be shown that blow-up may occur only at the origin, which in turn enables the analysis of convergence to stationary solutions away from the singularity. In contrast, the parabolic system \eqref{mass:3} derived from \eqref{p} is genuinely coupled and cannot be reduced to a single parabolic equation. Hence the aforementioned techniques are not directly applicable. Therefore, in this section we rely on the comparison principle developed in Chapter~\ref{chap:2} to locate the possible blow-up point and to prepare the analysis required for convergence to stationary solutions. Here we make use of the monotonicity of $(M,W)$ for $r > 0$.

\begin{proposition}\label{non-inc}
Let $(u_0,w_0)\in (C(\R^d) \cap L^\infty(\R^d))^2$ be radially nonincreasing and $T > 0$. Assume that $(u,v,w)$ is the solution of \eqref{p} in $\R^d \times (0,T)$ given by Proposition~\ref{local_2}.
Then the associated mass function $(M,W) \in [C([0,\infty) \times [0,T)) \cap C^{2,1}([0,\infty) \times (0,T))]^2 $ is nonincreasing in $r > 0$ for each $t \in [0,T)$.

\end{proposition}

\begin{proof}
Noticing that the solution $(u,v,w)$ of \eqref{p} corresponding to the radially nonincreasing initial data $(u_0,w_0)$ is also radially nonincreasing for each $t \in [0,T)$ thanks to Lemma~\ref{CP_4}
and that the mass function $(M,W)$ is given by \eqref{mass:1} and \eqref{mass:2}, we deduce that for all $(r,t) \in (0,\infty) \times (0,T)$
\begin{align*}
M_r(r,t) &= -dr^{-1-d}\int_0^r \rho^{d-1}u(\rho,t)d\rho + \dfrac{u(r,t)}{r}\\
&\leq -d u(r,t) r^{-1-d} \int_0^r \rho^{d-1} d\rho + \dfrac{u(r,t)}{r}\\
&= 0.
\end{align*}
The same applies to $W(r,t)$ as well. Hence we complete the proof.
\end{proof}

The next proposition serves as the core of this section.
We derive a sufficient condition under which solutions blow up in finite time.
A key feature of the proof is the construction of a new blowing up subsolution for \eqref{mass:3}.

\begin{proposition}\label{finite}
Let $d \ge 5$ and let $(u,v,w)$ be a solution of \eqref{p} in $\R^d \times (0,T_\mathrm{max})$ given by Proposition~\ref{local_2}, corresponding to radially symmetric initial data $(u_0,w_0) \in (C(\R^d) \cap L^\infty(\R^d))^2$.
Let $(M,W) \in [C([0,\infty) \times [0,T_\mathrm{max})) \cap C^{2,1}([0,\infty) \times (0,T_\mathrm{max}))]^2
$ denote the mass functions defined by \eqref{mass:1} and \eqref{mass:2}.
Fix constants $A,B,k$ such that 
\begin{align}\label{asp:finite1}
4 < B,\quad \dfrac{A}{B} > \dfrac{2d-3 + \sqrt{(2d-3)^2 + 32}}{2},
\end{align}
where $A$ is chosen sufficiently large and $B$ is sufficiently close to $4$, as well as
\begin{align}\label{asp:finite2}
\dfrac{2B(d+2)}{A-B} < k < \min\left\{\dfrac{A}{B} -2(d-2), \dfrac{d}{2}B\right\}
\end{align}
with $k$ chosen sufficiently close to the lower bound $\frac{2B(d+2)}{A-B}$.
Suppose that there exist $r_0 > 0$ and $t_0 \in (0,T_\mathrm{max})$ such that
\begin{align*}
\dfrac{A}{r^4 + k} \leq M(r,t_0),\quad \dfrac{B}{r^2 + k} \leq W(r,t_0)\quad \mathrm{for\ all}\ r \in (0,r_0)
\end{align*}
as well as
\begin{align*}
\dfrac{A}{r_0^4} \leq M(r_0,t),\quad \dfrac{B}{r_0^2} \leq W(r_0,t)\quad \mathrm{for\ all}\ t \in [t_0,T_\mathrm{max}).
\end{align*}
Then $T_\mathrm{max} < \infty$.
\end{proposition}

\begin{remark}\label{finite:rem2}
We may choose $A$ sufficiently large and $B$ sufficiently close to $4$ in the proof of Proposition \ref{finite}. This guarantees \eqref{asp:finite1} and also makes the interval in \eqref{asp:finite2} nonempty. In particular, since
\begin{align*}
\dfrac{2B(d+2)}{A-B} = \dfrac{2(d+2)}{\frac{A}{B}-1} \to 0\quad \mathrm{as}\ \frac{A}{B} \to \infty,
\end{align*}
under this choice, 
the constant $k$ appearing in \eqref{asp:finite2} can be taken sufficiently small.
\end{remark}

\begin{proof}
Assume contrary to the claim that $T_\mathrm{max} = \infty$.
We construct a blowing up subsolution for the system \eqref{mass:3} to employ Lemma~\ref{CP_2}.
Let us denote $f$ and $g$ by
\begin{align*}
f(r,t) = \dfrac{A}{r^4+k(1-(t-t_0))^2} ,\quad g(r,t) = \dfrac{B}{r^2+k(1-(t-t_0))}
\end{align*}
for $(r,t) \in [0,\infty) \times [t_0, t_0 + 1)$. In other words, the assumption means that 
\begin{align*}
f(r,t_0) \leq M(r,t_0),\quad g(r,t_0) \leq W(r,t_0)\quad\mathrm{for\ all}\ r \in (0,r_0]
\end{align*}
and it specifies the boundary condition at $r=r_0$ for all $t \in [t_0,t_0+1)$
\begin{align*}
f(r_0,t) &= \dfrac{A}{r_0^4+k(1-(t-t_0))^2} \leq \dfrac{A}{r_0^4} \leq M(r_0,t)
\end{align*}
and likewise,
\begin{align*}
g(r_0,t) = \dfrac{B}{r_0^2+k(1-(t-t_0))} \leq \dfrac{B}{r_0^2}\leq W(r_0,t).
\end{align*}
To check the subsolution property for $f$ and $g$, setting $\tau = 1-(t-t_0)$, we calculate that for all $(r,t) \in (0,\infty) \times (t_0,t_0+1)$
\begin{align*}
f_t(r,t) = \dfrac{2Ak\tau}{(r^4+k\tau^2)^2}
\end{align*}
and 
\begin{align*}
f_r(r,t) = \dfrac{-4Ar^3}{(r^4+k\tau^2)^2}
\end{align*}
and consequently
\begin{align*}
f_{rr}(r,t) = \dfrac{4Ar^2(5r^4-3k\tau^2)}{(r^4+k\tau^2)^3}.
\end{align*}
Proceeding in the same way, we obtain for $(r,t) \in (0,\infty) \times (t_0,t_0+1)$
\begin{align*}
g_t(r,t) = \dfrac{Bk}{(r^2 + k\tau)^2}
\end{align*}
and 
\begin{align*}
g_r(r,t) = \dfrac{-2Br}{(r^2 + k\tau)^2}
\end{align*}
as well as
\begin{align*}
g_{rr}(r,t) = \dfrac{2B(3r^2-k\tau)}{(r^2+k\tau)^3}.
\end{align*}
Having performed these computations, we are in a position to check that $(f,g)$ is a subsolution of \eqref{mass:3}. Therefore we have
\begin{align*}
&\left(f_t - f_{rr} - \dfrac{d+1}{r}f_r - rf_r g- dfg\right) \times (r^4 + k\tau^2)^3\times (r^2 + k\tau)\\
&= 2Ak\tau(r^4 + k\tau^2)(r^2+k\tau) -20Ar^6(r^2 + k\tau) + 12Ak\tau^2r^2(r^2+k\tau)\\
&\hspace{1cm} + 4A(d+1)r^2(r^4 + k\tau^2)(r^2 + k\tau) + 4ABr^4(r^4 + k\tau^2)-dAB(r^4 + k\tau^2)^2\\
&= \{4A(d-4) + 4AB - dAB\} r^8 + \{4A(d-4)k\tau + 2Ak\tau\}r^6\\
&\hspace{1cm} + \{2Ak^2\tau^2 + 4Ak(d+4)\tau^2 + 4ABk\tau^2 - 2dABk\tau^2\}r^4\\
&\hspace{1cm} + \{2Ak^2\tau^3 + 4Ak^2(d+4)\tau^3\}r^2 +2Ak^3\tau^4 - dABk^2\tau^4\\
&=A\{4(d-4) + (4 - d)B\} r^8 + 2Ak\tau \{2(d-4) + 1\}r^6\\
&\hspace{1cm} + 2Ak\tau \{k\tau + 2(d+4)\tau + 2B\tau - dB\tau\}r^4\\
&\hspace{1cm} + 2Ak^2\tau^3 \{1 + 2(d+4)\}r^2 +Ak^2\tau^4(2k-dB).
\end{align*}
If $r \ge1$, then the term $r^8$ dominates all other terms because $k$ is sufficiently small as noted in Remark~\ref{finite:rem2}.
Consequently, in order for the right-hand side above to be negative, it is necessary that the coefficient $4(d-4) + (4 - d)B$ must be negative, that is $4 < B$. On the other hand, if $r < 1$, then the constant term dominates all other terms. Therefore, it suffices to require $k < \frac{d}{2}B$. We next verify that the subsolution property for $g$ holds, so that
\begin{align*}
&\left(g_t - g_{rr} - \dfrac{d+1}{r}g_r - f\right) \times (r^2 + k\tau)^3 \times (r^4 + k\tau^2)\\
&= Bk(r^2 + k\tau)(r^4 + k\tau^2) -6Br^2(r^4 + k\tau^2)\\
&\hspace{1cm} + 2Bk\tau (r^4 + k\tau^2) + 2B(d+1)(r^2 + k\tau)(r^4 + k\tau^2)-A(r^2 + k\tau)^3\\
&= \{Bk -6B + 2B(d+1)-A\}r^6 + \{Bk^2\tau + 2Bk\tau + 2B(d+1)k\tau - 3Ak\tau\}r^4\\
&\hspace{1cm} + \{Bk^2\tau^2 -6Bk\tau^2 + 2B(d+1)k\tau^2 - 3Ak^2\tau^2\}r^2\\
&\hspace{1cm} + Bk^3\tau^3 + 2Bk^2\tau^3 + 2B(d+1)k^2\tau^3 -Ak^3\tau^3\\
&= \{B(k + 2(d-2)) -A\}r^6 + k\tau\{Bk + 2B + 2B(d+1) - 3A\}r^4\\
&\hspace{1cm} + k\tau\{Bk\tau - 6B\tau + 2B(d+1)\tau - 3Ak\tau\}r^2\\
&\hspace{1cm} + k^2\tau^3\{Bk + 2B + 2B(d+1) -Ak\}.
\end{align*}
When $ r\ge 1$, the smallness of $k$ (Remark~\ref{finite:rem2}) ensures that the term $r^6$ dominates the expression.
Therefore, achieving negativity on the right-hand side requires
\begin{align}\label{finite:ine1}
k < \dfrac{A}{B} -2(d-2)
\end{align}
and thus necessarily,
\begin{align*}
\dfrac{A}{B} > 2(d-2).
\end{align*}
Conversely, when $r < 1$, the constant part plays the decisive role, so the condition inequality 
\begin{align}\label{finite:ine2}
k > \dfrac{2B(d+2)}{A-B}
\end{align}
is sufficient. Collecting \eqref{finite:ine1} and \eqref{finite:ine2}, and the requirement $k < \frac{d}{2}B$, we see that 
$k$ must satisfy
\begin{align*}
\dfrac{2B(d+2)}{A-B} < k < \min\left\{\dfrac{A}{B} -2(d-2), \dfrac{d}{2}B\right\}.
\end{align*}
However, as noted in Remark~\ref{finite:rem2}, we choose $\frac{A}{B}$ sufficiently large and hence $k$ sufficiently small.
Therefore the minimum on the right-hand side is attained at $\frac{d}{2}B$. Moreover, this inequality is consistent under the second condition in assumption~\eqref{asp:finite1} and the condition $d\ge5$, as can be easily checked.
To apply Lemma~\ref{CP_2}, it suffices to verify that the functions
$f$, $g$, $M$, and $W$ satisfy the assumptions of Lemma~\ref{CP_2}. The required boundedness properties on any interval $T_0\in(t_0,t_0+1)$
follow directly from the definitions of $f$ and $g$, and the $L^\infty$-boundedness
of the solution $(u,v,w)$ provided by Proposition~\ref{local_2}.
Furthermore it readily follows from the definition of $M(r,t)$ and nonnegativity of $u$ that for all $(r,t) \in (0,\infty) \times (0,T_\mathrm{max})$
\begin{align*}
 r^{d-1}(rM_r(r,t) + dM(r,t)) = (r^dM(r,t))_r = r^{d-1}u(r,t) \ge 0.
\end{align*}
As a consequence, we obtain the subsolution $(f,g)$ of \eqref{mass:3}. From the above considerations, using Lemma~\ref{CP_2}, we conclude that
\begin{align*}
f(r,t) \leq M(r,t) ,\quad g(r,t) \leq W(r,t)\quad \mathrm{for\ all}\ (r,t) \in (0,r_0) \times [t_0,t_0 + 1),
\end{align*}
which implies that blow-up must occur earlier than $t = t_0 + 1$. This contradicts the assumption $T_\mathrm{max} = \infty$, thereby we conclude that $T_\mathrm{max} < \infty$.
\end{proof}

Building on Proposition~\ref{finite}, we are now in a position to establish the next proposition.

\begin{proposition}\label{blowup point}
Let $(u_0,w_0)\in (C(\R^d) \cap L^\infty(\R^d))^2$ be radially nonincreasing. Assume that $(u,v,w)$ is a classical solution of \eqref{p} in $\R^d \times (0,T_\mathrm{max})$ given by Proposition~\ref{local_2}.
Let $(M,W) \in [C([0,\infty) \times [0,T_\mathrm{max})) \cap C^{2,1}([0,\infty) \times (0,T_\mathrm{max}))]^2
$ denote the mass functions defined by \eqref{mass:1} and \eqref{mass:2}.
If $T_\mathrm{max} = \infty$ and the pointwise limit $\lim_{t \to \infty}M(r,t)$ and $\lim_{t \to \infty}W(r,t)$ exist for any $r \ge 0$, then it holds that
\begin{align*}
\lim_{t \to \infty}M(r,t) < \infty,\quad \lim_{t \to \infty}W(r,t) < \infty\quad\mathrm{for}\ r > 0.
\end{align*}

\end{proposition}

\begin{remark}
Proposition~\ref{blowup point} states that if blow-up occurs at infinite time under the assumption that $(M,W)$ is monotone in $t > 0$, then the blow-up point is only the origin. Therefore, one may assume that the solution $(M,W)$ remains bounded away from the origin.
\end{remark}

\begin{proof}
Since the initial data $(u_0, w_0)$ is radially nonincreasing, Proposition~\ref{non-inc} implies that $(M,W)$ is also nonincreasing in $r > 0$ for eqch $t \in (0,T_\mathrm{max})$.
Suppose, for the sake of contradiction, that there exists $r_* > 0$ such that
\begin{align*}
\lim_{t \to \infty} M(r_*,t) = \infty.
\end{align*}
The assumption of monotonicity with respect to $r > 0$ implies that
\begin{align}\label{ine1:blowup point}
\lim_{t \to \infty}M(r,t) = \infty\quad \mathrm{for}\ r \in [0,r_*].
\end{align}
Then it holds that
\begin{align*}
\lim_{t \to \infty}W(r,t) = \infty\quad \mathrm{for}\ r \in [0,r_*).
\end{align*}
In fact, if it is false, then there exists $r_0 \in [0,r_*)$ such that
\begin{align*}
\sup_{t > 0} W(r_0,t) < \infty.
\end{align*}
Since $W(\cdot,t)$ is nonincreasing with respect to $r>0$, it follows that
\begin{align*}
\sup_{t > 0} W(r,t) < \infty\quad \mathrm{for\ all}\ r \in [r_0,r_*).
\end{align*}
Since the radial function $v$ solves the equation $-\Delta v =w$, the definition of $W$ and the uniform bound for $W$ above ensure that
\begin{align*}
|v_r(r,t)| = |-rW(r,t)| \leq r_* \sup_{t > 0}|W(r,t)|\quad \mathrm{for\ all}\ r\in [r_0,r_*)
\end{align*}
and hence $|v_r|$ is uniformly bounded on $[r_0,r_*)\times(0,\infty)$.
Standard interior parabolic estimates (see \cite[Proposition 3.7]{S2025})
allow us to obtain a uniform-in-time bound for $u(r,t)$ on $(r_0,r_*)$. Consequently, $M(r,t)$ is uniformly bounded in time for any $r\in(r_0,r_*)$,
which contradicts our assumption \eqref{ine1:blowup point}. 
Therefore for any fixed $r_0 \in (0,r_*)$ sufficiently small and any $K_* >0$, there exists $t_0 \in (0,\infty)$ such that
\begin{align}\label{ine2:blowup point}
M(r_0, t) \ge K_*,\quad W(r_0, t) \ge K_*\quad \mathrm{for\ all}\ t \ge t_0.
\end{align}
By the arbitrariness of $K_*$, we may choose $K_*$ such that
\begin{align*}
K_* = \max\left\{\dfrac{A}{k}, \dfrac{B}{k}\right\},
\end{align*}
where the constants $A$ and $B$ , and $k$ are taken as in Proposition~\ref{finite}. Here in view of Remark~\ref{finite:rem2}, we may choose $A$ sufficiently large and $B$ sufficiently close to $4$.
In this situation, the constant $k$ can be taken sufficiently small such that $k \leq r_0^4$
, thereby the inequality \eqref{ine2:blowup point} immediately gives that 
\begin{align*}
M(r_0, t) \ge \dfrac{A}{r_0^4},\quad W(r_0, t) \ge \dfrac{B}{r_0^2}\quad \mathrm{for\ all}\ t \ge t_0.
\end{align*}
Furthermore, for each $r \in (0,r_0)$ we observe that
\begin{align*}
\dfrac{A}{r^4 + k} \leq K_* \leq M(r_0,t_0)
\end{align*}
and using the monotonicity $M_r \leq 0$, we infer
\begin{align*}
\dfrac{A}{r^4 + k} \leq M(r,t_0)\quad \mathrm{for\ all}\ r \in (0,r_0).
\end{align*}
A completely analogous conclusion applies to $W$. Therefore Proposition~\ref{finite} forces $T_\mathrm{max} < \infty$, which leads to a contradiction.
\end{proof}

\section{Stationary Problems and Convergence Properties of solutions}\label{chap:sta}

In this section, we investigate the properties of stationary solutions to the original system \eqref{p}, and subsequently analyze the convergence of solutions to the system \eqref{mass:3} for the mass function toward such stationary states.
We begin by showing that the stationary solution to \eqref{p} necessarily coincides with a solution of the fourth-order elliptic problem \eqref{S2}. This reveals a precise correspondence between steady states of \eqref{p} and those of the biharmonic Gel'fand problem.

\begin{proposition}\label{stationaly sol}
Let $d \ge 5$ and $\alpha > 0$, and $\beta_0(\alpha) \in [-4d\alpha^\frac{1}{2}, 0)$.
Then the following two assertions are equivalent.
\begin{enumerate}
\item[\rm{(i)}] the radially symmetric pair $(u,v)$, with $u > 0$ everywhere,  
solves \eqref{S1} in the classical sense in $\R^d$ and satisfies the following conditions:
\begin{align*}
(u,v)(0) = (\alpha, \log \alpha),\quad \Delta v(0) = \beta_0(\alpha),\quad v^\prime(0) = (\Delta v)^\prime (0) = 0.
\end{align*}
\item[\rm{(ii)}] the functions $u$ and $v$ are given by $e^\phi$ and $\phi$ respectively, where the radially symmetric function $\phi \in C^4(\R^d)$ solves \eqref{S2} in $\R^d$ with $\phi(0) = \log \alpha$ and $\Delta \phi(0) =\beta_0(\alpha)$.
\end{enumerate}
Moreover, any solution of \eqref{S1} satisfying the conditions above is radially decreasing and also enjoys the property
\begin{align}\label{asp:stationaly sol}
\nabla \cdot (u\nabla v) < 0\quad\mathrm{in}\ \R^d,
\end{align}
and moreover setting $w = -\Delta v$ yields a positive, radially decreasing function $w$ with
\begin{align}\label{asp:stationaly sol_2}
\lim_{|x| \to \infty} w(x) = 0.
\end{align}
\end{proposition}

\begin{remark}\label{remark:sub or super}
The fact that the sign of the advective term is completely determined by the inequality \eqref{asp:stationaly sol} is highly significant. This observation plays a crucial role in the subsequent analysis. Indeed, once the sign is identified, it enables us to construct suitable functions satisfying subsolution or supersolution properties for \eqref{mass:3}.
\end{remark}

\begin{proof}
We now assume (i). Multiplying the radial version of the first equation in \eqref{S1} by $r^{d-1}$ and integrating it over $(0,r)$, we have
\begin{align*}
0&= \int_0^r \rho^{d-1}\rho^{1-d} (\rho^{d-1}(u_r-uv_r))_r d\rho\\
&= \int_0^r (\rho^{d-1}(u_r-uv_r))_r d\rho\\
&= r^{d-1}(u_r - uv_r) (r).
\end{align*}
Positivity of $u$ ensures that $\log u - v = C_0$ for all $r > 0$ with some constant $C_0$. Thanks to the smoothness  of $(u,v)$ and the conditions at $r = 0$, we find
\begin{align*}
\log \alpha = \lim_{r \to 0} \log u(r) = C_0 + \lim_{r \to 0} v(r) = C_0 + \log \alpha,
\end{align*}
so that the constant $C_0$ is necessarily zero. Consequently we obtain that $u = e^v$ for $r \ge 0$.
Substituting this relation into the second equation of \eqref{S1}, we arrive at the fourth–order equation
\begin{align*}
\Delta^2 v = e^v\quad \mathrm{in}\ [0,\infty).
\end{align*}
Noticing that the radial solution $\phi$ of \eqref{S2} is uniquely determined by the prescribed conditions at the origin, by the regularity of $(u,v)$ and the same conditions at the origin, we show that $v = \phi$, and hence (ii) follows. 
Conversely, under assumption (ii), the associated function $(u,v)$ is radially symmetric with $u > 0$ and one can directly check that it fulfills all requirements of (i). Indeed, this follows immediately from a straightforward computation. Thus what is left is to show \eqref{asp:stationaly sol} and the properties of $w$ given by $-\Delta v$. Now, let us write the problem \eqref{S2} with $u = e^\phi$ and $v = \phi$ as a system in the following  way:
\begin{equation}\label{S22_1}
\begin{cases}
-\Delta w = u = e^\phi\qquad \mathrm{for}\ r \in [0,\infty),\\
-\Delta v = w\qquad \mathrm{for}\ r \in [0,\infty).
\end{cases}
\end{equation}
Since $u = e^\phi$, it holds that $-\Delta w > 0$. The radial symmetry then enforces the monotonicity
$w_r < 0$ for $ r > 0$, while regularity at the origin guarantees
$w_r(0) = (-\Delta v)^\prime (0) = 0$. The assumption $\beta_0(\alpha) \in [-4d\alpha^\frac{1}{2}, 0)$ implies $w(0) = -\beta_0(\alpha) > 0$. Hence $w$ is positive near the origin, and in order to conclude the global positivity of $w$ it suffices to show 
\eqref{asp:stationaly sol_2}. For the Gel'fand problem for the biharmonic operator, namely $\Delta^2 \phi = e^\phi$,
under the assumption on the parameters $\alpha$ and $\beta_0(\alpha)$, it is shown by \cite[Theorem 2]{AGG2006} that
\begin{align}\label{stationary sol:1}
\lim_{r \to \infty} (\phi(r) + 4\log r) = \log 8(d-4)(d-2).
\end{align}
Moreover, thanks to the refined analysis developed in \cite[Theorem 3.1, Theorem 4.4]{G2014}, the following result holds:
\begin{align}\label{stationary sol:2}
\lim_{r \to \infty} \left(\Delta \phi (r) + \dfrac{4(d-2)}{r^2}\right) = 0.
\end{align}
This asymptotic behavior immediately implies that $w (r) \to 0$ as $ r \to \infty$ and hence all required qualitative properties of $w$ follow. To compete the proof, we assume that
\begin{align*}
-\Delta u = -\Delta e^\phi \leq 0.
\end{align*}
Then $\phi = v$ must be nondecreasing on $(0,\infty)$. Actually, the condition $-\Delta e^\phi \leq 0$ yields
 $(e^\phi)_r \ge 0$ due to $\phi_r(0) = 0$ and hence $\phi_r \ge 0$ for $r \in[0,\infty)$. On the other hand, the second equation of \eqref{S22_1} together with positivity of $w$ implies that $v$ is strictly decreasing in $r > 0$. This is incompatible with the monotonicity conclusion derived above for $\phi = v$, and hence we necessarily have $-\Delta e^\phi > 0$ for all $r \ge 0$. Consequently we obtain that
\begin{align*}
0 < -\Delta e^\phi = -\Delta u = -\nabla \cdot (u\nabla v)\quad \mathrm{in}\ \R^d.
\end{align*}
Furthermore since $v=\phi$ is strictly decreasing in $r > 0$, one can check that
\begin{align*}
u_r(r) = (e^{\phi(r)})_r = \phi_r(r) e^{\phi(r)} < 0\quad\mathrm{for\ all}\ r >0,
\end{align*}
that is $u$ is also strictly decreasing in $r > 0$. This completes the proof.
\end{proof}

As stated in Remark~\ref{remark:sub or super}, we can obtain the following claim.

\begin{proposition}\label{sub and super}
Let $(u_{\lambda_0}, 
w_{\lambda_1})$ defined by
\begin{align*}
(u_{\lambda_0}, w_{\lambda_1}) := (\lambda_0 u,
\lambda_1 w),
\end{align*}
where $(u,v,w)$ is radially symmetric solution of \eqref{S0} with $u > 0$ in $\R^d$ and the properties
\begin{align*}
(u,v)(0) = (\alpha, \log \alpha),\quad \Delta v(0) = \beta_0(\alpha),\quad v^\prime(0) = (\Delta v)^\prime (0) = 0
\end{align*}
for $\alpha > 0$ and $\beta_0(\alpha) \in [-4d\alpha^\frac{1}{2}, 0)$, and $\lambda_0$ and $\lambda_1$ are some constants. Then the following three properties hold:
\begin{enumerate}
\item[\rm{(i)}] if $\lambda_0 = \lambda_1 > 1$, then $(u_{\lambda_0}, 
w_{\lambda_1})$ satisfies \eqref{super:1}.
\item[\rm{(ii)}] if $\lambda_0 > \lambda_1 \ge 1$, then $(u_{\lambda_0}, 
w_{\lambda_1})$ satisfies \eqref{super:1}.
\item[\rm{(iii)}] if $\lambda_0 = \lambda_1 \in (0,1]$, then $(u_{\lambda_0},
w_{\lambda_1})$ satisfies \eqref{sub:1}.
\end{enumerate}
\end{proposition}

\begin{proof}
(i) Setting $\lambda = \lambda_1 = \lambda_2$ and then multiplying both equations in \eqref{S0} by $\lambda$, we get
\begin{equation*}
\begin{cases}
0 = \Delta u_{\lambda} - \nabla \cdot (u_{\lambda}\nabla v)\qquad &\mathrm{in}\ \R^d,\\
0 = \Delta w_{\lambda} + u_{\lambda} \qquad &\mathrm{in}\ \R^d,
\end{cases}
\end{equation*}
where $v = (-\Delta)^{-1} w + C$ for some constant $C$. Since $(u,v,w)$ with the appropriate assumptions at the origin solves \eqref{S0} classically in $\R^d$, from Proposition~\ref{stationaly sol} it follows that
\begin{align*}
\nabla \cdot (u\nabla v) < 0\quad \mathrm{in}\ \R^d.
\end{align*}
The assumption $\lambda > 1$ ensures that
\begin{align*}
\nabla \cdot (u_\lambda \nabla v_\lambda) &= \lambda \nabla \cdot (u_\lambda \nabla v) < \nabla \cdot (u_\lambda \nabla v),
\end{align*}
where $v_\lambda$ is defined by $v_\lambda = (-\Delta)^{-1} w_\lambda + \lambda C$.
Hence this immediately gives the claim.

\medskip
(ii) If we multiply the first equation by $\lambda_0$ as well as the third equations of \eqref{S0} by $\lambda_1$, then we obtain
\begin{equation*}
\begin{cases}
0 = \Delta u_{\lambda_0} - \nabla \cdot (u_{\lambda_0}\nabla v)\qquad &\mathrm{in}\ \R^d,\\
0 = \Delta w_{\lambda_1} + \lambda_1 u \qquad &\mathrm{in}\ \R^d.
\end{cases}
\end{equation*}
A completely analogous argument to that in (i) provides the desired inequality for the first equation above, on the other hand, the second equation needs to be handled with care. Since we have the assumption $\lambda_0 > \lambda_1$, it readily follows from positivity of $u$ that
\begin{align*}
0 = -\Delta w_{\lambda_1} - \lambda_1 u > -\Delta w_{\lambda_1} - \dfrac{\lambda_0}{\lambda_1}\lambda_1 u = -\Delta w_{\lambda_1} - u_{\lambda_0},
\end{align*}
which completes the proof. 

\medskip
(iii) The argument is exactly the same as in (i), except that the inequality is reversed.
\end{proof}

The above arguments enable us to carry out the convergence analysis toward stationary solutions of \eqref{mass:3}.

\begin{proposition}\label{convergence}
Let $(u,v,w)$ be a global-in-time solution of \eqref{p} given by Proposition~\ref{local_2}, corresponding to radially symmetric initial data $(u_0,w_0) \in (C(\R^d) \cap L^\infty(\R^d))^2$.
Then the following two assertions hold.

\begin{enumerate}
\item[\rm{(i)}] if $(M,W)$, defined by \eqref{mass:1} and \eqref{mass:2}, is nonincreasing in $t > 0$ for each $r \in [0,\infty)$, then  by setting
\begin{align*}
M_\infty(r) = \lim_{t \to \infty} M(r,t),\quad W_\infty(r) = \lim_{t \to \infty} W(r,t)\quad \mathrm{for}\ r \in [0,\infty),
\end{align*}
$(M_\infty, W_\infty) \in [C^2([0,\infty)) \cap L^\infty(0,\infty)]^2
$ satisfies 
$(r^d M_\infty)_r \ge 0$ and $(r^d W_\infty)_r \ge 0$ for $r \in [0,\infty)$ as well as
\begin{equation}\label{convergence:1}
\begin{cases}
0 = (M_\infty)_{rr} + \dfrac{d+1}{r}(M_\infty)_r + r(M_\infty)_r W_\infty + d M_\infty W_\infty,\quad &r > 0,\\
0 = (W_\infty)_{rr} + \dfrac{d+1}{r}(W_\infty)_r + M_\infty,\quad &r > 0.
\end{cases}
\end{equation}
\item[\rm{(ii)}] if $(u_0,w_0)$ is additionally nonincreasing in $r > 0$ and $(M,W)$ is nondecreasing in $t > 0$ for each $r \in [0,\infty)$, then  by setting
\begin{align*}
M_\infty(r) = \lim_{t \to \infty} M(r,t),\quad W_\infty(r) = \lim_{t \to \infty} W(r,t)\quad \mathrm{for}\ r \in [0,\infty),
\end{align*}
$(M_\infty, W_\infty) \in [C^2((0,\infty)) \cap L^\infty(r_0,\infty)]^2
$ for any $r_0 > 0$ satisfies $(r^d M_\infty)_r \ge 0$ and $(r^d W_\infty)_r \ge 0$ for $r \in (0,\infty)$ as well as
\eqref{convergence:1}.
\end{enumerate}
\end{proposition}

\begin{proof}
(i) Since $M_t(r,t) \leq 0$ and $W_t(r,t) \leq 0$ for all $(r,t) \in [0,\infty) \times (0,\infty)$, we have
\begin{align*}
M_\infty(r) \leq M(r,t) \leq M(r,0)\quad \mathrm{for\ all}\ (r,t) \in [0,\infty) \times (0,\infty)
\end{align*}
and 
\begin{align*}
W_\infty(r) \leq W(r,t) \leq W(r,0)\quad \mathrm{for\ all}\ (r,t) \in [0,\infty) \times (0,\infty).
\end{align*}
The definition of $(M,W)$, together with the assumption $(u_0,w_0) \in L^\infty(\R^d)^2$, ensures that 
\begin{align}\label{convergence:3}
(r^d M(r,t))_r = r^{d-1}u(r,t) \ge 0,\quad (r^d W(r,t))_r = r^{d-1}w(r,t) \ge 0
\end{align}
for all $(r,t) \in (0,\infty) \times [0,\infty)$
as well as 
\begin{align}\label{convergence:2}
\sup_{t > 0}\sup_{r \ge 0}M(r,t) \leq \dfrac{\|u_0\|_{L^\infty(\R^d)}}{d},\quad \sup_{t > 0}\sup_{r \ge 0}W(r,t) \leq \dfrac{\|w_0\|_{L^\infty(\R^d)}}{d}.
\end{align}
This, in turn, implies the validity of $(M_\infty, W_\infty) \in L^\infty(0,\infty)^2$. 
Consequently the parabolic Schauder theory \cite{LSU1968} and the Arzel\`{a}--Ascoli theorem yield that $(M_\infty, W_\infty) \in [C^2([0,\infty)) \cap L^\infty(0,\infty)]^2
$ and 
\begin{align*}
(M,W)(\cdot, t) \to (M_\infty, W_\infty)\quad \mathrm{in}\ C_{loc}^2([0,\infty))\quad \mathrm{as}\ t \to \infty.
\end{align*}
Thanks to the the monotonicity in $t$ and the fact \eqref{convergence:2}, we derive that for any $T >0$ and $r \ge 0$
\begin{align*}
\int_0^T |M_t(r,t)| dt &= -\int_0^T M_t(r,t) dt\\
&\leq \dfrac{\|u_0\|_{L^\infty(\R^d)}}{d}
\end{align*}
and the same argument applies to $W$. This implies that
\begin{align*}
(M_t, W_t)(\cdot, t) \to (0, 0)\quad \mathrm{in}\ C_{loc}([0,\infty))\quad \mathrm{as}\ t \to \infty.
\end{align*}
As a result, taking $t \to \infty$, we confirm that $(M_\infty, W_\infty) \in [C^2([0,\infty)) \cap L^\infty(0,\infty)]^2
$ satisfies the system \eqref{convergence:1} and we further note that \eqref{convergence:3} guarantees property 
\begin{align*}
(r^d M_\infty(r))_r \ge 0,\quad (r^d W_\infty(r))_r \ge 0\quad\mathrm{for}\ r \in (0,\infty).
\end{align*}

\medskip
(ii) Since the initial data $(u_0,w_0)$ is radially nonincreasing, from Proposition~\ref{non-inc} we have
\begin{align*}
M_r(r,t) \leq 0,\quad W_r(r,t) \leq 0\quad\mathrm{for\ all}\ (r,t) \in (0,\infty) \times (0,\infty).
\end{align*}
Therefore Proposition~\ref{blowup point} ensures that the blow-up point can only be the origin even if the solution blows up at infinite time, so that we would infer that
for all $0 < r_0 < r_1$ there exists $C$ depending on $r_0$ and $r_1$ such that
\begin{align*}
M(r,t) < C,\quad W(r,t) < C\quad\mathrm{for\ all}\ (r,t) \in (r_0, r_1) \times (0,\infty).
\end{align*}
Accordingly, the analysis in (i) remains applicable away from the origin, leading to the claimed result.
\end{proof}

The above proposition ensures the existence of the stationary solution $(M_\infty,W_\infty)$ to the system \eqref{mass:3}.
However, this does not immediately yield a stationary solution $(u_\infty, v_\infty, w_\infty)$ of the original system \eqref{p}. Further work is required to construct $(u_\infty, v_\infty, w_\infty)$ from $(M_\infty,W_\infty)$ established by Proposition~\ref{convergence}.

\begin{lemma}\label{r^dM}
Let $d \ge 5$. Assume the same hypotheses as in Proposition~\ref{convergence} either (i) or (ii).
Then it holds that
\begin{align}\label{r^dM:1}
\lim_{r \to 0}r^dM_\infty(r) = 0,\quad \lim_{r \to 0}r^dW_\infty(r) = 0.
\end{align}

\begin{proof}
Assuming that the limiting functions $M_\infty$ and $W_\infty$ lie in $L^\infty(0,\infty)$, we readily obtain \eqref{r^dM:1}, thereby it is sufficient to show our claim in the case $(M_\infty, W_\infty) \to \infty$ as $r \to 0$ due to Proposition~\ref{blowup point}. Owing to Proposition~\ref{convergence},  $r^dM_\infty$ and $r^dW_\infty$ are nondecreasing in $r > 0$, namely we can find some constants $c_0$ and $c_1$ such that
\begin{align*}
\lim_{r \to 0} r^d M_\infty(r) = c_0,\quad \lim_{r \to 0} r^dW_\infty(r) = c_1.
\end{align*}
We remark that, since $M(r,t)$ and $W(r,t)$ are nonnegative functions,
their limit functions $M_\infty$ and $W_\infty$ are also nonnegative. Consequently, the constants $c_0$ and $c_1$ satisfy $c_0\ge0$ and $c_1\ge0$.
In what follows, we will show that both constants must in fact be equal to zero.
Assume for the contradiction that $c_1 > 0$. By setting 
\begin{align}\label{r^dM:3}
\Phi (r) = r(M_\infty)_r + dM_\infty, 
\end{align}
it follows from the first equation of \eqref{convergence:1} that
\begin{align}
\Phi_r (r) + rW_\infty \Phi(r) = 0\quad \mathrm{for}\ r > 0.\label{r^dM:32}
\end{align}
In particular, since we assume the initial data $(u_0,w_0)$ as in Proposition~\ref{convergence} (ii), Proposition~\ref{non-inc} enables us to obtain the monotonicity $M_r(r,t) \leq 0$ and $W_r(r,t) \leq 0$ for $(r,t) \in (0,\infty) \times (0,\infty)$. Consequently, the limit functions $M_\infty$ and $W_\infty$ obtained via locally uniform convergence in Proposition~\ref{convergence}
inherit the same monotonicity, that is, $(M_\infty)_r \leq 0$ and $(W_\infty)_r \leq 0$ for any $r > 0$. Hence we get from the equality \eqref{r^dM:3}
\begin{align}\label{r^dM:4}
\limsup_{r \to 0} r^d\Phi(r) \leq d\limsup_{r \to 0}r^dM_\infty(r) = dc_0.
\end{align}
By employing Proposition~\ref{convergence}, providing the monotonicity of $r^dW$, it holds that
\begin{align*}
r^d W_\infty(r) \ge c_1\quad \mathrm{for}\ r > 0,
\end{align*}
and by solving the equation \eqref{r^dM:32} we arrive, for $0<r<r_0$, at
\begin{align}
\Phi (r) &= \Phi(r_0) \exp \left(\int_r^{r_0} \rho W_\infty(\rho) d\rho\right)\notag\\
&\ge \Phi(r_0) \exp \left(c_1\int_r^{r_0} \rho^{1-d} d\rho\right)\notag\\
&= \Phi(r_0) \exp \left(\dfrac{c_1}{d-2}(r^{2-d}-r_0^{2-d})\right).\label{r^dM:7}
\end{align}
Multiplying the above inequality by $r^d$, we have 
\begin{align*}
r^d\Phi(r) \ge \Phi(r_0) \exp \left(\dfrac{c_1}{d-2}(r^{2-d}-r_0^{2-d}) + d\log r\right).
\end{align*}
From this expression it is immediate to see that
\begin{align*}
\lim_{r \to 0}r^d \Phi (r) = \infty.
\end{align*}
This clearly contradicts \eqref{r^dM:4} and hence we conclude that $c_1 = 0$. We will next  show $c_0 = 0$ by contradiction. Now, since the second equation of \eqref{convergence:1} together with the nonnegativity of
$M_\infty$ yields that, for all $r>0$,
\begin{align*}
(r^{d+1}(W_\infty)_r)_r = -r^{d+1} M_\infty \leq 0
\end{align*}
and $(W_\infty)_r \leq 0$ for $r > 0$, there exists $c_2 \leq 0$ such that
\begin{align*}
\lim_{r \to 0}r^{d+1}(W_\infty)_r(r) = c_2.
\end{align*}
If the constant $c_2$ is strictly negative, then from the monotonicity of $r^{d+1}(W_\infty)_r$ we see that
\begin{align*}
(W_\infty)_r(r) \leq c_2r^{-1-d}\quad \mathrm{for}\ r > 0
\end{align*}
and with this bound at hand, for $0 < r < r_0$ we estimate 
\begin{align*}
W_\infty(r) &= -\int_r^{r_0}(W_\infty)_r(\rho) d\rho + W_\infty(r_0)\\
&\ge -c_2\int_r^{r_0} \rho^{-1-d} d\rho\\
&= -\dfrac{c_2}{d}r^{-d} + \dfrac{c_2}{d}r_0^{-d}.
\end{align*}
Consequently upon multiplying the above inequality by $r^d$ and taking the limit $r \to 0$, we derive
\begin{align*}
\lim_{r \to 0}r^d W_\infty(r) \ge -\dfrac{c_2}{d} > 0.
\end{align*}
This is clearly impossible, since we already know that $\lim_{r \to 0}r^dW_\infty(r) =0$. We thus obtain 
\begin{align*}
\lim_{r \to 0}r^{d+1}(W_\infty)_r(r) = 0.
\end{align*}
We integrate the second equation of \eqref{convergence:1} over $(\varepsilon, r)$ for $0 < \varepsilon < r$ to proceed
\begin{align*}
r^{d+1}(W_\infty)_r(r) - \varepsilon^{d+1}(W_\infty)_r(\varepsilon) = -\int_\varepsilon^r \rho^{d+1}M_\infty(\rho) d\rho
\end{align*}
and letting $\varepsilon \to 0$ we further deduce that
\begin{align}\label{r^dM:5}
r^{d+1}(W_\infty)_r (r) = -\int_0^r \rho^{d+1}M_\infty(\rho) d\rho.
\end{align}
Because we are arguing by contradiction under the assumption $c_0 > 0$, the monotonicity of $r^d M_\infty$ yields that 
\begin{align}\label{r^dM:6}
r^dM_\infty(r) \ge c_0>0\quad \mathrm{for}\ r > 0.
\end{align}
Combing \eqref{r^dM:5} and \eqref{r^dM:6}, we first observe that
\begin{align*}
r^{d+1}(W_\infty)_r(r) &\leq -c_0 \int_0^r \rho d\rho= -\dfrac{c_0}{2}r^2.
\end{align*}
This differential inequality immediately provides a useful control on $(W_\infty)_r$, so that 
we integrate the above relation over the interval $(r,r_0)$ for $0 < r < r_0$, which implies
\begin{align*}
W_\infty(r) &= -\int_r^{r_0} (W_\infty)_r(\rho)d\rho + W_\infty(r_0)\\
&\ge \dfrac{c_0}{2}\int_r^{r_0}\rho^{1-d} d\rho\\
&= \dfrac{c_0}{2(d-2)}(r^{2-d} - r_0^{2-d}).
\end{align*}
Consequently, we arrive at the lower bound
\begin{align*}
\liminf_{r \to 0} r^{d-2} W_\infty(r) \ge \dfrac{c_0}{2(d-2)} > 0,
\end{align*}
namely this asymptotic behavior shows that
for any $0 < K < \frac{c_0}{2(d-2)}$ we could find $r_0 > 0$ satisfying 
\begin{align*}
W_\infty(r) \ge Kr^{2-d}\quad\mathrm{for}\ r < r_0.
\end{align*}
By making a slight modification of the argument used in \eqref{r^dM:7}, taking into account the inequality established above, we are able to get a lower estimate for $r^d\Phi$ as follows:
\begin{align*}
r^d\Phi (r) &= r^d\Phi(r_0) \exp \left(\int_r^{r_0} \rho W_\infty(\rho) d\rho\right)\\
&\ge r^d\Phi(r_0) \exp \left(K\int_r^{r_0} \rho^{3-d} d\rho\right)\\
&= \Phi(r_0) \exp \left(\dfrac{K}{d-4}(r^{4-d}-r_0^{4-d}) + d\log r\right).
\end{align*}
However, this lower estimate is incompatible with \eqref{r^dM:4} in the limit $r \to 0$ and hence this completes the proof.
\end{proof}
\end{lemma}

\begin{proposition}\label{u,v,w}
Let $d \ge 5$. Assume the same hypotheses as in Proposition~\ref{convergence} either (i) or (ii). For $r>0$, define
\begin{align*}
u_\infty &:= dM_\infty + r(M_\infty)_r,\\
w_\infty &:= dW_\infty + r(W_\infty)_r,\\
v_\infty &:= c_0(r_0) -\int_{r_0}^r \rho W_\infty(\rho) d\rho,
\end{align*}
where $r_0>0$ is arbitrary and $c_0(r_0)$ is some constant.
Then it holds that for any $r > 0$
\begin{align*}
M_\infty(r) &= r^{-d}\int_0^r \rho^{d-1} u_\infty(\rho) d\rho,\\
W_\infty(r) &= r^{-d}\int_0^r \rho^{d-1}w_\infty(\rho) d\rho.
\end{align*}
Moreover the triple $(u_\infty, v_\infty, w_\infty)$ belongs to $C^2((0,\infty))\times C^4((0,\infty)) \times C^2((0,\infty))$ and satisfies the system \eqref{S0} classically for all $r > 0$, and moreover $u_\infty$ and $w_\infty$ are nonnegative for $r > 0$.
\end{proposition}

\begin{proof}
The definition of $u_\infty$ implies that
\begin{align*}
r^{d-1}u_\infty = r^d(M_\infty)_r + dr^{d-1}M_\infty = (r^dM_\infty)_r.
\end{align*}
Integrating the equation above by $(\varepsilon, r)$ for any $\varepsilon > 0$, we have
\begin{align*}
\int_\varepsilon^r \rho^{d-1}u_\infty d\rho = r^dM_\infty(r) - \varepsilon^d M_\infty(\varepsilon).
\end{align*}
Since Lemma~\ref{r^dM} allows the passage to the limit $\varepsilon \to 0$, we arrive at
\begin{align*}
M_\infty(r) = r^{-d}\int_0^r \rho^{d-1}u_\infty(\rho) d\rho.
\end{align*}
An entirely analogous argument yields the corresponding identity for $W_\infty$. Since the regularity of $(M_\infty, W_\infty)$ established in Proposition~\ref{convergence} can be iteratively improved to $C^\infty((0,\infty))$ by exploiting the structure of the system \eqref{convergence:1}, this clearly guarantees
\begin{align*}
(u_\infty, v_\infty, w_\infty) \in C^2((0,\infty))\times C^4((0,\infty)) \times C^2((0,\infty))
\end{align*}
and moreover, a straightforward computation shows that the functions satisfy the system \eqref{S0} in $r > 0$. The nonnegativity of $u_\infty$ and $w_\infty$ is guaranteed by the property, which is proved in Proposition~\ref{convergence}, 
\begin{align*}
(r^dM_\infty)_r \ge 0,\quad (r^d W_\infty)_r \ge 0\quad \mathrm{for}\  r > 0.
\end{align*}
\end{proof}

\section{Proof of Theorem~\ref{th:3}}\label{chap:th_2}

In what follows, we give a rigorous proof of Theorem~\ref{th:3}.
Although the previous section established the stationary solution $(M_\infty, W_\infty)$ for the mass function formulation \eqref{mass:3} and then constructed the associated stationary solution $(u_\infty, v_\infty, w_\infty)$ of the original system \eqref{p}, the analysis required for Theorem~\ref{th:3} depends essentially on the full stationary structure of the original system \eqref{p}. Our approach highlights that the key ingredient lies in the asymptotic analysis of the original stationary solution $(u_\infty, v_\infty, w_\infty)$. This perspective appears to be novel and provides a deeper structural understanding of the system \eqref{p}.

\begin{proposition}\label{limit}
Let $r_0 \ge 0$ and $d \ge 5$. Assume that $(u,v,w)$ is a radially symmetric solution, with $u$ and $w$ nonnegative in $r > r_0$, to 
\begin{equation}\label{limit:0}
\begin{cases}
0 = \Delta u - \nabla \cdot (u\nabla v)\qquad &\mathrm{for}\ r > r_0,\\
-\Delta v = w\qquad &\mathrm{for}\ r > r_0,\\
0 = \Delta w + u \qquad &\mathrm{for}\ r > r_0.
\end{cases}
\end{equation}
If it holds that
\begin{align}\label{limit:1}
\liminf_{r \to \infty} (v(r) + 4\log r) \not= -\infty,
\end{align}
then the following properties hold
\begin{align}
\liminf_{r \to \infty}r^4 u(r) &\leq 8(d-4)(d-2),\label{limit:2}\\
\liminf_{r \to \infty}r^2w(r) &\leq 4(d-2).\label{limit:3}
\end{align}
\end{proposition}

\begin{remark}
The assumption \eqref{limit:1} in Proposition~\ref{limit} is reasonable in view of the conclusion obtained in \cite{AGG2006, BFFG2012}. More precisely, in \cite[Theorem 2]{AGG2006} and \cite[Theorem 4]{BFFG2012} it is shown that when $d \ge 5$, then for any $\alpha >0$
there is a unique parameter $\beta_0 \in [-4d\alpha^\frac{1}{2}, 0)$ such that the corresponding solution $\phi_{\alpha, \beta_0} \in C^4((0,\infty))$ of \eqref{S2}
satisfies the asymptotic profile
\begin{align*}
\lim_{r \to \infty} (\phi_{\alpha, \beta_0} (r) + 4\log r) = \log 8(d-4)(d-2).
\end{align*}
On the other hand, it is also proved that if $\beta < \beta_0$, then the solution $\phi_{\alpha, \beta} \in C^4((0,\infty))$ of \eqref{S2} enjoys the upper estimate
\begin{align*}
\phi_{\alpha, \beta} (r) \leq \alpha - \dfrac{\beta_0-\beta}{2d}r^2\quad\mathrm{for\ all}\ r \in [0,\infty).
\end{align*}
Hence the validity of the asymptotic profile depends crucially on the choice of $\beta$. Consequently, our assumption expresses precisely the condition under which the expected asymptotic behavior holds.
\end{remark}

\begin{proof}
We first show \eqref{limit:3} by contradiction, namely we assume that there exist $r_1 > r_0$ and $K_0 > 4(d-2)$ such that
\begin{align}\label{limit:4}
w(r) \ge \dfrac{K_0}{r^2}\quad \mathrm{for}\ r \ge r_1.
\end{align}
Integrating the radial form of second equation of \eqref{limit:0} over $(r_1, r)$ yields that
\begin{align*}
r^{d-1}v_r(r) = r_1^{d-1} v_r(r_1) - \int_{r_1}^r \rho^{d-1}w(\rho)d\rho
\end{align*}
and using the lower bound \eqref{limit:4} we further estimate
\begin{align*}
\int_{r_1}^r \rho^{d-1}w(\rho) d\rho \ge K_0\int_{r_1}^r \rho^{d-3}.
\end{align*}
Since the right-hand side of the above inequality diverges to $\infty$ as $r \to \infty$, we can choose $r_2 \ge r_1$ so that
\begin{align}\label{limit:5}
r^{d-1}v_r(r) < 0\quad \mathrm{for\ all}\ r \ge r_2.
\end{align}
Performing a further integration of the second equation of \eqref{limit:0} over $(r_2, r)$ and making use of the properties \eqref{limit:4} and \eqref{limit:5}, we obtain for $r \ge r_2$
\begin{align*}
r^{d-1}v_r(r) &= r_2^{d-1}v_r(r_2) -\int_{r_2}^r \rho^{d-1}w(\rho)d\rho\\
& < -\int_{r_2}^r \rho^{d-1}w(\rho)d\rho\\
&\leq \dfrac{K_0}{d-2}r_2^{d-2} - \dfrac{K_0}{d-2}r^{d-2}.
\end{align*}
In turn, this inequality yields that for $r \ge r_2$
\begin{align*}
v(r) &\leq v(r_2) + \dfrac{K_0}{d-2}r_2^{d-2}\int_{r_2}^r \rho^{1-d} d\rho - \dfrac{K_0}{d-2}\int_{r_2}^r \dfrac{1}{\rho} d\rho\\
&= v(r_2) + \dfrac{K_0}{(d-2)^2}r_2^{d-2}(r_2^{2-d} - r^{2-d}) -\dfrac{K_0}{d-2}(\log r- \log r_2)\\
&\leq -\dfrac{K_0}{d-2}\log r + v(r_2) + \dfrac{K_0}{(d-2)^2} + \dfrac{K_0}{d-2}\log r_2,
\end{align*}
and hence we get the upper bound
\begin{align}\label{limit:6}
v(r) + \dfrac{K_0}{d-2}\log r \leq v(r_2) + \dfrac{K_0}{(d-2)^2} + \dfrac{K_0}{d-2}\log r_2.
\end{align}
Now we note
\begin{align*}
v(r) + 4\log r= v(r) + \dfrac{K_0}{d-2}\log r + \left(4-\dfrac{K_0}{d-2}\right)\log r.
\end{align*}
Noting that the first term on the right-hand side of the above equality remains bounded by \eqref{limit:6} , whereas the second one diverges to $-\infty$ as $r \to \infty$ because of $4-\frac{K_0}{d-2} < 0$, we obtain a contradiction with the assumed asymptotic behavior \eqref{limit:1}, which implies \eqref{limit:3}. Building upon this result, we now proceed to prove 
\eqref{limit:2} by contradiction, thereby we assume that there exist $r_3 > r_0$ and $K_1 > 8(d-4)(d-2)$ such that
\begin{align}\label{limit:7}
u(r) \ge \dfrac{K_1}{r^4}\quad\mathrm{for\ all}\ r \ge r_3.
\end{align}
Following the same line of reasoning as in the proof \eqref{limit:5} of the monotonicity for $v_r$, it readily follows from the combination of the third equation of \eqref{limit:0} and \eqref{limit:7} as well as $d \ge 5$ that we can find $r_4 \ge r_3$ such that
\begin{align}\label{limit:8}
r^{d-1}w_r(r) < 0\quad\mathrm{for}\ r \ge r_4.
\end{align}
We thus integrate the third equation of \eqref{limit:0} over $(r_4, r)$ and utilize \eqref{limit:7} to deduce
\begin{align*}
r^{d-1}w_r(r) &= r_4^{d-1}w_r(r_4) - \int_{r_4}^r \rho^{d-1}u(\rho) d\rho\\
&< -\int_{r_4}^r \rho^{d-1}u(\rho) d\rho\\
&\leq -K_1\int_{r_4}^r \rho^{d-5} d\rho\\
&= -\dfrac{K_1}{d-4}r^{d-4} + \dfrac{K_1}{d-4}r_4^{d-4}
\end{align*}
and hence
\begin{align*}
w_r(r) \leq -\dfrac{K_1}{d-4}r^{-3} + \dfrac{K_1}{d-4}r_4^{d-4}r^{1-d}\quad \mathrm{for}\ r > r_4.
\end{align*}
Here the inequality \eqref{limit:8} ensures that $w_r$ remains negative for $r \ge r_4$. Using this along with the earlier result \eqref{limit:3}, one readily verifies the property $\lim_{r \to \infty} w(r) = 0$.  Hence, once we fix 
$r \in (r_4, r_*)$ with $r_4 < r_*$, using the preceding inequality leads to
\begin{align*}
w(r) &= -\int_r^{r_*} w_r(\rho)d\rho + w(r_*)\\
&\ge \dfrac{K_1}{d-4}\int_r^{r_*} \rho^{-3}d\rho - \dfrac{K_1}{d-4}r_4^{d-4}\int_r^{r_*}\rho^{1-d}d\rho + w(r_*)\\
&= \dfrac{K_1}{2(d-4)}(r^{-2}- r_*^{-2}) + \dfrac{K_1}{(d-2)(d-4)}r_4^{d-4}(r_*^{2-d} - r^{2-d}) + w(r_*)\\
&\ge \dfrac{K_1}{2(d-4)}r^{-2}- \dfrac{K_1}{(d-2)(d-4)}r_4^{d-4}r^{2-d} - \dfrac{K_1}{2(d-4)}r_*^{-2} + w(r_*).
\end{align*}
Letting $r_* \to \infty$, since we already know  $\lim_{r \to \infty} w(r) = 0$, 
we infer that for every $r \in (r_4, \infty)$
\begin{align*}
w(r) \ge \dfrac{K_1}{2(d-4)}r^{-2} - \dfrac{K_1}{(d-2)(d-4)}r_4^{d-4}r^{2-d}.
\end{align*}
In particular, the choice of $K_1$ guarantees that $\frac{K_1}{2(d-2)} > 4(d-2)$, and therefore the lower bound above is incompatible with the asymptotic behavior \eqref{limit:3}. This contradiction completes the argument.
\end{proof}

\begin{proof}[Proof of Theorem~\ref{th:3}]
(i) Let $T > 0$ and $(u,v,w)$ be a radially symmetric solution of \eqref{p} in $\R^d  \times (0,T)$, where $u$ and $w$ are nonnegative.
Let $(M,W)$ denote the associated mass function defined by \eqref{mass:1} and \eqref{mass:2}.
Owing to these definitions, the identity holds that 
\begin{align*}
r M_r + dM = r^{1-d} (r^d M)_r = u \ge 0\quad\mathrm{for\ all}\ (r,t) \in (0,\infty) \times (0,T).
\end{align*}
Since $\lambda \in (0,1]$,  Proposition~\ref{stationaly sol} ensures that $(\lambda e^\phi, \lambda\phi, \lambda (-\Delta)\phi)$ solves the system \eqref{S0} classically in $\R^d$ with the appropriate conditions at the origin. Hence employing Proposition~\ref{sub and super} (iii) together with Proposition~\ref{CP_3} (i), we deduce that $(M,W)$ is nonincreasing in $t \in (0,T)$ for each $r \in [0,\infty)$.
Here we note that $(\lambda e^\phi, \lambda(-\Delta)\phi)$ is radially nonincreasing (see Proposition~\ref{stationaly sol}) and is regular at the origin, and also decays to zero at infinity.
In particular, the functions belong to $(C(\R^d)\cap L^\infty(\R^d))^2$.
Therefore, when $(\lambda e^\phi, \lambda(-\Delta)\phi)$ is taken as initial data,
Proposition~\ref{local_2} ensures the existence of a corresponding solution to
\eqref{p}.
Moreover since $(\lambda e^\phi, \lambda (-\Delta)\phi)$ satisfies \eqref{sub:1}, one can check that the mass functions generated by $(\lambda e^\phi, \lambda (-\Delta)\phi)$ inherit the supersolution property of the system \eqref{mass:3}. Therefore, since all the assumptions of Lemma~\ref{CP_2} are satisfied for solutions
generated by Proposition~\ref{local_2}, we may invoke Lemma~\ref{CP_2} to obtain,
for all $(r,t)\in[0,\infty)\times[0,T)$,
\begin{align*}
M(r,t) &\leq \lambda r^{-d}\int_0^r \rho^{d-1} e^\phi d\rho,\\
W(r,t) &\leq \lambda r^{-d} \int_0^r \rho^{d-1}(-\Delta)\phi d\rho.
\end{align*}
These inequalities and the definition of $(M,W)$ yield that for all $r >0$
\begin{align*}
\int_0^r \rho^{d-1}u(\rho,t) d\rho &\leq \lambda \int_0^r \rho^{d-1} e^\phi d\rho,\\
\int_0^r \rho^{d-1}w(\rho,t) d\rho &\leq \lambda  \int_0^r \rho^{d-1}(-\Delta)\phi d\rho.
\end{align*}
Consequently we conclude that the solution $(u,v,w)$ exists globally in time and enjoys the upper bound \eqref{th3:bound}.

\medskip
(ii) Let $T > 0$ and $(u,v,w)$ be a radially symmetric solution of \eqref{p} in $\R^d  \times (0,T)$, where $u$ and $w$ are nonnegative.
Seeing $(\lambda e^\phi, \lambda (-\Delta)\phi)$ as the initial data, and noting that $(\lambda e^\phi, \lambda (-\Delta)\phi)$ belongs to $(C(\R^d) \cap L^\infty(\R^d))^2$, Proposition~\ref{local_2} ensures that there exist $T_\lambda \in (0,\infty]$ and a classical solution $(u_\lambda, v_\lambda, w_\lambda)$ of \eqref{p} in $\R^d \times (0,T_\lambda)$. We denote by $(M_\lambda, W_\lambda)$ the mass function corresponding to the solution $(u_\lambda, v_\lambda, w_\lambda)$. Now since we have the ordering assumption on the initial data, this implies that
\begin{align*}
M_\lambda(r,0) \leq M (r,0),\quad W_\lambda (r,0) \leq W (r,0)\quad\mathrm{for\ all}\ r \in [0,\infty).
\end{align*}
Therefore, we may readily apply Lemma~\ref{CP_2}, which yields
\begin{align}\label{th3:ine1}
M_\lambda(r,t) \leq M(r,t),\quad W_\lambda(r,t) \leq W(r,t)\quad\mathrm{for}\ (r,t) \in [0,\infty) \times [0,\min\{T_\lambda, T\}).
\end{align}
Especially in light of $\lambda > 1$, Proposition~\ref{sub and super} (i) and Proposition~\ref{CP_3} (ii) with $(u_\lambda, w_\lambda)$ give that $(M_\lambda,W_\lambda)$ is nondecreasing in $t \in (0,T_\lambda)$ for each $r \in [0,\infty)$, so that
\begin{align}\label{th3:ine2}
M_{\lambda, 0}(r) := M_\lambda (r,0) \leq M_\lambda (r,t),\quad W_{\lambda,0} := W_\lambda (r,0) \leq W_\lambda (r,t)
\end{align}
for all $(r,t) \in [0, \infty) \times [0,T_\lambda)$.
Consequently, the argument will be complete once we demonstrate that $(M_\lambda, W_\lambda)$ blows up in finite time.
Assume for the contradiction that $T_\lambda = \infty$. 
Here we remark that $(u_\lambda, w_\lambda)$ is radially nonincreasing by employing Lemma~\ref{CP_4}, since $(\lambda e^\phi, \lambda (-\Delta)\phi)$ is radially nonincreasing due to Proposition~\ref{stationaly sol}.
Hence owing to the monotonicity in $t > 0$ for $(M_\lambda, W_\lambda)$, Proposition~\ref{convergence} enables us to get the limit function $(M_{\lambda, \infty}, W_{\lambda, \infty}) \in [C^2((0,\infty)) \cap L^\infty(r_0,\infty)]^2$ for any $r_0 > 0$ satisfying \eqref{convergence:1} and moreover,
\begin{align}\label{th3:ine8_1}
r (M_{\lambda, \infty})_r + dM_{\lambda, \infty} \ge 0\quad \mathrm{for}\ r \in (0,\infty).
\end{align}
Gathering the monotonicity in $t > 0$ for $(M_\lambda, W_\lambda)$ and \eqref{th3:ine2}, we see that for all $(r,t) \in (0,\infty) \times (0,\infty)$
\begin{align}\label{th3:ine3}
M_{\lambda,0} (r) \leq M_\lambda (r,t) \leq M_{\lambda, \infty}(r),\quad W_{\lambda, 0} (r)\leq W_\lambda (r,t) \leq W_{\lambda, \infty}(r).
\end{align}
As observed in the proof of Proposition~\ref{stationaly sol} (see \eqref{stationary sol:1}, \eqref{stationary sol:2}), the asymptotic behavior for $(e^\phi,  (-\Delta)\phi)$ allows us to apply L'Hospital's rule and obtain 
\begin{align}
\lim_{r \to \infty} r^4 M_{\lambda, 0}(r) &= \lim_{r \to \infty} \dfrac{\lambda\int_0^r \rho^{d-1}e^{\phi(\rho)} d\rho }{r^{d-4}}\notag\\
&= \lim_{r \to \infty} \dfrac{\lambda}{d-4}r^4e^{\phi(r)}\notag\\
&= 8(d-2) \lambda > 8(d-2).\label{th3:ine4}
\end{align}
A similar argument applies to $W_{\lambda,0}$, leading to
\begin{align}\label{th3:ine5}
\lim_{r \to \infty} r^2 W_{\lambda, 0}(r) = 4\lambda > 4.
\end{align}
By combining \eqref{th3:ine4} and \eqref{th3:ine5} with \eqref{th3:ine3}, we deduce that
\begin{align}\label{th3:ine6}
\liminf_{r \to \infty} r^4 M_{\lambda, \infty}(r) \ge \lim_{ r\to \infty} r^4 M_{\lambda, 0}(r) > 8(d-2)
\end{align}
as well as
\begin{align}\label{th3:ine7}
\liminf_{r\to\infty}r^2 W_{\lambda, \infty}(r) \ge \lim_{r \to \infty} r^2 W_{\lambda, 0}(r) > 4.
\end{align}
The asymptotic behavior of $(M_{\lambda, \infty}, W_{\lambda, \infty})$ has been derived.
Our next goal is to examine the corresponding asymptotics of the stationary solution $(u_{\lambda,\infty}, v_{\lambda,\infty}, w_{\lambda,\infty})$ of \eqref{p}.
This can be accomplished by applying Proposition~\ref{u,v,w}, which links the mass functions to the original variables.
At this stage, exploiting \eqref{th3:ine8_1} and \eqref{th3:ine7} together with the definition of $(u_{\lambda,\infty}, v_{\lambda,\infty}, w_{\lambda,\infty})$ given in Proposition~\ref{u,v,w}
, we arrive at the estimate, valid for sufficiently large $r$
\begin{align*}
(r^4 u_{\lambda, \infty})_r &= r^3 (r(M_{\lambda, \infty})_r + dM_{\lambda, \infty}) (4 - r^2W_{\lambda, \infty}) \leq 0.
\end{align*}
In fact, by setting
\begin{align*}
\Phi (r) = r(M_{\lambda, \infty})_r + dM_{\lambda, \infty},
\end{align*}
the system \eqref{convergence:1}, which is satisfied by $(M_{\lambda, \infty}, W_{\lambda, \infty})$, yields
\begin{align}\label{th3:ine8_2}
\Phi_r (r) + rW_{\lambda, \infty}(r) \Phi (r) = 0\quad \mathrm{for}\ r >0.
\end{align}
Therefore in light of Proposition~\ref{u,v,w}, we obtain for $r > 0$
\begin{align*}
(r^4 u_{\lambda, \infty})_r &= (r^4 \Phi)_r = r^3(r\Phi_r + 4\Phi) = r^3 \Phi (4-r^2W_{\lambda, \infty}).
\end{align*}
As a consequence, since it holds that $(r^4 u_{\lambda, \infty})_r \leq 0$ for sufficiently large $r$, the quantity $r^4 u_{\lambda, \infty}$ is nonincreasing for sufficiently large $r$. Thus the limit $ \lim_{r \to \infty}r^4 u_{\lambda, \infty}(r)$ exists. Here the relationship \eqref{th3:ine3} as well as Proposition~\ref{u,v,w} implies that for $r > 0$
\begin{align*}
\int_0^r \rho^{d-1}u_{\lambda, \infty} (\rho) d\rho &= r^d M_{\lambda, \infty} \ge r^d M_{\lambda, 0} = \lambda \int_0^r \rho^{d-1} e^\phi d \rho.
\end{align*}
Since $e^\phi$ exhibits the asymptotic behavior $e^\phi = 8(d-4)(d-2)r^{-4} + o(1)$ as $r\to \infty$, the right-hand side diverges to $\infty$ as $r \to \infty$. Therefore we may apply L'Hospital's rule to obtain
\begin{align*}
\liminf_{r \to \infty}r^4 M_{\lambda, \infty}(r) &= \lim_{r \to \infty}\dfrac{\int_0^r \rho^{d-1}u_{\lambda, \infty}(\rho)d\rho}{r^{d-4}}\\
&= \lim_{r \to \infty} \dfrac{1}{d-4}r^4u_{\lambda, \infty}(r),
\end{align*}
where the last equality follows by the existence of the limit $ \lim_{r \to \infty}r^4 u_{\lambda, \infty}(r)$.
Furthermore, combining this identity with \eqref{th3:ine6}, we deduce that
\begin{align}\label{th3:ine8}
\lim_{r \to \infty} r^4 u_{\lambda, \infty} (r) > 8(d-4)(d-2).
\end{align}
The definition of $(u_{\lambda, \infty}, v_{\lambda, \infty}, w_{\lambda, \infty})$ established by Proposition~\ref{u,v,w}  and the equality \eqref{th3:ine8_2} yield that
\begin{align*}
(u_{\lambda, \infty})_r - u_{\lambda, \infty}(v_{\lambda, \infty})_r = (r(M_{\lambda, \infty})_r + dM_{\lambda, \infty})_r - (r(M_{\lambda, \infty})_r + dM_{\lambda, \infty})(-rW_{\lambda, \infty}) = 0.
\end{align*}
Here $u_{\lambda, \infty}$ cannot be identically zero for $r > 0$ because of the asymptotic profile \eqref{th3:ine8}. This fact provides that $u_{\lambda, \infty}$ is strictly positive for all $r > 0$ in the same way as in \cite[Lemma 3.1]{W2010}. Consequently, the identity, together with the positivity of $u_{\lambda, \infty}$, implies that 
$ \log u_{\lambda, \infty} = v_{\lambda, \infty} + C$ for all $r > 0$ and some constant $C$. Condition \eqref{th3:ine8} corresponds to the case in which the asymptotic behavior \eqref{limit:2} in Proposition~\ref{limit} fails to hold.
Therefore, by considering the contrapositive of Proposition~\ref{limit}, we conclude that
\begin{align}\label{th3:ine8_4}
\liminf_{r \to \infty} (v_{\lambda, \infty}(r) + 4\log r) = -\infty.
\end{align}
Nevertheless, by recalling that for $r > 0$
\begin{align*}
r^4 u_{\lambda, \infty}(r) = Cr^4 e^{v_{\lambda, \infty}(r)} = C e^{v_{\lambda, \infty}(r) + 4\log r}, 
\end{align*}
the asymptotic behavior \eqref{th3:ine8_4} for $v_{\lambda, \infty}$ enforces $\lim_{r \to \infty} r^4 u_{\lambda, \infty}(r) = 0$, which is incompatible with \eqref{th3:ine8}. Hence $T_\lambda $ must be finite. This necessarily forces the solution $(u,v,w)$ blows up at some time earlier than $T_\lambda$.

\medskip
(iii) Since $\lambda_1 > \lambda_2 > 1$, the proof of (ii) extends directly, by Proposition~\ref{sub and super}.
\end{proof}

\section{Proof of Theorem~\ref{th:3.5} and Theorem~\ref{th:4}}\label{chap:th_3}

In this section, motivated by the results obtained in \cite{AGG2006, BFFG2012}, we demonstrate that the qualitative behavior of solutions to the system \eqref{p} changes depending on the spatial dimension. Our approach is to focus first on the radial solution of the fourth-order elliptic problem \eqref{S2} and to show that the Morrey norms for $(e^\phi, (-\Delta)\phi)$ differ significantly with the dimension.

\begin{lemma}\label{d13}
Let $d \ge 13$ and $\phi \in C^4(\R^d)$ be a radially symmetric solution of \eqref{S2} with initial conditions
\begin{align*}
\phi(0) = \log \alpha,\quad\Delta \phi(0) = \beta_0(\alpha),\quad \phi^\prime(0) = (\Delta \phi)^\prime (0) = 0,
\end{align*}
where  $\alpha > 0$ and $\beta_0(\alpha) \in [-4d\alpha^\frac{1}{2}, 0)$. Then it holds that
\begin{align}\label{d13:1}
\sup_{R >0} R^{4-d}\int_0^R r^{d-1}e^{\phi (r)} dr = 8(d-2)
\end{align}
as well as
\begin{align}\label{d13:2}
\sup_{R > 0}R^{2-d}\int_0^R r^{d-1}(-\Delta)\phi(r) dr = 4.
\end{align}
\end{lemma}

\begin{proof}
According to  \cite[Lemma 12]{BFFG2012}, when $d \ge 13$, the regular radial solution $\phi$ of \eqref{S2} with the prescribed conditions at the origin and a singular radial solution $-4\log r + \log 8(d-4)(d-2)$ of \eqref{S2} do not intersect each other for $r > 0$, namely we have
\begin{align}\label{d13:4}
\phi(r) < \log \dfrac{8(d-4)(d-2)}{r^4}\quad \mathrm{for}\ r \in (0,\infty).
\end{align}
This behavior implies that for $R > 0$
\begin{align}
R^{4-d}\int_0^R r^{d-1} e^{\phi(r)} dr &< R^{4-d} \int_0^R r^{d-1} \dfrac{8(d-4)(d-2)}{r^4} dr= 8(d-2).\label{d13:3}
\end{align}
On the other hand, the asymptotic behavior of $\phi$ in light of \cite[Theorem 2]{AGG2006} and \cite[Theorem 4]{BFFG2012}:
\begin{align*}
\phi(r) = -4\log r + \log 8(d-4)(d-2) + o(1)\quad \mathrm{as}\  r\to \infty
\end{align*}
immediately implies that
$ \int_0^R r^{d-1}e^{\phi(r)} dr \to \infty$ as $R \to \infty$, allowing the use of L'Hospital's rule.
Consequently we get
\begin{align*}
\lim_{R \to \infty}R^{4-d}\int_0^R r^{d-1}e^{\phi (r)} dr &= \lim_{R \to \infty} \dfrac{R^{d-1}e^{\phi(R)}}{(d-4) R^{d-5}} = 8(d-2),
\end{align*}
which lead to our desired conclusion \eqref{d13:1} together with \eqref{d13:3}. We next confirm \eqref{d13:2}. Since the solution $\phi$ solves the fourth-order equation $\Delta^2 \phi = e^\phi$ on $ [0,\infty)$, the result in \cite[Theorem 1.1]{G2014} leads to the following estimate:
\begin{align*}
-\Delta \phi (r) <  \dfrac{4(d-2)}{r^2}\quad \mathrm{for\ all}\  r \in (0,\infty).
\end{align*}
Hence we obtain for any $R > 0$
\begin{align*}
R^{2-d}\int_0^R r^{d-1}(-\Delta)\phi (r) dr < 4(d-2) R^{2-d}\int_0^Rr^{d-3} dr = 4.
\end{align*}
Moreover as noted in the proof (see \eqref{stationary sol:2}) in Proposition~\ref{stationaly sol}, the detailed asymptotic expansion of $\phi$ is available thanks to the results in \cite{G2014}. This allows us to apply L'Hospital's rule to deduce
\begin{align*}
\lim_{R\to \infty} R^{2-d}\int_0^R r^{d-1}(-\Delta )\phi(r) dr = \lim_{R\to \infty}\dfrac{R^{d-1}(-\Delta)\phi(R)}{(d-2)R^{d-3}} = 4,
\end{align*}
which clearly proves \eqref{d13:2}.
\end{proof}

\begin{lemma}\label{d5~12}
Let $ 5 \leq d \leq 12$ and let $\phi \in C^4(\R^d)$ be a radially symmetric solution of \eqref{S2} with the conditions at the origin
\begin{align*}
\phi(0) = \log \alpha,\quad\Delta \phi(0) = \beta_0(\alpha),\quad \phi^\prime(0) = (\Delta \phi)^\prime (0) = 0,
\end{align*}
where  $\alpha > 0$ and $\beta_0(\alpha) \in [-4d\alpha^\frac{1}{2}, 0)$. Then there exist $K_0 > 8(d-2)$ and $K_1 > 4$ such that
\begin{align}\label{d5~12:claim1}
\sup_{R >0} R^{4-d}\int_0^R r^{d-1}e^{\phi (r)} dr = K_0
\end{align}
and 
\begin{align}\label{d5~12:claim2}
\sup_{R > 0}R^{2-d}\int_0^R r^{d-1}(-\Delta)\phi(r) dr = K_1.
\end{align}
\end{lemma}

\begin{proof}
As shown in \cite[Theorem 1.1]{G2014}, it turns out that when $5 \leq d \leq 12$, the regular radial solution $\phi$ of \eqref{S2} subject to the conditions at the origin oscillates around the singular profile
$-4\log r + \log 8(d-4)(d-2)$, and in particular, the two solutions intersect infinitely many times. Moreover thanks to \cite[Theorem 3.1]{G2014}
the regular radial solution $\phi$ satisfies the following oscillatory asymptotic expansion as $r \to \infty$
\begin{align}\label{d5~12:1}
\phi(r) = -4\log r + \log 8(d-4)(d-2) + C_0 r^\tau \sin(k_0 \log r + k_1) + C_1r^\nu + O(r^{2\tau}),
\end{align}
where $\tau = -\frac{d-4}{2}$ and $\nu < 2-d$, and for some constants $k_0, k_1, C_0,C_1$ with $C_0 \not= 0$.
Our goal is to show that $R^{4-d}\int_0^R r^{d-1}e^{\phi(r)} dr - 8(d-2)$ changes sign infinitely many times as $R \to \infty$. 
Using the asymptotic profile \eqref{d5~12:1}, we write
\begin{align*}
e^{\phi(r)} = \dfrac{8(d-4)(d-2)}{r^4}e^{\varepsilon(r)}\quad \mathrm{as}\ r \to \infty,
\end{align*}
where $\varepsilon (r) =  C_0 r^\tau \sin(k_0 \log r + k_1) + C_1r^\nu + O(r^{2\tau})$. Since $\varepsilon \to 0$ as $r\to \infty$, we may expand $e^{\varepsilon(r)}$ by Taylor's formula to obtain
\begin{align*}
e^{\varepsilon (r)} = 1 + \varepsilon(r) + O(\varepsilon (r)^2)\quad\mathrm{as}\ r \to \infty.
\end{align*}
Noticing that $\varepsilon (r) = O(r^\tau)$ due to \eqref{d5~12:1}, we deduce
\begin{align}\label{d5~12:2}
e^{\phi(r)} = \dfrac{8(d-4)(d-2)}{r^4}(1 + \varepsilon (r) + O(r^{2\tau}))\quad \mathrm{as}\ r \to \infty.
\end{align}
Let $R_0 > 0$ be sufficiently large so that \eqref{d5~12:1} holds for all $r \ge R_0$. Then we have
\begin{align}\label{d5~12:3}
R^{4-d}\int_0^R r^{d-1}e^{\phi(r)} dr = R^{4-d}\int_0^{R_0} r^{d-1}e^{\phi(r)} dr + R^{4-d}\int_{R_0}^R r^{d-1}e^{\phi(r)} dr.
\end{align}
Since $\phi$ is regular near the origin, the first term contributes only a bounded quantity.
More precisely, there exists a constant $K_0 > 0$ such that
\begin{align}\label{d5~12:9}
R^{4-d}\int_0^{R_0} r^{d-1}e^{\phi(r)} dr \leq K_0 R^{4-d}.
\end{align}
Therefore, our task is reduced to understanding the asymptotics of the second term in \eqref{d5~12:3}. Combing \eqref{d5~12:2} with the second term on \eqref{d5~12:3} yields that for $R \ge R_0$
\begin{align*}
R^{4-d}\int_{R_0}^R r^{d-1}e^{\phi(r)} dr &= 8(d-4)(d-2)R^{4-d}\int_{R_0}^R r^{d-5} dr\notag\\
&\hspace{1cm} + 8(d-4)(d-2)R^{4-d}\int_{R_0}^R r^{d-5} \varepsilon(r)dr\notag\\
&\hspace{1cm} + 8(d-4)(d-2)R^{4-d}\int_{R_0}^R r^{d-5} O(r^{2\tau}) dr\notag\\
&=: I + II + III.
\end{align*}
The term $I$ can be computed explicitly.
Indeed, we have
\begin{align}\label{d5~12:4}
I = 8(d-2) - 8(d-2)R_0^{d-4}R^{4-d},
\end{align}
which implies that $I = 8(d-2) + O(R^{4-d})$ as $R \to \infty$. With regard to the term $III$, since $d-5 + 2\tau = -1$, we can see that there exists $C > 0$ such that
\begin{align}\label{d5~12:8}
III \leq CR^{4-d}\log R + CR^{4-d}\log R_0.
\end{align}
This clearly implies that $III = O(R^\tau)$ as $R \to \infty$.
As for the term $II$, we expand it as
\begin{align*}
II&= 8(d-4)(d-2)C_0R^{4-d}\int_{R_0}^R r^{d-5 + \tau} \sin (k_0\log r+ k_1) dr\\
&\hspace{1cm} + 8(d-4)(d-2)C_1 R^{4-d}\int_{R_0}^R r^{d-5 + \nu} dr\\
&\hspace{1cm} + 8(d-4)(d-2)R^{4-d} \int_{R_0}^R r^{d-5} O(r^{2\tau}) dr.
\end{align*}
Among these contributions, only the first term requires careful attention.
Indeed, the remaining two parts are of order $O(R^{4-d}) + O(R^{\nu}) + O(R^{\tau})$, so we may safely disregard them when analyzing the asymptotic behavior. Hereafter, every constant appearing in the computations will be denoted by the same symbol $C$. We proceed to observe that
\begin{align}
&8(d-4)(d-2)C_0R^{4-d}\int_{R_0}^R r^{d-5 + \tau} \sin (k_0\log r+ k_1) dr\notag\\
 &\hspace{1cm}= CR^{4-d}\int_{R_0}^R r^{d-5 + \tau} \sin (k_0\log r+ k_1) dr\notag\\
 &\hspace{1cm} = CR^{4-d}\int_{\log R_0}^{\log R} e^{(d-5 + \tau)s} \sin (k_0s + k_1) e^s ds\notag\\
 &\hspace{1cm} = CR^{4-d}\int_{\log R_0}^{\log R} e^{(d-4 + \tau)s} \sin (k_0s + k_1)ds.\label{d5~12:5}
\end{align}
By a standard computation, one can check that
\begin{align*}
\int_{\log R_0}^{\log R} e^{(d-4 + \tau)s} \sin (k_0s + k_1)ds &= CR^{d-4 + \tau} \sin (k_0 \log R + k_1)\\
&\hspace{1cm} + CR^{d-4 + \tau} \cos (k_0 \log R + k_1) + C.
\end{align*}
Hence using the trigonometric addition formulas, we obtain
\begin{align*}
\int_{\log R_0}^{\log R} e^{(d-4 + \tau)s} \sin (k_0s + k_1)ds &= CR^{d-4 + \tau} \sin (k_0 \log R + \theta) + C
\end{align*}
for some phase shift $\theta$. Combining \eqref{d5~12:5} with the identity above yields
\begin{align}
CR^{4-d}\int_{R_0}^{R} r^{d-5+\tau}\sin(k_0\log r + k_1) dr
= CR^{\tau}\sin(k_0\log R + \theta) + CR^{4-d}.\label{d5~12:6}
\end{align}
Substituting \eqref{d5~12:9}, \eqref{d5~12:4}, \eqref{d5~12:8}, and \eqref{d5~12:6} into \eqref{d5~12:3}, we arrive at
\begin{align*}
R^{4-d}\int_0^R r^{d-1}e^{\phi(r)} dr - 8(d-2) = CR^{\tau}\sin(k_0\log R + \theta) + O(R^\tau)\quad\mathrm{as}\ R \to \infty
\end{align*}
for some $\theta$. Since $\tau < 0$, the above asymptotic expansion immediately implies that
\begin{align*}
\lim_{R \to \infty} R^{4-d}\int_0^R r^{d-1}e^{\phi(r)} dr = 8(d-2).
\end{align*}
Nevertheless, the oscillatory term $\sin (k_0\log R + \theta)$ changes sign infinitely many times as $R\to \infty$. Thus we obtain the desired conclusion \eqref{d5~12:claim1}. Similarly, differentiating the asymptotic expansion of 
$\phi$ shows that
\begin{align*}
-\Delta \phi (r) = 4(d-2)r^{-2} + \tilde{C_0}r^{\tau-2}\sin (k_0 \log r + \tilde{\theta}) + O(r^{\nu-2}) + O(r^{2\tau -2})\quad \mathrm{as}\ r \to \infty
\end{align*}
for some constant $\tilde{C_0} \not=0$ and some phase shift $\tilde{\theta}$. Arguing in exactly the same way as above, but now with $-\Delta\phi$ in place of $e^{\phi}$, we conclude that $R^{2-d}\int_0^R r^{d-1}(-\Delta)\phi (r) dr$ oscillates infinitely many times around its limit $4$. This completes the proof.
\end{proof}

Finally by applying the results of Theorem~\ref{th:3}, we show that the long-time behavior of the solution to \eqref{p} changes dramatically depending on the dimension.

\begin{proof}[Proof of Theorem~\ref{th:3.5}]
(i)  Since we have the assumption that the initial data $(u_0,w_0)$ satisfies
\begin{align*}
0 \leq u_0 \leq e^{\phi},\quad 0 \leq w_0 \leq (-\Delta)\phi\quad \mathrm{for\ all}\  x \in \R^d,
\end{align*}
according to Theorem~\ref{th:3} (i), the solution $(u,v,w)$ of \eqref{p} exists globally in time and enjoys the following bounds 
\begin{align*}
\int_{|x| < R} u(x,t) dx \leq  \int_{|x| < R} e^{\phi(x)} dx,\quad \int_{|x| < R} w(x,t) dx \leq  \int_{|x| < R} (-\Delta)\phi(x) dx
\end{align*}
for all $(R,t) \in (0,\infty) \times [0,\infty)$. From Lemma~\ref{d13}, valid for 
$d \ge 13$, we have
\begin{align*}
\sup_{R > 0}R^{4-d}\int_{|x| < R} e^{\phi(x)} dx = 8(d-2)\sigma_d,\quad \sup_{R > 0}R^{4-d}\int_{|x| < R} (-\Delta)\phi(x) dx = 4\sigma_d.
\end{align*}
Thus the upper bounds claimed in \eqref{th3.5:bound} follow.

\medskip
(ii) If $5 \leq d \leq 12$, then, in light of Lemma~\ref{d5~12}, we can choose a parameter $\lambda \in (0,1)$ so that $\lambda K_0 \leq 8(d-2)$ and $\lambda K_1 \leq 4$, where the constants $K_0$ and $K_1$ are those provided in Lemma~\ref{d5~12}. Therefore whenever the initial data $(u_0,w_0)$ satisfy 
\begin{align*}
0 \leq u_0 \leq \lambda e^{\phi},\quad 0 \leq w_0 \leq \lambda (-\Delta)\phi\quad\mathrm{for\ all} \ x \in \R^d,
\end{align*}
from Theorem~\ref{th:3}, we conclude that the solution exists globally in time and satisfy the estimates \eqref{th3.5:bound}.
\end{proof}

\begin{proof}[Proof of Theorem~\ref{th:4}]
(i) Let $\lambda > 1$. Assume that the initial data $(u_0,w_0)$
are given by $(u_0,w_0) = (\lambda e^\phi, \lambda(-\Delta)\phi)$, where $\phi \in C^4(\R^d)$ is a radially symmetric solution of \eqref{S2} with the initial conditions:
\begin{align*}
\phi(0) = \log \alpha,\quad\Delta \phi(0) = \beta_0(\alpha),\quad \phi^\prime(0) = (\Delta \phi)^\prime (0) = 0,
\end{align*}
where  $\alpha > 0$ and $\beta_0(\alpha) \in [-4d\alpha^\frac{1}{2}, 0)$. According to Theorem~\ref{th:3}, the corresponding solution $(u,v,w)$ obtained via Proposition~\ref{local_2} blows up in finite time. Moreover, Lemma~\ref{d13} yields that
\begin{align*}
\sup_{R > 0} R^{4-d}\int_0^R r^{d-1}u_0(r) dr &= \lambda \sup_{R > 0}R^{4-d}\int_0^R r^{d-1}e^{\phi(r)} dr\\
&= 8(d-2)\lambda\\
&> 8(d-2)
\end{align*}
and similarly,
\begin{align*}
\sup_{R > 0}R^{2-d}\int_0^Rr^{d-1}w_0(r) dr = 4\lambda> 4.
\end{align*}
This completes the proof.

\medskip
(ii) Let $\lambda \in (0,1]$ so that
\begin{align*}
\lambda K_0 > 8(d-2),\quad \lambda K_1 > 4,
\end{align*}
where $K_0$ and $K_1$ are the constants in Lemma~\ref{d5~12}.
In analogy with (i), the initial data are chosen to be $(u_0,w_0) = (\lambda e^\phi, \lambda (-\Delta)\phi)$, where the same function $\phi$ is employed. In light of Theorem~\ref{th:3}, the solution $(u,v,w)$ established by Proposition~\ref{local_2} exists globally in time, nevertheless Lemma~\ref{d5~12} implies that
the initial data $(u_0,w_0)$ satisfies the following condition
\begin{align*}
\sup_{R > 0} R^{4-d}\int_0^R r^{d-1}u_0(r) dr &= \lambda \sup_{R > 0}R^{4-d}\int_0^R r^{d-1}e^{\phi(r)} dr =\lambda K_0 > 8(d-2)
\end{align*}
and moreover
\begin{align*}
\sup_{R > 0}R^{2-d}\int_0^Rr^{d-1}w_0(r) dr = \lambda K_1> 4.
\end{align*}
Thus we arrive at the desired conclusion.
\end{proof}

\textbf{Acknowledgments}

The author wishes to thank my supervisor Kentaro Fujie in Tohoku University for several useful comments concerning this paper.

\bigskip
\address{ 
Mathematical Institute \\
Tohoku University \\
Sendai 980-8578 \\
Japan
}
{soga.yuri.q6@dc.tohoku.ac.jp}


\begin{bibdiv}
\begin{biblist}

\bib{A1995}{book}{
   author={Amann, Herbert},
   title={Linear and quasilinear parabolic problems. Vol. I},
   series={Monographs in Mathematics},
   volume={89},
   note={Abstract linear theory},
   publisher={Birkh\"auser Boston, Inc., Boston, MA},
   date={1995},
   pages={xxxvi+335},
}

\bib{AGG2006}{article}{
   author={Arioli, Gianni},
   author={Gazzola, Filippo},
   author={Grunau, Hans-Christoph},
   title={Entire solutions for a semilinear fourth order elliptic problem
   with exponential nonlinearity},
   journal={J. Differential Equations},
   volume={230},
   date={2006},
   number={2},
   pages={743--770},
}

\bib{BBTW2015}{article}{
   author={Bellomo, N.},
   author={Bellouquid, A.},
   author={Tao, Y.},
   author={Winkler, M.},
   title={Toward a mathematical theory of Keller-Segel models of pattern
   formation in biological tissues},
   journal={Math. Models Methods Appl. Sci.},
   volume={25},
   date={2015},
   number={9},
   pages={1663--1763},
}

\bib{BFFG2012}{article}{
   author={Berchio, Elvise},
   author={Farina, Alberto},
   author={Ferrero, Alberto},
   author={Gazzola, Filippo},
   title={Existence and stability of entire solutions to a semilinear fourth
   order elliptic problem},
   journal={J. Differential Equations},
   volume={252},
   date={2012},
   number={3},
   pages={2596--2616},
}


\bib{B1995}{article}{
   author={Biler, Piotr},
   title={The Cauchy problem and self-similar solutions for a nonlinear
   parabolic equation},
   journal={Studia Math.},
   volume={114},
   date={1995},
   number={2},
   pages={181--205},
}

\bib{B1998}{article}{
   author={Biler, Piotr},
   title={Local and global solvability of some parabolic systems modelling
   chemotaxis},
   journal={Adv. Math. Sci. Appl.},
   volume={8},
   date={1998},
   number={2},
   pages={715--743},
}

\bib{B1999}{article}{
   author={Biler, Piotr},
   title={Global solutions to some parabolic-elliptic systems of chemotaxis},
   journal={Adv. Math. Sci. Appl.},
   volume={9},
   date={1999},
   number={1},
   pages={347--359},
   issn={1343-4373},
   review={\MR{1690388}},
}

\bib{BHeN1994}{article}{
   author={Biler, Piotr},
   author={Hebisch, Waldemar},
   author={Nadzieja, Tadeusz},
   title={The Debye system: existence and large time behavior of solutions},
   journal={Nonlinear Anal.},
   volume={23},
   date={1994},
   number={9},
   pages={1189--1209},
}

\bib{BHN1994}{article}{
   author={Biler, Piotr},
   author={Hilhorst, Danielle},
   author={Nadzieja, Tadeusz},
   title={Existence and nonexistence of solutions for a model of
   gravitational interaction of particles. II},
   journal={Colloq. Math.},
   volume={67},
   date={1994},
   number={2},
   pages={297--308},
}

\bib{BKP2019}{article}{
   author={Biler, Piotr},
   author={Karch, Grzegorz},
   author={Pilarczyk, Dominika},
   title={Global radial solutions in classical Keller-Segel model of
   chemotaxis},
   journal={J. Differential Equations},
   volume={267},
   date={2019},
   number={11},
   pages={6352--6369},
}

\bib{BKW2023}{article}{
   author={Biler, Piotr},
   author={Karch, Grzegorz},
   author={Wakui, Hiroshi},
   title={Large self-similar solutions of the parabolic-elliptic
   Keller-Segel model},
   journal={Indiana Univ. Math. J.},
   volume={72},
   date={2023},
   number={3},
   pages={1027--1054},
}

\bib{BKZ2016}{article}{
   author={Biler, Piotr},
   author={Karch, Grzegorz},
   author={Zienkiewicz, Jacek},
   title={Morrey spaces norms and criteria for blowup in chemotaxis models},
   journal={Netw. Heterog. Media},
   volume={11},
   date={2016},
   number={2},
   pages={239--250},
}

\bib{BKZ2018}{article}{
   author={Biler, Piotr},
   author={Karch, Grzegorz},
   author={Zienkiewicz, Jacek},
   title={Large global-in-time solutions to a nonlocal model of chemotaxis},
   journal={Adv. Math.},
   volume={330},
   date={2018},
   pages={834--875},
}

\bib{BN1994}{article}{
   author={Biler, Piotr},
   author={Nadzieja, Tadeusz},
   title={Existence and nonexistence of solutions for a model of
   gravitational interaction of particles. I},
   journal={Colloq. Math.},
   volume={66},
   date={1994},
   number={2},
   pages={319--334},
}

\bib{CC2008}{article}{
   author={Calvez, Vincent},
   author={Corrias, Lucilla},
   title={The parabolic-parabolic Keller-Segel model in $\Bbb R^2$},
   journal={Commun. Math. Sci.},
   volume={6},
   date={2008},
   number={2},
   pages={417--447},
}

\bib{Chan1942}{book}{
   author={Chandrasekhar, S.},
   title={Principles of Stellar Dynamics},
   publisher={University of Chicago Press, Chicago, IL},
   date={1942},
   pages={x + 251},
}

\bib{CD2000}{book}{
   author={Cholewa, Jan W.},
   author={Dlotko, Tomasz},
   title={Global attractors in abstract parabolic problems},
   series={London Mathematical Society Lecture Note Series},
   volume={278},
   publisher={Cambridge University Press, Cambridge},
   date={2000},
   pages={xii+235},
}

\bib{FS2017}{article}{
   author={Fujie, Kentarou},
   author={Senba, Takasi},
   title={Application of an Adams type inequality to a two-chemical
   substances chemotaxis system},
   journal={J. Differential Equations},
   volume={263},
   date={2017},
   number={1},
   pages={88--148},
}

\bib{FS2019}{article}{
   author={Fujie, Kentarou},
   author={Senba, Takasi},
   title={Blowup of solutions to a two-chemical substances chemotaxis system
   in the critical dimension},
   journal={J. Differential Equations},
   volume={266},
   date={2019},
   number={2-3},
   pages={942--976},
}

\bib{GM1989}{article}{
   author={Giga, Yoshikazu},
   author={Miyakawa, Tetsuro},
   title={Navier-Stokes flow in $\bold R^3$ with measures as initial
   vorticity and Morrey spaces},
   journal={Comm. Partial Differential Equations},
   volume={14},
   date={1989},
   number={5},
   pages={577--618},
}

\bib{GT1983}{book}{
   author={Gilbarg, David},
   author={Trudinger, Neil S.},
   title={Elliptic partial differential equations of second order},
   series={Grundlehren der mathematischen Wissenschaften [Fundamental
   Principles of Mathematical Sciences]},
   volume={224},
   edition={2},
   publisher={Springer-Verlag, Berlin},
   date={1983},
   pages={xiii+513},
}

\bib{G2014}{article}{
   author={Guo, Zongming},
   title={Further study of entire radial solutions of a biharmonic equation
   with exponential nonlinearity},
   journal={Ann. Mat. Pura Appl. (4)},
   volume={193},
   date={2014},
   number={1},
   pages={187--201},
}

\bib{HV1996}{article}{
   author={Herrero, Miguel A.},
   author={Vel\'azquez, Juan J. L.},
   title={Chemotactic collapse for the Keller-Segel model},
   journal={J. Math. Biol.},
   volume={35},
   date={1996},
   number={2},
   pages={177--194},
}

\bib{HV1997}{article}{
   author={Herrero, Miguel A.},
   author={Vel\'azquez, Juan J. L.},
   title={A blow-up mechanism for a chemotaxis model},
   journal={Ann. Scuola Norm. Sup. Pisa Cl. Sci. (4)},
   volume={24},
   date={1997},
   number={4},
   pages={633--683 (1998)},
}

\bib{HL2025}{article}{
   author={Hosono, Tatsuya},
   author={Lauren\c cot, Philippe},
   title={Global existence and boundedness of solutions to a fully parabolic
   chemotaxis system with indirect signal production in $\Bbb R^4$},
   journal={J. Differential Equations},
   volume={416},
   date={2025},
   pages={2085--2133},
}

\bib{J2015}{article}{
   author={Jin, Hai-Yang},
   title={Boundedness of the attraction-repulsion Keller-Segel system},
   journal={J. Math. Anal. Appl.},
   volume={422},
   date={2015},
   number={2},
   pages={1463--1478},
}

\bib{K1963}{article}{
   author={Kaplan, Stanley},
   title={On the growth of solutions of quasi-linear parabolic equations},
   journal={Comm. Pure Appl. Math.},
   volume={16},
   date={1963},
   pages={305--330},
}

\bib{K1992}{article}{
   author={Kato, Tosio},
   title={Strong solutions of the Navier-Stokes equation in Morrey spaces},
   journal={Bol. Soc. Brasil. Mat. (N.S.)},
   volume={22},
   date={1992},
   number={2},
   pages={127--155},
}

\bib{KS1970}{article}{
   author={Keller, Evelyn F.},
   author={Segel, Lee A.},
   title={Initiation of slime mold aggregation viewed as an instability},
   journal={J. Theoret. Biol.},
   volume={26},
   date={1970},
   number={3},
   pages={399--415},
}


\bib{LSU1968}{book}{
   author={Lady\v zenskaja, O. A.},
   author={Solonnikov, V. A.},
   author={Ural\cprime ceva, N. N.},
   title={Linear and quasilinear equations of parabolic type},
   language={Russian},
   series={Translations of Mathematical Monographs},
   volume={Vol. 23},
   note={Translated from the Russian by S. Smith},
   publisher={American Mathematical Society, Providence, RI},
   date={1968},
}

\bib{Lau2019}{article}{
   author={Lauren\c cot, Philippe},
   title={Global bounded and unbounded solutions to a chemotaxis system with
   indirect signal production},
   journal={Discrete Contin. Dyn. Syst. Ser. B},
   volume={24},
   date={2019},
   number={12},
   pages={6419--6444},
}

\bib{Luca2003}{article}{
author = {Luca, Magdalena},
author = {Chavez-Ross, Alexandra},
author = {Edelstein-Keshet, Leah},
author = {Mogilner, Alex},
title = {Chemotactic Signaling, Microglia, and Alzheimer's Disease Senile Plaques: Is There a Connection?},
journal = {Bulletin of Mathematical Biology},
volume = {65},
number = {4},
pages = {693-730},
date = {2003},
}

\bib{L1995}{book}{
   author={Lunardi, Alessandra},
   title={Analytic semigroups and optimal regularity in parabolic problems},
   series={Modern Birkh\"auser Classics},
   note={[2013 reprint of the 1995 original] [MR1329547]},
   publisher={Birkh\"auser/Springer Basel AG, Basel},
   date={1995},
   pages={xviii+424},
}

\bib{M2013}{article}{
   author={Mizoguchi, Noriko},
   title={Global existence for the Cauchy problem of the parabolic-parabolic
   Keller-Segel system on the plane},
   journal={Calc. Var. Partial Differential Equations},
   volume={48},
   date={2013},
   number={3-4},
   pages={491--505},
}

\bib{M2020}{article}{
   author={Mizoguchi, Noriko},
   title={Finite-time blowup in Cauchy problem of parabolic-parabolic
   chemotaxis system},
   language={English, with English and French summaries},
   journal={J. Math. Pures Appl. (9)},
   volume={136},
   date={2020},
   pages={203--238},
}

\bib{N1995}{article}{
   author={Nagai, Toshitaka},
   title={Blow-up of radially symmetric solutions to a chemotaxis system},
   journal={Adv. Math. Sci. Appl.},
   volume={5},
   date={1995},
   number={2},
   pages={581--601},
   issn={1343-4373},
   review={\MR{1361006}},
}


\bib{NSS2000}{article}{
   author={Nagai, Toshitaka},
   author={Senba, Takasi},
   author={Suzuki, Takashi},
   title={Chemotactic collapse in a parabolic system of mathematical
   biology},
   journal={Hiroshima Math. J.},
   volume={30},
   date={2000},
   number={3},
   pages={463--497},
   issn={0018-2079},
   review={\MR{1799300}},
}

\bib{NSY1997}{article}{
   author={Nagai, Toshitaka},
   author={Senba, Takasi},
   author={Yoshida, Kiyoshi},
   title={Application of the Trudinger-Moser inequality to a parabolic
   system of chemotaxis},
   journal={Funkcial. Ekvac.},
   volume={40},
   date={1997},
   number={3},
   pages={411--433},
   issn={0532-8721},
   review={\MR{1610709}},
}

\bib{N2021}{article}{
   author={Naito, Y\=uki},
   title={Blow-up criteria for the classical Keller-Segel model of
   chemotaxis in higher dimensions},
   journal={J. Differential Equations},
   volume={297},
   date={2021},
   pages={144--174},
}

\bib{S2025}{article}{
   author={Soga, Yuri},
   title={Concentration phenomena to a chemotaxis system with indirect
   signal production},
   journal={J. Evol. Equ.},
   volume={25},
   date={2025},
   number={4},
   pages={Paper No. 95, 45},
}

\bib{STP2013}{article}{
   author={Strohm, S.},
   author={Tyson, R. C.},
   author={Powell, J. A.},
   title={Pattern formation in a model for mountain pine beetle dispersal:
   linking model predictions to data},
   journal={Bull. Math. Biol.},
   volume={75},
   date={2013},
   number={10},
   pages={1778--1797},
}

\bib{TW2017}{article}{
   author={Tao, Youshan},
   author={Winkler, Michael},
   title={Critical mass for infinite-time aggregation in a chemotaxis model
   with indirect signal production},
   journal={J. Eur. Math. Soc. (JEMS)},
   volume={19},
   date={2017},
   number={12},
   pages={3641--3678},
}

\bib{T1992}{article}{
   author={Taylor, Michael E.},
   title={Analysis on Morrey spaces and applications to Navier-Stokes and
   other evolution equations},
   journal={Comm. Partial Differential Equations},
   volume={17},
   date={1992},
   number={9-10},
   pages={1407--1456},
}

\bib{W2010}{article}{
   author={Winkler, Michael},
   title={Aggregation vs. global diffusive behavior in the
   higher-dimensional Keller-Segel model},
   journal={J. Differential Equations},
   volume={248},
   date={2010},
   number={12},
   pages={2889--2905},
}

\bib{W2023_1}{article}{
   author={Winkler, Michael},
   title={Classical solutions to Cauchy problems for parabolic-elliptic
   systems of Keller-Segel type},
   journal={Open Math.},
   volume={21},
   date={2023},
   number={1},
   pages={Paper No. 20220578, 19},
}

\bib{W2023_2}{article}{
   author={Winkler, Michael},
   title={Solutions to the Keller-Segel system with non-integrable behavior
   at spatial infinity},
   journal={J. Elliptic Parabol. Equ.},
   volume={9},
   date={2023},
   number={2},
   pages={919--959},
}

\end{biblist}
\end{bibdiv}
\end{document}